\documentclass[11pt]{article}
\usepackage{amsmath}
\usepackage{amsfonts}
\usepackage[mathscr]{eucal}
\usepackage{amsthm}
\usepackage{amssymb}
\usepackage[noadjust]{cite}
\usepackage{enumerate}
\usepackage{tikz}\usetikzlibrary{matrix, calc, arrows}
\usepackage{hyperref}
\definecolor{myblue}{rgb}{0,0,0.6}
\hypersetup{pdftitle={The boundary element method on fractals: acoustic scattering by screens},
     bookmarksopen=true, bookmarksopenlevel=4,
     colorlinks=true, linkcolor=myblue,  citecolor=myblue, filecolor=myblue,   urlcolor=myblue,  }

\usepackage{dsfont}
\usepackage{color}
\usepackage{soul}

\graphicspath{{./}{figs/}}

\DeclareMathOperator{\supp}{supp}

\DeclareMathOperator{\diam}{diam}
\DeclareMathOperator{\dist}{dist}

\newcommand{\Hull}{\mathrm{Hull}}

\usepackage{bbm}
\usepackage{color}
\definecolor{amcol}{rgb}{0.8,0,0}
\definecolor{dhcol}{rgb}{0,0.5,0}

\begin{document}

\newcommand{\rf}[1]{(\ref{#1})}
\newcommand{\mmbox}[1]{\fbox{\ensuremath{\displaystyle{ #1 }}}}	

\newcommand{\hs}[1]{\hspace{#1mm}}
\newcommand{\vs}[1]{\vspace{#1mm}}

\newcommand{\ri}{{\mathrm{i}}}
\newcommand{\re}{{\mathrm{e}}}
\newcommand{\rd}{\mathrm{d}}

\newcommand{\R}{\mathbb{R}}
\newcommand{\Q}{\mathbb{Q}}
\newcommand{\N}{\mathbb{N}}
\newcommand{\Z}{\mathbb{Z}}
\newcommand{\C}{\mathbb{C}}
\newcommand{\K}{{\mathbb{K}}}

\newcommand{\cA}{\mathcal{A}}
\newcommand{\cB}{\mathcal{B}}
\newcommand{\cC}{\mathcal{C}}
\newcommand{\cS}{\mathcal{S}}
\newcommand{\cD}{\mathcal{D}}
\newcommand{\cH}{\mathcal{H}}
\newcommand{\cI}{\mathcal{I}}
\newcommand{\cItilde}{\tilde{\mathcal{I}}}
\newcommand{\cIhat}{\hat{\mathcal{I}}}
\newcommand{\cIcheck}{\check{\mathcal{I}}}
\newcommand{\cIstar}{{\mathcal{I}^*}}
\newcommand{\cJ}{\mathcal{J}}
\newcommand{\cM}{\mathcal{M}}
\newcommand{\cP}{\mathcal{P}}
\newcommand{\cV}{{\mathcal V}}
\newcommand{\cW}{{\mathcal W}}
\newcommand{\scrD}{\mathscr{D}}
\newcommand{\scrS}{\mathscr{S}}
\newcommand{\scrJ}{\mathscr{J}}
\newcommand{\sD}{\mathsf{D}}
\newcommand{\sN}{\mathsf{N}}
\newcommand{\sS}{\mathsf{S}}
 \newcommand{\sT}{\mathsf{T}}
 \newcommand{\sH}{\mathsf{H}}
 \newcommand{\sI}{\mathsf{I}}
 
\newcommand{\bs}[1]{\mathbf{#1}}
\newcommand{\bb}{\mathbf{b}}
\newcommand{\bd}{\mathbf{d}}
\newcommand{\bn}{\mathbf{n}}
\newcommand{\bp}{\mathbf{p}}
\newcommand{\bP}{\mathbf{P}}
\newcommand{\bv}{\mathbf{v}}
\newcommand{\bx}{\mathbf{x}}
\newcommand{\by}{\mathbf{y}}
\newcommand{\bz}{{\mathbf{z}}}
\newcommand{\bxi}{\boldsymbol{\xi}}
\newcommand{\boldeta}{\boldsymbol{\eta}}	

\newcommand{\ts}{\tilde{s}}
\newcommand{\tGamma}{{\tilde{\Gamma}}}
 \newcommand{\tbx}{\tilde{\bx}}
 \newcommand{\tbd}{\tilde{\bd}}
 \newcommand{\txi}{\xi}
 
\newcommand{\done}[2]{\dfrac{d {#1}}{d {#2}}}
\newcommand{\donet}[2]{\frac{d {#1}}{d {#2}}}
\newcommand{\pdone}[2]{\dfrac{\partial {#1}}{\partial {#2}}}
\newcommand{\pdonet}[2]{\frac{\partial {#1}}{\partial {#2}}}
\newcommand{\pdonetext}[2]{\partial {#1}/\partial {#2}}
\newcommand{\pdtwo}[2]{\dfrac{\partial^2 {#1}}{\partial {#2}^2}}
\newcommand{\pdtwot}[2]{\frac{\partial^2 {#1}}{\partial {#2}^2}}
\newcommand{\pdtwomix}[3]{\dfrac{\partial^2 {#1}}{\partial {#2}\partial {#3}}}
\newcommand{\pdtwomixt}[3]{\frac{\partial^2 {#1}}{\partial {#2}\partial {#3}}}
\newcommand{\bnabla}{\boldsymbol{\nabla}}
\newcommand{\dive}{\boldsymbol{\nabla}\cdot}
\newcommand{\curl}{\boldsymbol{\nabla}\times}
\newcommand{\Phixy}{\Phi(\bx,\by)}
\newcommand{\PhiOxy}{\Phi_0(\bx,\by)}
\newcommand{\dxPhixy}{\pdone{\Phi}{n(\bx)}(\bx,\by)}
\newcommand{\dyPhixy}{\pdone{\Phi}{n(\by)}(\bx,\by)}
\newcommand{\dxPhiOxy}{\pdone{\Phi_0}{n(\bx)}(\bx,\by)}
\newcommand{\dyPhiOxy}{\pdone{\Phi_0}{n(\by)}(\bx,\by)}

\newcommand{\eps}{\varepsilon}
\newcommand{\real}[1]{{\rm Re}\left[#1\right]} 
\newcommand{\im}[1]{{\rm Im}\left[#1\right]}
\newcommand{\ol}[1]{\overline{#1}}
\newcommand{\ord}[1]{\mathcal{O}\left(#1\right)}
\newcommand{\oord}[1]{o\left(#1\right)}
\newcommand{\Ord}[1]{\Theta\left(#1\right)}

\newcommand{\hsnorm}[1]{||#1||_{H^{s}(\bs{R})}}
\newcommand{\hnorm}[1]{||#1||_{\tilde{H}^{-1/2}((0,1))}}
\newcommand{\norm}[2]{\left\|#1\right\|_{#2}}
\newcommand{\normt}[2]{\|#1\|_{#2}}
\newcommand{\on}[1]{\Vert{#1} \Vert_{1}}
\newcommand{\tn}[1]{\Vert{#1} \Vert_{2}}

\newcommand{\xt}{\mathbf{x},t}
\newcommand{\PhiF}{\Phi_{\rm freq}}
\newcommand{\cone}{{c_{j}^\pm}}
\newcommand{\ctwo}{{c_{2,j}^\pm}}
\newcommand{\cthree}{{c_{3,j}^\pm}}

\newtheorem{thm}{Theorem}[section]
\newtheorem{lem}[thm]{Lemma}
\newtheorem{defn}[thm]{Definition}
\newtheorem{prop}[thm]{Proposition}
\newtheorem{cor}[thm]{Corollary}
\newtheorem{rem}[thm]{Remark}
\newtheorem{conj}[thm]{Conjecture}
\newtheorem{ass}[thm]{Assumption}
\newtheorem{example}[thm]{Example} 

\newcommand{\tH}{\widetilde{H}}
\newcommand{\Hze}{H_{\rm ze}} 	
\newcommand{\uze}{u_{\rm ze}}		
\newcommand{\dimH}{{\rm dim_H}}
\newcommand{\dimB}{{\rm dim_B}}
\newcommand{\IntClosOm}{\mathrm{int}(\overline{\Omega})}
\newcommand{\IntClosOmOne}{\mathrm{int}(\overline{\Omega_1})}
\newcommand{\IntClosOmTwo}{\mathrm{int}(\overline{\Omega_2})}
\newcommand{\Ccomp}{C^{\rm comp}}
\newcommand{\tCcomp}{\tilde{C}^{\rm comp}}
\newcommand{\uC}{\underline{C}}
\newcommand{\utC}{\underline{\tilde{C}}}
\newcommand{\oC}{\overline{C}}
\newcommand{\otC}{\overline{\tilde{C}}}
\newcommand{\capcomp}{{\rm cap}^{\rm comp}}
\newcommand{\Capcomp}{{\rm Cap}^{\rm comp}}
\newcommand{\tcapcomp}{\widetilde{{\rm cap}}^{\rm comp}}
\newcommand{\tCapcomp}{\widetilde{{\rm Cap}}^{\rm comp}}
\newcommand{\hcapcomp}{\widehat{{\rm cap}}^{\rm comp}}
\newcommand{\hCapcomp}{\widehat{{\rm Cap}}^{\rm comp}}
\newcommand{\tcap}{\widetilde{{\rm cap}}}
\newcommand{\tCap}{\widetilde{{\rm Cap}}}
\newcommand{\ccap}{{\rm cap}}
\newcommand{\ucap}{\underline{\rm cap}}
\newcommand{\uCap}{\underline{\rm Cap}}
\newcommand{\cCap}{{\rm Cap}}
\newcommand{\ocap}{\overline{\rm cap}}
\newcommand{\oCap}{\overline{\rm Cap}}
\DeclareRobustCommand
{\mathringbig}[1]{\accentset{\smash{\raisebox{-0.1ex}{$\scriptstyle\circ$}}}{#1}\rule{0pt}{2.3ex}}
\newcommand{\cirH}{\mathringbig{H}}
\newcommand{\cirHs}{\mathringbig{H}{}^s}
\newcommand{\cirHt}{\mathringbig{H}{}^t}
\newcommand{\cirHm}{\mathringbig{H}{}^m}
\newcommand{\cirHzero}{\mathringbig{H}{}^0}
\newcommand{\deO}{{\partial\Omega}}
\newcommand{\OO}{{(\Omega)}}
\newcommand{\Rn}{{(\R^n)}}
\newcommand{\Id}{{\mathrm{Id}}}
\newcommand{\gap}{\mathrm{Gap}}
\newcommand{\ggap}{\mathrm{gap}}
\newcommand{\isom}{{\xrightarrow{\sim}}}
\newcommand{\half}{{1/2}}
\newcommand{\mhalf}{{-1/2}}
\newcommand{\inter}{{\mathrm{int}}}

\newcommand{\Hsp}{H^{s,p}}
\newcommand{\Htq}{H^{t,q}}
\newcommand{\tHsp}{{{\widetilde H}^{s,p}}}
\newcommand{\SP}{\ensuremath{(s,p)}}
\newcommand{\Xsp}{X^{s,p}}

\newcommand{\dd}{{d}}\newcommand{\pp}{{p_*}}

\newcommand{\Rnn}{\R^{n_1+n_2}}
\newcommand{\Tr}{{\mathrm{Tr}}}

\renewcommand{\arraystretch}{1.7}
\renewcommand{\bs}[1]{\boldsymbol{#1}}
\newcommand{\be}{\bs{e}}
\renewcommand{\bn}{\bs{n}}
\renewcommand{\bx}{x}%
\renewcommand{\by}{y}%
\newcommand{\bbx}{\Psi}
\newcommand{\bby}{\widetilde{\Psi}}
\newcommand{\bg}{\bs{g}}
\newcommand{\bu}{\bs{u}}
\newcommand{\bw}{\bs{w}}
\newcommand{\bA}{\bs{A}}
\newcommand{\bC}{\bs{C}}
\newcommand{\bL}{\bs{L}}
\newcommand{\bS}{\bs{S}}
\newcommand{\bT}{\bs{T}}
\newcommand{\bU}{\bs{U}}
\newcommand{\bV}{\bs{V}}
\newcommand{\bX}{\bs{X}}
\newcommand{\bgamma}{{\bs{\gamma}}}
\newcommand{\bH}{\bs{H}}
\newcommand{\bnu}{\boldsymbol{\nu}}
\newcommand{\btau}{\boldsymbol{\tau}}
\newcommand{\bseta}{\boldsymbol{\eta}}
\newcommand{\rD}{\mathrm{D}}
\newcommand{\rN}{\mathrm{N}}
\newcommand{\bm}{{\bs{m}}}
\newcommand{\bl}{\bs{l}}
\newcommand{\No}{{\mathbb{N}_0}}
\newcommand{\tbX}{\tilde{\bX}}
\newcommand{\tA}{\tilde{A}}
\renewcommand{\tH}{\widetilde{H}{}}
\newcommand{\tbH}{\widetilde{\bH}{}}
\newcommand{\sM}{\mathsf{M}}
\newcommand{\sE}{\mathsf{E}}
\newcommand{\cT}{\mathcal{T}}%
\newcommand{\cU}{\mathcal{U}}%
\newcommand{\cF}{\mathcal{F}}
\newcommand{\cL}{\mathcal{L}}
\newcommand{\cK}{\mathcal{K}}
\newcommand{\cN}{\mathcal{N}}
\newcommand{\cE}{\mathcal{E}}
\newcommand{\cR}{\mathcal{R}}
\newcommand{\tcA}{\tilde{\mathcal{A}}}
\newcommand{\tcL}{\tilde{\mathcal{L}}}
\newcommand{\tcK}{\tilde{\mathcal{K}}}
\newcommand{\bcD}{\boldsymbol{\mathcal{D}}}%
\newcommand{\vbcU}{\vec{\boldsymbol{\cU}}}
\newcommand{\vbcE}{\vec{\boldsymbol{\cE}}}
\newcommand{\vbcS}{\vec{\boldsymbol{\cS}}}
\newcommand{\bscrS}{\boldsymbol{\scrS}}
\newcommand{\dudnjump}{\left[ \pdone{u}{n}\right]}
\newcommand{\dudnjumptext}{[ \pdonetext{u}{n}]}
\newcommand{\sumpm}[1]{\overbracket[0.5pt]{\underbracket[0.5pt]{\,#1\,}}}
\renewcommand{\dudnjump}{\left[ \pdone{u^s}{n}\right]}
\renewcommand{\dudnjumptext}{[ \pdonetext{u^s}{n}]}
\newcommand{\gradd}{\vec{\rm \bf grad}}	%
\newcommand{\curld}{\vec{\rm \bf curl}}	%
\newcommand{\dived}{{\rm div}}			%
\newcommand{\gradp}{{\rm \bf grad}}	%
\newcommand{\curlp}{{\rm \bf curl}}	%
\newcommand{\scurlp}{{\rm curl}}		%
\newcommand{\divep}{{\rm div}}			%
\newcommand{\gradg}{{\rm \bf grad}_\Gamma}	%
\newcommand{\curlg}{{\rm \bf curl}_\Gamma}	%
\newcommand{\scurlg}{{\rm curl}_\Gamma}	%
\newcommand{\diveg}{{\rm div}_\Gamma}		%
\newcommand{\Deltag}{{\Delta_\Gamma}}		%
\newcommand{\kap}{m}
\newcommand{\buk}{\bu_\kap}
\newcommand{\hbuk}{\hat{\bu}_\kap}
\newcommand{\bvk}{\bv_\kap}
\newcommand{\hbvk}{\hat{\bv}_\kap}
\newcommand{\bwk}{\bw_\kap}
\newcommand{\hbwk}{\hat{\bw}_\kap}
\newcommand{\hbu}{\hat{\bu}}
\newcommand{\hbv}{\hat{\bv}}
\newcommand{\hbw}{\hat{\bw}}
\newcommand{\hphi}{\hat{\phi}}
\newcommand{\hpsi}{\hat{\psi}}
\newcommand{\deG}{{\partial\Gamma}}
\newcommand{\GG}{(\Gamma)}
\newcommand{\tr}{\mathrm{tr}_\Gamma}
\newcommand{\trs}{\mathrm{tr}_{\Gamma,s}}
\newcommand{\trhalf}{\mathrm{tr}_{\Gamma,\frac{1}{2}}}
\newcommand{\IH}{\mathbb{H}}
\newcommand{\IS}{\mathbb{S}}
\newcommand{\IL}{\mathbb{L}}
\newcommand{\II}{\mathbb{I}}
\newcommand{\IV}{\mathbb{V}}
\newcommand{\IW}{\mathbb{W}}
\newcommand{\IC}{\mathbb{C}}
\newcommand{\IX}{\mathbb{X}}
\newcommand{\IP}{\mathbb{P}}
\newcommand{\IQ}{\mathbb{Q}}
\newcommand{\IY}{\mathbb{Y}}
\newcommand{\tf}{\tilde{f}}
\newcommand{\dn}{\partial_{\mathrm{n}}}
\newcommand{\ih}{\mathfrak{h}}
\newcommand{\Nn}{\mathbb{N}^{n}}%
\newcommand{\Non}{\mathbb{N}_{0}^{n}}
\newcommand{\Zn}{\mathbb{Z}^{n}}%
\newcommand{\Cn}{\mathbb{C}^{n}}%
\newcommand{\dual}[2]{\left\langle #1\,,\,#2\right\rangle} %
\newcommand{\wt}[1]{\widetilde{#1}}
\newcommand{\divepR}{{\rm div}_{\R^2}}
\newcommand{\scurlpR}{{\rm curl}_{\R^2}}

\definecolor{purple0}{rgb}{0.4,0,0.5}
\definecolor{orange}{rgb}{1,0.4,0}
\newcommand{\cu}[1]{{\color{purple0} #1 }}
\definecolor{orange0}{rgb}{1,0.3,0}
\newcommand{\ctodo}[1]{{\color{orange0} {\bf TODO:} #1 }}
\definecolor{green0}{rgb}{0.1,0.6,0}
\newcommand{\dnote}[1]{{\color{green0} {\bf DH:} #1 }}
\newcommand{\snote}[1]{{\color{red} {\bf SC:} #1 }}
\newcommand{\acnote}[1]{{\color{orange} {\bf AC:} #1 }}
\newcommand{\hBEM}{h_{\mathrm{BEM}}}
\newcommand{\hquad}{h_{\mathrm{quad}}}
\newcommand{\vb}{{\vec b}}
\newcommand{\vc}{{\vec c}}
\newcommand{\va}{{\vec a}}
\newcommand{\vv}{{\vec v}}
\newcommand{\vphi}{{\vec\phi}}
\newcommand{\vpsi}{{\vec\psi}}
\newcommand{\vxi}{{\vec\xi}}
\newcommand{\hmeshref}{h_{\mathrm{ref}}}
\newcommand{\ellref}{\ell_\mathrm{ref}}
\newcommand{\Nref}{N_\mathrm{ref}}

\allowdisplaybreaks[4]

\newtheorem{claim}[thm]{Claim}
\newtheorem{prob}[thm]{Problem}

\title{A Hausdorff-measure boundary element method for acoustic scattering by fractal screens}\author{}
\author{A. M. Caetano$^{\text{a}}$,
S. N. Chandler-Wilde$^{\text{b}}$,
A. Gibbs$^{\text{c}}$,
D. P. Hewett$^{\text{c}}$ and A. Moiola$^{\text{d}}$\\
$^{\text{a}}${\footnotesize
Center for R\&D in Mathematics and Applications, 
Departamento de Matem\'atica, Universidade de Aveiro, Aveiro, Portugal}\\
$^{\text{b}}${\footnotesize Department of Mathematics and Statistics, University of Reading, Reading, United Kingdom}\\
$^{\text{c}}${\footnotesize Department of Mathematics, University College London, London, United Kingdom}\\
$^{\text{d}}${\footnotesize Dipartimento di Matematica ``F. Casorati'', Universit\`a degli studi di Pavia, Pavia, Italy}
}
\maketitle
\renewcommand{\thefootnote}{\arabic{footnote}}

\begin{abstract}
Sound-soft fractal screens can scatter acoustic waves even when they have zero surface measure. To solve such scattering problems we make what appears to be the first application of the boundary element method (BEM) where each BEM basis function is supported in a fractal set, and the integration involved in the formation of the BEM matrix is with respect to a non-integer order Hausdorff measure rather than the usual (Lebesgue) surface measure. Using recent results on function spaces on fractals, we prove convergence of the Galerkin formulation of this ``Hausdorff BEM'' for acoustic scattering in $\R^{n+1}$ ($n=1,2$) when the scatterer, assumed to be a compact subset of $\R^n\times\{0\}$, is a $d$-set for
some $d\in (n-1,n]$, so that, in particular, the scatterer has Hausdorff dimension $d$.
For a class of fractals that are attractors of iterated function systems, we prove convergence rates for the Hausdorff BEM and superconvergence for smooth antilinear functionals, under certain natural regularity assumptions on the solution of the underlying boundary integral equation.
We also propose numerical quadrature routines for the implementation of our Hausdorff BEM, along with a fully discrete convergence analysis, via numerical (Hausdorff measure) integration estimates and inverse estimates on fractals, estimating the discrete condition numbers.
Finally, we show numerical experiments that support the sharpness of  our theoretical results, 
and our solution regularity assumptions, 
including results  for scattering in $\R^2$ by  Cantor sets, and in $\R^3$ by  Cantor dusts.

\end{abstract}

\section{Introduction}
\label{sec:Introduction}
A classical problem in the study of acoustic, electromagnetic and elastic wave propagation is the scattering of a time-harmonic incident wave by an infinitesimally thin screen (or ``crack''). In the simplest configuration the incident wave propagates in $\R^{n+1}$ (typically $n=1,2$) and the screen $\Gamma$ is assumed to be a bounded subset of the hyperplane $\Gamma_\infty=\R^{n}\times\{0\}$. %
In standard analyses
the set $\Gamma$ is assumed (either explicitly or implicitly) to be a relatively open subset of $\Gamma_\infty$ with smooth relative boundary $\partial\Gamma$. But in a recent series of papers \cite{ChaHewMoi:13,ScreenPaper,BEMfract,ImpedanceScreen} it has been shown how well-posed boundary value problems (BVPs) and associated boundary integral equations (BIEs) for the acoustic version of this screen problem (with either Dirichlet, Neumann or impedance boundary conditions) can be formulated, analysed and discretized for arbitrary screens with no regularity assumption on $\Gamma$. In particular, this encompasses situations where either $\partial\Gamma$ or $\Gamma$ itself has a fractal nature.
The study of wave scattering by such fractal structures is not only interesting from a mathematical point of view, but is also relevant for numerous applications including the scattering of electromagnetic waves by complex ice crystal aggregates in weather and climate science \cite{So:01}
and the modelling of fractal antennas
in electrical engineering \cite{WeGa:03}.  
In applications the physical object generally only exhibits a certain number of levels of fractal structure; nonetheless, %
fractals provide an idealised mathematical model for objects that have self-similar structure at multiple lengthscales.

\begin{figure}
\centering\includegraphics[height=20mm]{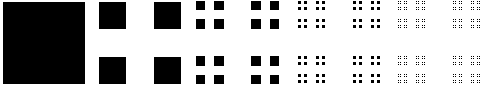}
\caption{The first five standard prefractal approximations, $\Gamma_0,\ldots,\Gamma_4$, to the middle-third Cantor dust $\Gamma$, defined by $\Gamma_0:=[0,1]^2$, $\Gamma_n:=s(\Gamma_{n-1})$, $n\in \N$, where $s$ is defined by \eqref{eq:fixedfirst} and  \eqref{eq:CD_IFS} with $M=4$ and $\rho=1/3$.}
\label{Fig:Dust}
\end{figure}

Our focus in this paper is on the Dirichlet (sound soft) acoustic scattering problem in the case where $\Gamma$ itself is fractal.\footnote{We note that our methods and results apply, with obvious modifications, to the analogous (yet simpler) problem in potential theory, in which the Helmholtz equation is replaced by the Laplace equation.}
We shall assume throughout that, for some $n-1<d\leq n$, $\Gamma$ is a compact $d$-set (i.e., $\Gamma$ is compact as a subset of $\Gamma_\infty$ and is a $d$-set as defined in \S\ref{sec:HausdorffMeasure})
which in particular implies that $\Gamma$ has  Hausdorff dimension equal to $d$.  More specifically, our attention will be on the special case where $\Gamma$ is the self-similar attractor of an iterated function system of contracting similarities, in particular on the case where $\Gamma$ satisfies a certain disjointness condition (described in \S\ref{sec:IFS}), in which case $\Gamma$ has (as a subset of $\R^n$) empty interior and zero Lebesgue measure. An example in the case $n=1$ is the middle-third Cantor set, which is a $d$-set for $d=\log{2}/\log{3}$; an example in the case $n=2$ is the middle-third Cantor dust shown in Figure \ref{Fig:Dust}, which is a $d$-set for $d=\log{4}/\log{3}$.
For such $\Gamma$, well-posed BVP and BIE formulations for the Dirichlet scattering problem were analysed in \cite{ScreenPaper}, where it was shown that the exact solution of the BIE
lies in the function space $H^{-1/2}_{\Gamma}=\{u\in H^{-1/2}(\Gamma_\infty):\supp{u}\subset \Gamma\}$ \cite[\S3.3]{ScreenPaper}. The assumption that $d>n-1$ implies that this space is non-trivial, and that for non-zero incident data the BIE solution is non-zero, so the screen produces a non-zero scattered field.
Our aim in this paper is to develop and analyse a boundary element method (BEM) that can efficiently compute this BIE solution. %

One obvious approach, adopted in \cite{BEMfract} (and see also \cite{jones1994fast,panagouli1997fem,ImpedanceScreen}), is to apply a conventional BEM on a sequence of smoother (e.g.\ Lipschitz) ``prefractal'' approximations to the underlying fractal screen, such as those illustrated in Figure \ref{Fig:Dust} for the middle-third Cantor dust\footnote{For recent overviews of the conventional BEM literature for Lipschitz or smoother screens see \cite{BEMfract} or \cite{AlAv:21,Cletal:21,JePi:22}.}. (In this example each prefractal is a union of squares.)
When $\Gamma$ has empty interior (as in the current paper) this is necessarily a ``non-conforming'' approach, in the sense that the resulting discrete approximations do not lie in $H^{-1/2}_{\Gamma}$, the space in which the continuous variational problem is posed. This is because conventional BEM basis functions are elements of $L_2(\Gamma_\infty)$, the intersection of which with $H^{-1/2}_{\Gamma}$ is trivial. This complicates the analysis of Galerkin implementations, since C\'ea's lemma (e.g., \cite[Theorem 8.1]{Steinbach}), or its standard modifications, cannot be invoked. 
In \cite{BEMfract} we showed how this can be overcome using the framework of Mosco convergence,
proving that, in the case of piecewise-constant basis functions, the BEM approximations on the prefractals converge to the exact BIE solution on $\Gamma$ as the prefractal level tends to infinity, provided that the prefractals satisfy a certain geometric constraint and the corresponding mesh widths tend to zero at an appropriate rate \cite[Thm.~5.3]{BEMfract}. 
However, while \cite{BEMfract} provides, to the best of our knowledge, the first proof of convergence for a numerical method for scattering by fractals, we were unable in \cite{BEMfract} to prove any \textit{rates of convergence}.

In the current paper we present an alternative approach, in which the fractal nature of the scatterer is explicitly built into the numerical discretization.
Specifically, we propose and analyse a ``Hausdorff  BEM'', which is a Galerkin implementation of an $H^{-1/2}_{\Gamma}$-conforming discretization in which the basis functions are the product of piecewise-constant functions and $\cH^d|_\Gamma$, the Hausdorff $d$-measure restricted to $\Gamma$.
A key advantage of the conforming nature of our approximations is that convergence of our Hausdorff BEM can be proved using C\'ea's lemma. Furthermore, extensions that we make in \S\ref{sec:Wavelets} of  the wavelet decompositions from \cite{Jonsson98} to negative exponent spaces allow us to obtain error bounds quantifying the convergence rate of our approximations, under appropriate and natural smoothness assumptions on the exact BIE solution. While these smoothness assumptions have not been proved for the full range that we envisage (see Proposition \ref{prop:epsilon}), the convergence rates observed in our numerical results in \S\ref{sec:NumericalResults} support a conjecture (Conjecture \ref{ass:Smoothness}) that they hold.%

Implementation of our Hausdorff BEM requires the calculation of the entries of the Galerkin linear system, which involve both single and double integrals with respect to the Hausdorff measure $\cH^d$. To evaluate such integrals we apply the quadrature rules proposed and analysed in \cite{HausdorffQuadrature}, in which the self-similarity of $\Gamma$ is exploited to reduce the requisite singular integrals to regular integrals, which can be treated using a simple midpoint-type rule. By combining the quadrature error analysis provided in \cite{HausdorffQuadrature} with novel inverse inequalities on fractal sets (proved in \S\ref{sec:InverseEstimates}) we are able to present a fully discrete analysis of our Hausdorff BEM, subject to the aforementioned smoothness assumptions.

An outline of the paper is as follows.
In \S\ref{sec:prelim} we collect some basic results that will be used throughout the paper on Hausdorff measure and dimension, singular integrals on $d$-sets, iterated function systems, and function spaces; in particular,  in \S\ref{sec:FunctionSpaces} we introduce the function spaces $\IH^t(\Gamma)$ that are trace spaces on $d$-sets that will play a major role in our analysis, and recall connections to the classical Sobolev spaces $H^s_\Gamma$ established recently in \cite{caetano2019density}.
In \S\ref{sec:Wavelets} we recall from \cite{Jonsson98} the construction, for $n-1<d\leq n$,
of wavelets on $d$-sets %
that are the attractors of iterated function systems satisfying the standard open set condition, and, for $n-1<d<n$, the characterisations of Besov spaces on these $d$-sets (which we show in Appendix \ref{app:Besov} coincide with our trace spaces $\IH^t(\Gamma)$ for positive $t$) in terms of wavelet expansion coefficients. We also extend, in Corollary \ref{cor:Wavelets},  these characterisations, which are crucial to our later best-approximation error estimates,  to $\IH^t(\Gamma)$ for a range of negative $t$ via duality arguments.

In \S\ref{sec:BVPsBIEs} we state the BVP and BIE for the Dirichlet screen scattering problem, showing, in the case when $\Gamma$ is a $d$-set, that  the BIE can be formulated in terms of a version $\IS$ of the single-layer potential operator which we show, in Propositions \ref{lem:cont} and \ref{prop:epsilon}, maps $\IH^{t-t_d}(\Gamma)$ to $\IH^{t+t_d}(\Gamma)$, for $|t|<t_d$ and a particular $d$-dependent $t_d\in (0,1/2]$, indeed is invertible between these spaces for $|t|<\epsilon$ and some $0<\epsilon\leq t_d$. %
(The spaces $\IH^{t_d}(\Gamma)\subset \IL_2(\Gamma)\subset \IH^{-t_d}(\Gamma)$ form a Gelfand triple, with $\IL_2(\Gamma)$ the space of square-integrable functions on $\Gamma$ with respect to $d$-dimensional Hausdorff measure as the pivot space, analogous to the usual Gelfand triple $H^{1/2}(\Gamma)\subset L_2(\Gamma)\subset \widetilde H^{-1/2}(\Gamma)$ in scattering by a classical screen $\Gamma$ that is a bounded relatively open subset of $\Gamma_\infty$.) %
Moreover, as Theorem \ref{lem:ISasHauss}, we show the key result that, when acting on $\IL_\infty(\Gamma)$ (which contains our BEM approximation spaces), $\IS$ has the usual representation as an integral operator with the Helmholtz fundamental solution as kernel, but now integrating with respect to $d$-dimensional Hausdorff measure.  %

In \S\ref{sec:HausdorffBEM} we describe the design and implementation of our Hausdorff BEM, and state and prove our convergence results, showing that, at least in the case that $\Gamma$ is the disjoint attractor of an iterated function system with $n-1<\dim_H(\Gamma)<n$, all the results that are achievable for classical Galerkin BEM (convergence and superconvergence results in scales of Sobolev spaces, inverse and condition number estimates, fully discrete error estimates\footnote{Our fully discrete error estimates require, additionally, that $\Gamma$ is hull-disjoint in the sense introduced below \eqref{eq:JQdef2}.}) can be carried over to this Hausdorff measure setting (we defer to Appendix B the details of our strongest inverse estimates, derived via a novel extension of bubble-function type arguments to cases where the elements have no interior). %

In \S\ref{sec:NumericalResults} we present numerical results, for cases where $\Gamma$ is a Cantor set or Cantor dust, illustrating the sharpness of our theoretical predictions. We show that our error estimates appear to apply also in cases, such as the Sierpinski triangle, where $\Gamma$ is not disjoint so that the conditions of our theory are not fully satisfied. We also make comparisons, in terms of accuracy as a function of numbers of degrees of freedom, with numerical results obtained by applying conventional BEM on a sequence of prefractal approximations to $\Gamma$, for which we have, as discussed above, only a much more limited theory \cite{BEMfract}.

In \S\ref{sec:Conclusions} we offer some conclusions and suggestions for future work. In Appendix \ref{sec:TableOfDefns} we provide a table of definitions for easy reference.

\section{Preliminaries}
\label{sec:prelim}
In this section we collect a number of preliminary results that will underpin our analysis.

\subsection{Hausdorff measure and dimension}
\label{sec:HausdorffMeasure}
For $E\subset\R^n$ and $\alpha\geq 0$ we recall (e.g., from \cite{Fal}) the definition of the Hausdorff $\alpha$-measure
of $E$,
\[ \cH^\alpha(E):=\lim_{\delta\to 0} \left(\inf \sum_{i=1}^\infty (\mathrm{diam}(U_i))^\alpha\right)\in[0,\infty)\cup\{\infty\},\]
where, for a given $\delta>0$, the infimum is over all countable covers of $E$
by a collection $\{U_i\}_{i\in\N}$ of subsets of $\R^n$ with $\diam (U_i) \leq \delta$ for each $i$. %
Where $\R^+:= [0,\infty)$, the Hausdorff dimension of $E$ is then defined to be
\[ \dimH(E):= \sup\{\alpha\in \R^+: \cH^\alpha(E)=\infty\}=\inf\{\alpha\in \R^+: \cH^\alpha(E)=0\}\in[0,n].\]
In particular, if $E\subset \R^n$ is Lebesgue measurable then $\cH^n(E)=\mathfrak{c}_n|E|$, for some constant $\mathfrak{c}_n>0$ dependent only on $n$, where $|E|$ denotes the ($n$-dimensional) Lebesgue measure of $E$. Thus $\dimH(E)=n$ if $E\subset \R^n$ has positive Lebesgue measure. %

As in \cite[\S1.1]{JoWa84} and \cite[\S3]{Triebel97FracSpec}, given $0<d\leq n$, a closed set $\Gamma\subset \R^n$ is said to be a {\em$d$-set} if there exist $c_{2}>c_1>0$ such that%
\begin{align}
\label{eq:dset}
c_{1}r^{d}\leq\mathcal{H}^{d}\big(\Gamma\cap B_{r}(x)\big)\leq c_{2}r^{d},\qquad x\in\Gamma,\quad0<r\leq1,
\end{align}
where $B_r(x)\subset \R^n$ denotes the closed ball of radius $r$ centred on $x$.
Condition \rf{eq:dset} implies that $\Gamma$ is uniformly locally $d$-dimensional in the sense that $\dimH(\Gamma\cap B_r(x))=d$ for every $x\in \Gamma$ and $r>0$.
In particular (see the discussion in \cite[\S2.4]{Fal}) \eqref{eq:dset} implies that $0<\cH^d(\Gamma\cap B_R(0))<\infty$ for all sufficiently large $R>0$, so that $\dimH(\Gamma)=d$.

\subsection{Singular integrals on compact \texorpdfstring{$d$}{d}-sets}
Our Hausdorff BEM involves the discretization of a weakly singular integral equation in which integration is carried out with respect to Hausdorff measure.
In order to derive the basic integrability results we require, we appeal to the following lemma, which is \cite[Lemma 2.13]{CC08} with the dependence of the equivalence constants made explicit.
\begin{lem}
\label{lem:AntonioModified}
Let $0<d\leq n$ and let $\Gamma\subset \R^n$ be a compact $d$-set, satisfying \rf{eq:dset} for some constants $0<c_1<c_2$. Let $x\in \Gamma$ and let $f:(0,\infty)\to [0,\infty)$ be non-increasing and continuous. Then, for some constants $C_2>C_1>0$ depending only on $c_1$, $c_2$, $n$, and the diameter of $\Gamma$,
\begin{align}
\label{eq:AntonioModified}
C_1d\int_0^{\diam(\Gamma)} r^{d-1}f(r)\,\rd r \leq
\int_\Gamma f(|x-y|)\, \rd \cH^d(y) \leq C_2d\int_0^{\diam(\Gamma)} r^{d-1}f(r)\,\rd r.
\end{align}
\end{lem}

\begin{rem} \label{rem:AntonioModified} %
If  $\Gamma\subset \R^n$ is compact and the right-hand inequality in \rf{eq:dset} holds, i.e., $\cH^d(\Gamma\cap B_r(x))\leq c_2r^d$, for $x\in \Gamma$, $0<r\leq 1$, then, following the proof of \cite[Lemma 2.13]{CC08}, we see that the right-hand bound in \eqref{eq:AntonioModified} holds, with $C_2$ depending only on $c_2$, $n$, and the diameter of $\Gamma$.
\end{rem}

From the above lemma we obtain the following important corollary.
\begin{cor}\label{lem:1dIntBd}
Let $0<d\leq n$ and let $\Gamma$ be a compact $d$-set.
Let
$x\in \Gamma$ and $\alpha\in \R$. Then
\begin{enumerate}[(i)]
\item $\int_{\Gamma}|x-y|^{-\alpha}\,\rd \cH^d(y)<\infty$ and $\int_{\Gamma}\int_{\Gamma}|x-y|^{-\alpha}\,\rd \cH^d(y) \rd\cH^d(x)<\infty$
if and only if $\alpha<d$;
\item $\int_{\Gamma}|\log{|x-y|}|\,\rd \cH^d(y)<\infty$ and $\int_{\Gamma}\int_{\Gamma}|\log{|x-y|}|\,\rd \cH^d(y) \rd\cH^d(x)<\infty$.
\end{enumerate}
\end{cor}
\begin{rem}
Corollary \ref{lem:1dIntBd}(i) is related to the more general correspondence between %
Hausdorff dimension and so-called ``capacitary dimension''
- see, e.g., \cite[\S17.11]{Triebel97FracSpec}.
\end{rem}

\subsection{Iterated function systems}
\label{sec:IFS}
The particular example of a $d$-set
we focus on in this paper is the attractor of an iterated function system (IFS) of contracting similarities, by which we mean a collection $\{s_1,s_2,\ldots,s_M\}$, for some $M\geq 2$, where, for each $m=1,\ldots,M$, $s_m:\R^{n}\to\R^{n}$, with
$|s_m(x)-s_m(y)| = \rho_m|x-y|$, $x,y\in \R^{n}$,
for some $\rho_m\in (0,1)$.
The attractor of the IFS is the unique non-empty compact set $\Gamma$ satisfying
\begin{equation} \label{eq:fixedfirst}
\Gamma = s(\Gamma), \quad \mbox{where} \quad s(E) := \bigcup_{m=1}^M  s_m(E), \quad  E\subset \R^{n}.
\end{equation}
We shall assume throughout that $\Gamma$ satisfies the \textit{open set condition (OSC)} \cite[(9.11)]{Fal}, meaning that there exists a non-empty bounded open set $O\subset \R^{n}$ such that
\begin{align} \label{oscfirst}
s(O) \subset O \quad \mbox{and} \quad s_m(O)\cap s_{m'}(O)=\emptyset, \quad m\neq m'.
\end{align}
Then \cite[Thm.~4.7]{Triebel97FracSpec} $\Gamma$ is a $d$-set,
where $d\in (0,n]$ is the unique solution of \begin{align} \label{eq:dfirst}
\sum_{m=1}^M  (\rho_m)^d = 1.
\end{align}
For a \textit{homogeneous }IFS, where $\rho_m=\rho \in (0,1)$ for $m=1,\ldots,M $, the solution of \rf{eq:dfirst} is
\begin{equation} \label{eq:d2}
d = \log(M)/\log(1/\rho).
\end{equation}
Returning to the general, not necessarily homogeneous, case, the OSC \rf{oscfirst} also implies (again, see \cite[Thm.~4.7]{Triebel97FracSpec}) that $\Gamma$ is \textit{self-similar} in the sense that the sets
\begin{align}
\label{eq:GammamDef}
\Gamma_m:=s_m(\Gamma),\qquad m=1,\ldots,M,
\end{align}
which are similar copies of $\Gamma$, satisfy
\begin{align}
\label{eq:SelfSim}
\cH^d(\Gamma_{m}\cap \Gamma_{m'})=0, \qquad m\neq m'.
\end{align}
That is, $\Gamma$ can be decomposed into $M$ similar %
copies of itself, whose pairwise intersections have Hausdorff measure zero.
For many of our results we make the additional assumption that the sets $\Gamma_1,\ldots,\Gamma_M $ are disjoint, in which case we say that the IFS attractor $\Gamma$ is \emph{disjoint}.
We recall that if $\Gamma$ is disjoint then it is totally disconnected \cite[Thm.~9.7]{Fal}.
Examples of IFS attractors that are disjoint are the Cantor set (\S\ref{s:exp:Cantor}) and Cantor dust (Figure \ref{Fig:Dust} and \S\ref{s:exp:Dust}), while examples that satisfy the OSC but are not disjoint include the Sierpinski triangle (\S\ref{s:exp:SnowflakeWonky}(ii)) and the unit interval $[0,1]$. The latter is the attractor of the IFS \eqref{eq:CS_IFS} with $\rho=1/2$, showing that IFS attractors, while self-similar, need not be fractal. A compendium of well-known fractal IFS attractors can be found at \cite{RiddleWebSite}.

The next result relates disjointness to the OSC.

\begin{lem}
\label{lem:DisjointOSC}
Let $\Gamma$ satisfy \rf{eq:fixedfirst}. Then
$\Gamma$ is disjoint
if and only if \rf{oscfirst} is satisfied for some open set $O$ satisfying $\Gamma\subset O$.
\end{lem}
\begin{proof}
If $O$ satisfies the OSC and $\Gamma\subset O$ then for $m'\neq m$ we have $\Gamma_m \cap \Gamma_{m'} = s_m(\Gamma)\cap s_{m'}(\Gamma) \subset s_m(O)\cap s_{m'}(O)=\emptyset$, so $\Gamma_1,\ldots,\Gamma_M $ are disjoint.
Conversely, if $\Gamma_1,\ldots,\Gamma_M $ are disjoint then $O:=\{x: \dist(x,\Gamma)<\eps\}\supset \Gamma$ satisfies the OSC, provided $\eps<\min_{m\neq m'}(\dist(\Gamma_m,\Gamma_{m'}))/(2\max_{m}\rho_m)$, since then $s(O)= \cup_m s_m(O)\subset \cup_m \{x: \dist(x,\Gamma_m)<\rho_m\eps\}\subset \{x: \dist(x,\Gamma)<\eps\}=O$, and there cannot exist $x\in s_m(O)\cap s_{m'}(O)$ for $m\neq m'$ since otherwise $\dist(\Gamma_m,\Gamma_{m'})\leq \dist(x,\Gamma_m) + \dist(x,\Gamma_{m'})<(\rho_m+\rho_{m'})\eps\leq 2\eps\max_{m}\rho_m$, which would contradict the definition of $\eps$.
\end{proof}

The following lemma, which shows that $\Gamma$ is disjoint only if $d<n$, motivates the restriction of our results in large parts of \S\ref{sec:Wavelets} %
to the case $d<n$.
\begin{lem}\label{lem:disconnected}
Suppose that $\Gamma$ satisfies \rf{eq:fixedfirst} and the OSC \eqref{oscfirst} holds for some bounded open $O\subset \R^n$. Then  $\Gamma \subset \overline{O}$, with equality if and only if $d=\dimH(\Gamma)=n$. If $\Gamma$ is disjoint then $0<d<n$.
\end{lem}
\begin{proof}
Arguing as on \cite[p.~809]{BEMfract}, $\Gamma \subset \overline{O}$, and, if $d<n$, then the Lebesgue measure of $s(\overline O)$,  $|s(\overline O)|\leq \sum_{m=1}^M\rho_m^n|\overline O|< \sum_{m=1}^M\rho_m^d|\overline O|=|\overline{O}|$, so that $s(\overline O) \neq \overline{O}$, so that (since $\Gamma=s(\Gamma)$), $\Gamma\neq \overline{O}$.

Suppose now that $d=n$. Arguing as above and on \cite[p.~809]{BEMfract},
$
|s(O)| = \sum_{m=1}^M |s_m(O)|= \sum_{m=1}^M\rho_m^n|O|=|O|
$,
so that
$|s(O)| = |O|$.
As claimed in \cite[p.~809]{BEMfract}, this implies that $s(\overline O)= \overline O$.  To see this, note that $s(O)\subset O$, so that  $s(\overline O) = \overline{s(O)}\subset \overline O$. Thus, if $s(\overline O)\neq \overline O$, there exists $x\in \overline O\setminus s(\overline O)$ so that, for some $y\in O$ near $x$ and some $\epsilon>0$, $B_\epsilon(y)\subset O\setminus s(\overline O)\subset O\setminus s(O)$, which contradicts $|O|=|s(O)|$. Further, since $\Gamma$ is the unique fixed point of $s$, $s(\overline O) = \overline O$ implies $\Gamma=\overline O$, and that
$\Gamma$ is not disjoint then follows from Lemma \ref{lem:DisjointOSC}.
\end{proof}

\subsection{Function spaces on subsets of \texorpdfstring{$\R^n$}{Rn}}
\label{sec:FunctionSpaces}

Here we collect some results on function spaces from \cite{ChaHewMoi:13,caetano2019density,Jonsson98}. Our function spaces will be complex-valued, and we shall repeatedly use the following terminology relating to dual spaces\footnote{For notational convenience we follow the convention of \cite{ChaHewMoi:13,BEMfract} and work throughout with dual spaces of antilinear functionals rather than linear functionals.}. If $X$ and $Y$ are Hilbert spaces, $X^*$ is the dual space of $X$, and $I:Y\to X^*$ is a unitary isomorphism, we say that $(Y,I)$ is a \textit{unitary realisation} of $X^*$ and define the \textit{duality pairing} $\langle \cdot, \cdot\rangle_{Y\times X}$ by 
$\langle y, x\rangle_{Y\times X}:=Iy(x)$, for $y\in Y,x\in X$. Having selected a unitary realisation $(Y,I)$ of $X^*$, we adopt $(X,I^*)$ as our unitary realisation of $Y^*$, where $I^*x(y) := \overline{\langle y, x\rangle_{Y\times X}}$, so that $\langle x, y\rangle_{X\times Y}= \overline{\langle y, x\rangle_{Y\times X}}$, for $y\in Y,x\in X$.

For $s\in \R$ let $H^s\Rn$ denote the Sobolev space of tempered distributions $\varphi$ for which the norm $\|\varphi\|_{H^s\Rn}=\left(\int_{\R^n} (1+|\xi|^2)^s|\widehat{\varphi}(\xi)|^2\,\rd\xi\right)^{1/2}$ is finite\footnote{Here $\widehat{\varphi}$ is the Fourier transform of $\varphi$, normalised so that $\widehat{\varphi}(\xi)= (2\pi)^{-n/2}\int_{\R^n}\re^{-\ri \xi\cdot x}\varphi(x)\,\rd x$, for $\xi\in \R^n$, when $\varphi\in L_1(\R^n)$.}. %
We recall that $(H^s\Rn)^*$ can be unitarily realised
as $(H^{-s}\Rn,I^{-s})$,
where $I^{-s}:H^{-s}\Rn \to (H^s\Rn)^*$ is given by $I^{-s}\phi(\psi):=\int_{\R^n}\hat{\phi}(\xi)\overline{\hat{\psi}(\xi)}\,\rd\xi$ for $\phi \in H^{-s}\Rn$ and $\psi \in H^{s}\Rn$, so that the resulting duality pairing $\langle \cdot, \cdot\rangle_{H^{-s}\Rn\times H^s\Rn}$ extends both the $L_2(\R^n)$ inner product and the action of tempered distributions on Schwartz functions (see, e.g., \cite[\S3.1.3]{ChaHewMoi:13}).
For an open set $\Omega\subset \R^n$ we denote by $\tH^{s}(\Omega)$ the closure of $C^\infty_0(\Omega)$ in $H^{s}\Rn$, and for a closed set $E\subset \R^n$ we denote by $H^s_E$ the set of all elements of $H^s\Rn$ whose distributional support is contained in $E$. These two types of spaces are related by duality. Where $E^c:=\R^n\setminus E$ denotes  the complement of $E$ and ${}^\perp$ denotes orthogonal complement in $H^s\Rn$ ,
$(H^{-s}_E,\cI)$
is a unitary realisation of $(\tH^{s}(E^c)^{\perp})^*$ \cite[\S3.2]{ChaHewMoi:13},
where $\cI \phi(\psi):=I^{-s}\phi (\psi)$ for $\phi\in H^{-s}_E$ and $\psi\in \tH^{s}(E^c)^{\perp}$,
so that the associated duality pairing  is just the restriction to $H^{-s}_E\times \tH^{s}(E^c)^{\perp}$ of $\langle \cdot, \cdot\rangle_{H^{-s}\Rn\times H^s\Rn}$. Note also that, if $E\subset \Omega \subset \R^n$ and $E$ is compact, $\Omega$ is open, then, as a consequence of \cite[Lemma 3.24]{McLean}, $H^s_E$ is a closed subspace of $\widetilde H^s(\Omega)$. %
In addition to the spaces just introduced we use, at some points,  the standard Sobolev space $H^s(\Omega)$, for $\Omega\subset \R^n$ open and $s\in \R$, defined as the space of restrictions to $\Omega$ of the distributions $\varphi\in H^s(\R^n)$,
equipped with the quotient norm
$\|u\|_{H^s(\Omega)} := \inf_{\substack{\varphi\in H^s\Rn\\ \varphi|_\Omega=u}}\|\varphi\|_{H^s\Rn}$
(see \cite{McLean}, \cite[\S3.1.4]{ChaHewMoi:13}).%

Fix $0<d\leq n$ and let $\Gamma\subset\R^n$ be $\cH^d$-measurable. %
We denote by $\IL_2(\Gamma)$ the space of
(equivalence classes of) complex-valued functions on $\R^n$ that are measurable and square integrable with respect to
$\cH^d|_\Gamma$, normed by
\[
\|f\|_{\IL_2(\Gamma)} := \left(\int_\Gamma |f(x)|^2\,\rd \cH^d(x) \right)^{1/2}.\]
Similarly, $\IL_\infty(\Gamma)$ denotes the space of functions on $\R^n$ that are measurable and essentially bounded with respect to $\cH^d|_\Gamma$, normed by $\|f\|_{\IL_\infty}:= {\rm ess}\, \sup_{x\in \R^n}|f(x)|$.
In practice we shall view $\IL_2(\Gamma)$ and $\IL_\infty(\Gamma)$ as spaces of functions on $\Gamma$, by identifying elements of $\IL_2(\Gamma)$ and $\IL_\infty(\Gamma)$ with their restrictions to $\Gamma$.
The dual space $(\IL_2(\Gamma))^*$ can be realised in the standard way as $(\IL_2(\Gamma),\II)$, where $\II:\IL_2(\Gamma)\to \IL_2(\Gamma)^*$ is the Riesz map (a unitary isomorphism) defined by $\II f(\tf)=(f,\tf)_{\IL_2(\Gamma)}=\int_\Gamma f(x)\overline{\tf(x)}\,\rd\cH^d(x)$.

Now assume that $\Gamma$ is a $d$-set and that $0<d\leq n$. Then function spaces on $\R^n$ and $\Gamma$ are related via the trace operator $\tr$ of \cite[\S18.5]{Triebel97FracSpec}. Defining
$\tr(\varphi)=\varphi|_\Gamma\in \IL_2(\Gamma)$ for $\varphi\in C_0^\infty(\R^n)$,
one can show
\cite[Thm~18.6]{Triebel97FracSpec}
that if
$$
s>\frac{n-d}{2},
$$
which we assume through the rest of this section, and $0<d<n$, then $\tr$ extends to a continuous linear operator
\[\tr:H^{s}(\R^n) \to \IL_2(\Gamma)\]
with dense range. 
This trivially holds also for $d=n$, since the embedding of $H^s(\R^n)$ into $L_2(\R^n)$, for $s>0$, and the trace $\tr:L_2(\R^n)\to \IL_2(\Gamma)$ are both continuous with dense range.
Setting
\begin{equation} \label{eq:st}
t:=s-\frac{n-d}{2} >0,
\end{equation}
we define the trace space $\IH^{t}(\Gamma):=\tr(H^{s}(\R^n))\subset \IL_2(\Gamma)$, which we equip with the quotient norm
\[ \|f\|_{\IH^t(\Gamma)} := \inf_{\substack{\varphi\in H^s\Rn\\ \tr \varphi=f}}\|\varphi\|_{H^s\Rn}.\]
This makes $\IH^t(\Gamma)$ a Hilbert space unitarily isomorphic to the quotient space $H^{s}(\R^n)/\ker(\tr)$. Clearly,
\begin{equation} \label{eq:embed1}
\IH^{t^\prime}(\Gamma) \subset \IH^t(\Gamma)\subset \IL_2(\Gamma),
\end{equation}
for $t^\prime>t>0$, and the embeddings are continuous with dense range.
As explained in \cite[Rem.~6.4]{caetano2019density}  ($\IH^t(\Gamma)$ is denoted $\IH^t_{2,0}(\Gamma)$ in \cite[\S6]{caetano2019density}), for the case $0<d<n$, under the further assumption that $t<1$, $\IH^t(\Gamma)$ coincides (with equivalent norms) with the Besov space $B^t_{2,2}(\Gamma)$ of \cite{JoWa84}. Arguing in the same way, using \cite[Theorem VI.1]{JoWa84} and that $H^s(\R^n)$ coincides with the Besov space $B^s_{2,2}(\R^n)$ (e.g., \cite[p.~8]{JoWa84})\footnote{Our standard notation for Besov spaces on $\R^n$ is that of, e.g., \cite{Tri08}.}, this holds also for $d=n$.

For $t>0$ we denote by $\IH^{-t}(\Gamma)$ the dual space $(\IH^{t}(\Gamma))^*$. Since \eqref{eq:embed1} holds, and the embeddings are continuous with dense range, also $\IH^{-t}(\Gamma)\subset \IH^{{-t^\prime}}(\Gamma)$, for $t^\prime>t>0$, and this embedding is continuous with dense range.
Further, via the Riesz map $\II:\IL_2(\Gamma)\to \IL_2(\Gamma)^*$ introduced above,
$\IL_2(\Gamma)$ is continuously and densely embedded in $\IH^{-t}(\Gamma)$, for $t>0$.
Setting $\IH^0(\Gamma)=\IL_2(\Gamma)$, and combining these embeddings, we then have that $\IH^{t^\prime}(\Gamma)$ is embedded in $\IH^{t}(\Gamma)$ with dense image for any $t,t^\prime\in\R$ with $t^\prime>t$,
and that if 
$g\in\IH^t(\Gamma)$ for some $t\geq 0$ and $f\in\IL_2(\Gamma)$ then
\begin{align}
\label{eq:L2dualequiv}
\langle f,g\rangle_{\IH^{-t}(\Gamma) \times \IH^{t}(\Gamma)}=(f,g)_{\IL_2(\Gamma)}.
\end{align}

Suppose that \eqref{eq:st} holds. %
By the definition of $\IH^t(\Gamma)$, $\tr:H^{s}(\R^n) \to \IH^t(\Gamma)$ is a continuous linear surjection with unit norm.
Its Banach space adjoint
$\widetilde{\tr}^*: \IH^{-t}(\Gamma) \to (H^{s}(\R^n))^*$, defined by $\widetilde{\tr}^*y(x)=y(\tr x)$, for $y\in (\IH^t(\Gamma))^*$, $x\in H^s\Rn$, is then a continuous linear injection with unit norm, and composing $\widetilde{\tr}^*$ with the unitary isomorphism $(I^{-s})^{-1}$ produces a continuous linear injection
\[\tr^*:= (I^{-s})^{-1}\circ \widetilde{\tr}^* : \IH^{-t}(\Gamma) \to H^{-s}(\R^n)\]
with unit norm, which satisfies\footnote{We omit in our notation for $\tr:H^s(\R^n)\to \IH^t(\Gamma)$ any dependence on $s$. This is justified since, for every $s>\frac{n-d}{2}$, $\tr \varphi = \varphi|_\Gamma$ for $\varphi\in C_0^\infty(\R^{n})$, and $C_0^\infty(\R^{n})$ is dense in $H^s(\R^n)$. Likewise, we omit any dependence on $t$ in our notation for $\tr^*:\IH^{-t}(\Gamma)\to H^{-s}(\R^n)$. To see that this is justified, in particular that the values of the $\tr^*$ operators coincide where their domains intersect, denote $\tr^*:\IH^{-t}(\Gamma)\to H^{-s}(\R^n)$ temporarily by $\mathrm{tr}_{\Gamma,t}^*$ to make explicit the domain. Then, for $t^\prime>t>0$, $\mathrm{tr}_{\Gamma,t^\prime}^*f=\mathrm{tr}_{\Gamma,t}^*f$, for $f\in \IH^{-t}(\Gamma)\subset \IH^{-t^\prime}(\Gamma)$, as a consequence of \eqref{eq:tracedual}.}
\begin{equation} \label{eq:tracedual}
 \langle \varphi,\tr^* f \rangle_{H^{s}(\R^n) \times H^{-s}(\R^n)} = \langle \tr \varphi,f\rangle_{\IH^{t}(\Gamma)\times \IH^{-t}(\Gamma)},
\qquad f\in\IH^{-t}(\Gamma),\,\varphi\in H^s(\R^n).
\end{equation}
In particular, when $f\in \IL_2(\Gamma)$ we have that
\begin{align}
\label{eq:L2dualrep}
\langle \varphi,\tr^* f \rangle_{H^{s}(\R^n) \times H^{-s}(\R^n)} =(\tr \varphi,f)_{\IL_2(\Gamma)}.
\end{align}

Since $\tH^{s}(\Gamma^c)\subset\ker(\tr)$ the range of $\tr^*$ is contained in $H^{-s}_\Gamma$, which is the annihilator of $\tH^{s}(\Gamma^c)$ with respect to $\langle \cdot,\cdot \rangle_{H^{-s}\Rn\times H^{s}\Rn}$ \cite[Lemma 3.2]{ChaHewMoi:13}. Key to our analysis will be the following stronger result. This is proved, for the case $0<d<n$, in \cite[Prop.~6.7, Thm~6.13]{caetano2019density}, and the arguments given there (we need only the simplest special case $m=0$) extend to the case $d=n$, with the twist that, to justify the existence of a bounded right inverse $\mathcal{E}_{\Gamma,0}$ in Step 2 of the proof of \cite[Prop.~6.7]{caetano2019density}, we need to use (as above) that $H^s(\R^n)=B_{2,2}^s(\R^n)$ and \cite[Thm.~VI.3 on p.~155]{JoWa84}. %
Figure \ref{fig:FunctionSpaceScheme} shows the main relations between these function spaces.

\begin{figure}[t!]\centering
\begin{tikzpicture}
\matrix[matrix of math nodes,
column sep={13pt},
row sep={20pt,between origins},
text height=1.5ex, text depth=0.25ex] (s)
{
|[name=Ht]| \IH^t\GG &
|[name=]| \subset &
|[name=L2]| \IL_2\GG &
|[name=]| \subset &
|[name=Hmt]| \IH^{-t}\GG&
\\
\\
|[name=Hscp]| \tH^s(\Gamma^c)^\perp&
& & &
|[name=HmsG]| H^{-s}_\Gamma&
\\
\cap&&&&\cap
\\
|[name=Hs]| H^s\Rn&\subset&L_2\Rn&\subset&
|[name=Hms]| H^{-s}\Rn&
\\
};
\draw[->,>=angle 60] %
(Hscp) edge node[auto]{\(\tr\)} (Ht)
(Hmt) edge node[auto]{\(\tr^*\)} (HmsG)
;
\end{tikzpicture}
\caption{The main function spaces introduced in \S \ref{sec:FunctionSpaces} and their relations.
Here $\Gamma\subset\R^n$ is a $d$-set with $0<d\leq n$, $t=s-\frac{n-d}2\in(0,1)$, both arrows represent unitary isomorphisms. Where we write $A\subset B$, the function space $A$ is densely and continuously embedded in $B$.  Where $\cap$ separates $A$ and $B$ vertically, the space $A$ is a closed subspace of $B$.%
}
\label{fig:FunctionSpaceScheme}
\end{figure}

\begin{thm} %
\label{thm:Density}
Let $\Gamma\subset\R^n$ be a $d$-set for some $0<d\leq n$.
Let $\frac{n-d}{2}<s<\frac{n-d}{2}+1$, so that $t=s-\frac{n-d}2\in(0,1)$. Then $\ker(\tr)=\tH^{s}(\Gamma^c)$, so that $\tr|_{\tH^{s}(\Gamma^c)^\perp}:\tH^{s}(\Gamma^c)^\perp \to \IH^{t}(\Gamma)$ is a unitary isomorphism.
Accordingly,
the range of $\tr^*$ is equal to $H^{-s}_\Gamma$, and
\begin{equation}\label{eq:tr*}
\tr^*:\IH^{-t}\GG\to H^{-s}_\Gamma \quad\text{is a unitary isomorphism, where}\quad
t=s-\frac{n-d}2\in(0,1).
\end{equation}
Furthermore, $\tr^*(\IL_2(\Gamma))$ is dense in $H^{-s}_\Gamma$.
\end{thm}

\subsection{Function spaces on planar screens}
\label{sec:FunctionSpacesScreens}
For the screen scattering problem, we define function spaces on the hyperplane $\Gamma_\infty := \R^{n}\times\{0\}$ and subsets of it (for example, on the compact subset $\Gamma\subset\Gamma_\infty$ that forms the screen) by associating $\Gamma_\infty$ with $\R^{n}$ and $\Gamma$ with the set $\tilde{\Gamma}\subset\R^n$ such that $\Gamma=\tilde{\Gamma}\times\{0\}$ and applying the definitions above, so that $H^s(\Gamma_\infty) := H^s(\R^{n})$, $H^s(\Gamma) := H^s(\tilde{\Gamma})$ etc., and, when $\Gamma\subset\Gamma_\infty$ is a $d$-set, $\IH^t(\Gamma):=\IH^t(\tilde\Gamma)$. In the latter case, the operator
$\tr:H^s(\R^n)\to \IH^t(\tilde{\Gamma})$ naturally gives rise to an operator $\tr:H^s(\Gamma_\infty)\to \IH^t(\Gamma)$.

For open sets $\Omega\subset\R^{n+1}$ (e.g.\ the exterior domain $D:=\R^{n+1}\setminus{\Gamma}$)
we work with the classical Sobolev spaces\footnote{Our notation follows \cite{McLean}. Given an open set $\Omega\subset \R^{n+1}$, $W^1(\Omega)$ is the Sobolev space whose norm is defined ``intrinsically'', via integrals over $\Omega$, while $H^1(\Omega)$ is defined (see \S\ref{sec:FunctionSpaces}) ``extrinsically'' as the set of restrictions to $\Omega$ of functions in $H^1(\R^{n+1})$. These spaces coincide if $\Omega$ is Lipschitz (e.g., \cite[Theorem 3.16]{McLean}), in particular if $\Omega = \R^{n+1}$, but not in general, in particular not, in general, %
for $\Omega = D$.}  $W^1(\Omega)$ and $W^1(\Omega,\Delta)$, normed by $\|u\|_{W^1(\Omega)}^2=\|u\|_{L_2(\Omega)}^2+\|\nabla u\|_{L_2(\Omega)}^2$ and $\|u\|_{W^1(\Omega,\Delta)}^2=\|u\|_{L_2(\Omega)}^2+\|\nabla u\|_{L_2(\Omega)}^2+ \|\Delta u\|_{L_2(\Omega)}^2$ respectively, and their ``local'' versions $W^{1,{\rm loc}}(\Omega)$ and $W^{1,{\rm loc}}(\Omega,\Delta)$, defined as the sets of measurable functions on $\Omega$ whose restrictions to any bounded open $\Omega'\subset\Omega$ are in $W^{1}(\Omega')$ or $W^1(\Omega',\Delta)$ respectively. We denote by $\gamma^\pm:W^1(U^\pm)\to H^{1/2}(\Gamma_\infty)$ and $\dn^\pm:W^1(U^\pm,\Delta)\to H^{-1/2}(\Gamma_\infty)$ the standard Dirichlet and Neumann trace operators from the upper and lower half spaces $U^\pm:=\{x\in\R^{n+1},\pm x_{n+1}>0\}$ onto the hyperplane $\Gamma_\infty$, where
the normal vector is assumed to point into $U^+$ in the case of the Neumann trace. Explicitly, the traces are the extension by density of $\gamma^\pm (u)(x):=\lim_{\substack{x'\to x\\x'\in U^\pm}}u(x')$ and $\dn^\pm u(x) := \lim_{\substack{x'\to x\\x'\in U^\pm}}\frac{\partial u}{\partial x_{n+1}}(x')$ for $u\in C^\infty_0(\R^{n+1})|_{U^\pm}$ %
and $x\in \Gamma_\infty$.
We note that if $u\in W^{1}(\R^{n+1})$ then $\gamma^+(u|_{U^+})=\gamma^-(u|_{U^-})$ \cite[\S2.1]{ScreenPaper}.
Finally,
let $C^\infty_{0,\Gamma}$ denote the set of functions in $C^\infty_0(\R^{n+1})$ that equal one in a neighbourhood of $\Gamma$.

\section{Wavelet decompositions}
\label{sec:Wavelets}

The spaces $\IH^t(\Gamma)$ were defined in \S\ref{sec:FunctionSpaces} for $\Gamma\subset\R^n$ a general $d$-set with $0<d\leq n$. In the case where $\Gamma$ is a disjoint IFS attractor (in the sense of \S\ref{sec:IFS}, in which case, 
by Lemma \ref{lem:disconnected}, 
$d<n$), it was shown in \cite{Jonsson98} that, when $n-1<d=\dimH(\Gamma)<n$, %
the spaces $\IH^t(\Gamma)$ can be characterised in terms of wavelet decompositions for $t>0$.\footnote{More precisely, \cite{Jonsson98} showed that the Besov spaces $B^{p,q}_\alpha(\Gamma)$, as defined in \cite[\S6]{Jonsson98} (and see Definition \ref{def:B^p,q_alpha} below), can be characterised in this way for $\alpha>0$ and $1\leq p,q\leq \infty$, provided $\Gamma$ preserves Markov's inequality, which holds in particular (see Remark \ref{rem:key}) if $d>n-1$. We show in Appendix \ref{app:Besov} (see Corollary \ref{cor:equiv} and Remark \ref{rem:key}) that, for $t>0$, $\IH^t(\Gamma)=B^{2,2}_t(\Gamma)$ with equivalence of norms if $\Gamma$ is a $d$-set with $d>n-1$, so that this characterisation carries over to our trace spaces $\IH^t(\Gamma)$.} %
 This, more precisely an extension of this characterisation to negative $t$, will be central to our BEM convergence analysis later.
In the current section we recap the notation and main results from \cite{Jonsson98} that we will need, initially assuming only that $n-1<d\leq n$ and that the OSC is satisfied.

Let $\Gamma$ be the attractor of an IFS $\{s_1,\ldots,s_M\}$ as in \rf{eq:fixedfirst}, and assume that the OSC \rf{oscfirst} holds.
Following \cite{Jonsson98},
for $\ell\in \N$ we define the set of multi-indices $I_\ell:=\{1,\ldots,M\}^\ell\!=\{\bm = (m_1,m_2,\ldots,m_\ell)$, $\,1\leq m_l\leq M , \, l=1,2,\ldots,\ell\}$, and for $E\subset\R^n$ and $\bm\in I_\ell$ we define
$E_{\bm}=s_{m_1}\circ s_{m_2}\circ \ldots \circ s_{m_\ell}(E)$.
We also set $I_0:=\{0\}$ and adopt the convention that $E_0:=E$. (We will use these notations especially in the case $E=\Gamma$.) For $S=\N$ and $S= \N_0:= \N\cup\{0\}$ we use the notation
\begin{equation} \label{eq:ISdef}
I_S:= \bigcup_{\ell\in S} I_\ell
\end{equation}
and, for $\bm = (m_1,\ldots,m_\ell)$, set $\bm_-:= (m_1,...,m_{\ell-1})$ if $\ell\in \N$ with $\ell\geq 2$, and set $\bm_-:= 0$ if $\ell=1$.  %
 An example of use of the notation $E_\bm$ is shown in Figure \ref{fig:Basis}.
Note that this notation extends that of \rf{eq:GammamDef} where the sets $\Gamma_1,\ldots,\Gamma_M$ were introduced, corresponding to the case $E=\Gamma$ and $\ell=1$ here.
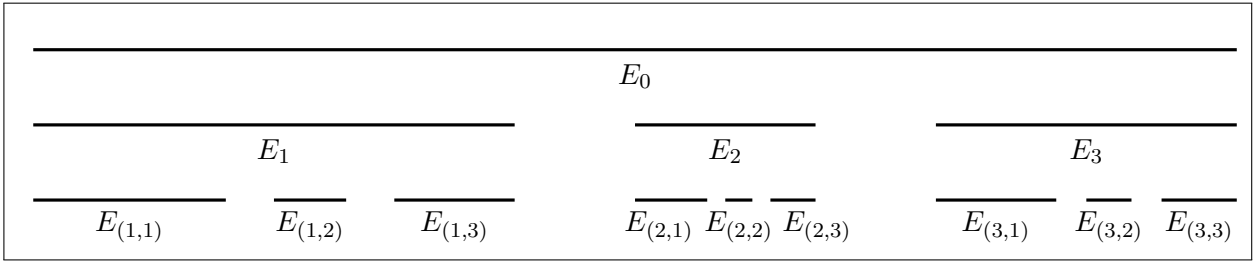
\begin{figure}
\fbox{
\def\xa{0}\def\xb{0.16}\def\xc{0.2}\def\xd{0.26}\def\xe{0.3}\def\xf{0.4}\def\xg{0.5}\def\xh{0.56}\def\xi{0.575}\def\xj{0.5975}\def\xk{0.6125}\def\xl{0.65}\def\xm{0.75}\def\xn{0.85}\def\xo{0.875}\def\xp{0.9125}\def\xq{0.9375}\def\xr{1}
\begin{tikzpicture}[x=16cm]
\def\dx{0.05}
\def\dy{0.35}
\newcommand{\y}{0}
\def\y{0}
\draw[very thick] (\xa,\y)--(\xr,\y);\node at ({(\xa+\xr)/2},\y-\dy) {$E_{0}$};
\def\y{-1}
\draw[very thick] (\xa,\y)--(\xf,\y);\node at ({(\xa+\xf)/2},\y-\dy) {$E_{1}$};
\draw[very thick] (\xg,\y)--(\xl,\y);\node at ({(\xg+\xl)/2},\y-\dy) {$E_{2}$};
\draw[very thick] (\xm,\y)--(\xr,\y);\node at ({(\xm+\xr)/2},\y-\dy) {$E_{3}$};
\def\y{-2}
\draw[very thick] (\xa,\y)--(\xb,\y);\node at ({(\xa+\xb)/2},\y-\dy) {$E_{(1,1)}$};
\draw[very thick] (\xc,\y)--(\xd,\y);\node at ({(\xc+\xd)/2},\y-\dy) {$E_{(1,2)}$};
\draw[very thick] (\xe,\y)--(\xf,\y);\node at ({(\xe+\xf)/2},\y-\dy) {$E_{(1,3)}$};
\draw[very thick] (\xg,\y)--(\xh,\y);\node at ({(\xg+\xh)/2-0.01},\y-\dy) {$E_{(2,1)}$};
\draw[very thick] (\xi,\y)--(\xj,\y);\node at ({(\xi+\xj)/2},\y-\dy) {$E_{(2,2)}$};
\draw[very thick] (\xk,\y)--(\xl,\y);\node at ({(\xk+\xl)/2+0.02},\y-\dy) {$E_{(2,3)}$};
\draw[very thick] (\xm,\y)--(\xn,\y);\node at ({(\xm+\xn)/2},\y-\dy) {$E_{(3,1)}$};
\draw[very thick] (\xo,\y)--(\xp,\y);\node at ({(\xo+\xp)/2},\y-\dy) {$E_{(3,2)}$};
\draw[very thick] (\xq,\y)--(\xr,\y);\node at ({(\xq+\xr)/2},\y-\dy) {$E_{(3,3)}$};
\draw[white](0,.5)--(1,.5); %
\end{tikzpicture}}
\caption{
Illustration of the sets $E_{\bm}$, $\bm\in I_\ell$, $\ell=0,1,2$, with $E=[0,1]$, for the IFS $s_1(x)=0.4x$, $s_2(x)=0.15x+0.5$, $s_3(x)=0.25x+0.75$, associated with a Cantor-type set with $M=3$.
}
\label{fig:Basis}
\end{figure}

Let $\IW_0$ be the space of constant functions on $\Gamma$, a one-dimensional subspace of $\IL_2(\Gamma)$ spanned by
\[\chi_0:=1_\Gamma/\cH^d(\Gamma)^{1/2}.\]
More generally, for $\ell\in \N$ let %
\[\IW_\ell:=\{f\in\IL_2(\Gamma):\forall\bm\in I_\ell, \, \exists c_{\bm}\in \C \mbox{ such that }  f(x)=c_{\bm} \mbox{ for }\cH^d\mbox{-a.e. } x\in \Gamma_{\bm}\}\subset\IL_2(\Gamma).\]
Since $\cH^d(\Gamma_{\bm}\cap \Gamma_{\bm'})=0$ for $\bm\neq \bm'$ (a consequence of \rf{eq:SelfSim} (self-similarity)),
$\IW_\ell$ is a $M^\ell$-dimensional subspace of $\IL_2(\Gamma)$ with orthonormal basis
\[ \{\chi_{\bm}\}_{\bm\in I_{\ell}},\]
where
\begin{align}
\label{eq:VkBasisDefn}
\chi_{\bm}(x):=
\begin{cases}
\frac{1}{\cH^d(\Gamma_{\bm})^{1/2}}, & x\in \Gamma_{\bm},\\
0,& \text{otherwise}.
\end{cases}
\end{align}
Clearly $\IW_0\subset \IW_1\subset \IW_2\subset\cdots \subset\IL_2(\Gamma)$.
But the bases we introduced above are not hierarchical, in the sense that the basis for $\IW_{\ell+1}$ does not contain that for $\IW_{\ell}$. %
Following \cite{Jonsson98} we introduce hierarchical wavelet bases on the $\IW_\ell$ spaces by decomposing
\begin{align}
\label{eq:VkDecomp}
\IW_\ell = \IW_0 \oplus \left(\bigoplus_{\ell'=0}^{\ell-1} (\IW_{\ell'+1}\ominus \IW_{\ell'})\right),
\end{align}
where $\IW_{\ell'+1}\ominus \IW_{\ell'}$ denotes the orthogonal complement of $\IW_{\ell'}$ in $\IW_{\ell'+1}$.
As already noted, $\IW_0$ is one-dimensional, with orthonormal basis $\{\psi_0\}$, where $\psi_0=\chi_0$.
The space $\IW_1\ominus \IW_0$ is $(M -1)$-dimensional, and an orthonormal basis
$\{\psi^m\}_{m=1,\ldots,M -1}$ of $\IW_1\ominus \IW_0$ can be obtained by applying the Gram-Schmidt orthonormalization procedure to the (non-orthonormal) basis $\{\widetilde\psi^m\}_{m=1,\ldots,M -1}$ defined by %
\begin{align}
\label{eq:BasisDefn}
\widetilde\psi^m (x):=
\begin{cases}
(\cH^d(\Gamma_{m}))^{-1}, & \mbox{for } \cH^d\mbox{-a.e. } x\in \Gamma_m,\\
-(\cH^d(\Gamma_{m+1}))^{-1}, & \mbox{for } \cH^d\mbox{-a.e. } x\in \Gamma_{m+1},\\
0,& \text{otherwise},
\end{cases}
\qquad m=1,\ldots,M -1.
\end{align}
For $\ell\in \N$ the space $\IW_{\ell+1}\ominus \IW_\ell$ is $((M-1)M^\ell)$-dimensional, and an orthonormal basis of $\IW_{\ell+1}\ominus \IW_\ell$ (see Figure \ref{fig:PsiBasisPlot}) is given by $\{\psi^m_{\bm}\}_{\bm\in I_{\ell},\,m=1,\ldots,M-1}$, %
where
\[ \psi^m_{\bm} := (\cH^d(\Gamma_{\bm}))^{-1/2}\psi^m\circ s_{m_\ell}^{-1}\circ s_{m_{\ell-1}}^{-1}\circ\cdots \circ s_{m_1}^{-1},
\quad m=1,2,\ldots,M -1,\,\,\bm\in I_\ell.
\]
Hence, recalling \rf{eq:VkDecomp} and setting $\psi^m_{0}=\psi^m$ for $m=1,2,\ldots,M -1$, we obtain the following orthonormal basis of $\IW_\ell$:
\[
\{\psi_0\}\cup\{\psi^m_{\bm}\}_
{\bm\in I_{\ell'},\,\ell'\in\{0,\ldots,\ell-1\},\,m\in\{1,\ldots,M-1\}}.
\]
\begin{figure}[t!]%
\includegraphics[width=\textwidth,clip,trim=100 15 95 15]{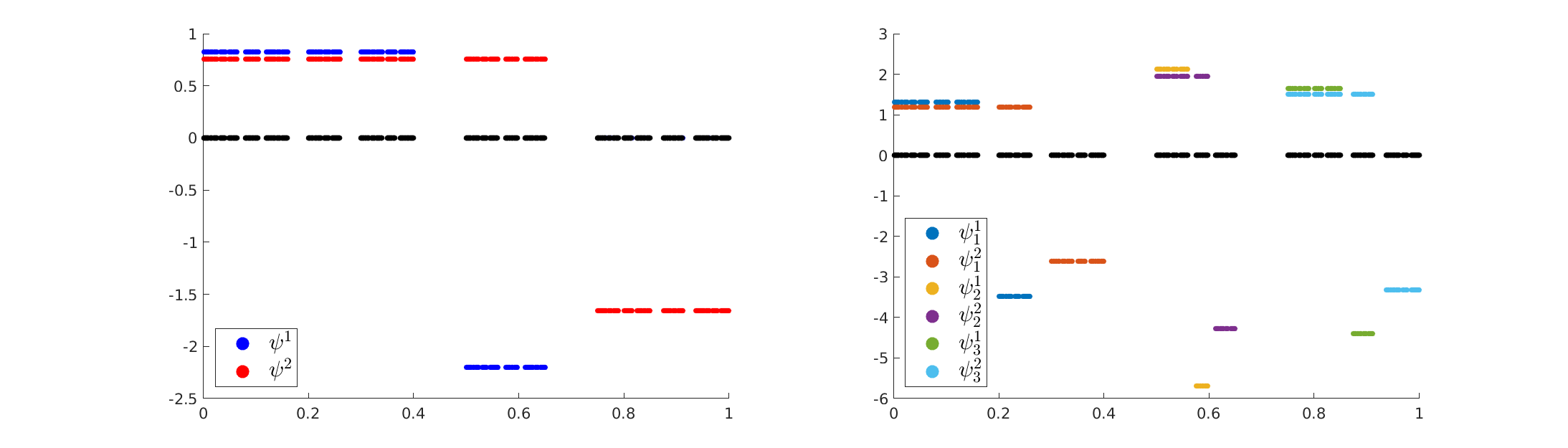}
\caption{Graphs of the orthonormal basis functions $\psi^1,\psi^2$ of $\IW_1\ominus\IW_0$ (left) and $\psi^m_\bm$, $\bm\in I_1$, of $\IW_2\ominus\IW_1$ (right) for the IFS of Figure~\ref{fig:Basis}. The black lines are the components of the attractor $\Gamma$. Where the values of $\psi^1$, $\psi^2$, and $\psi^n_\bm$ on $\Gamma$ are not shown explicitly, the values are zero, i.e.\ the graphs coincide with the black lines.
}
\label{fig:PsiBasisPlot}
\end{figure}
These bases are hierarchical, in the sense that the basis for $\IW_{\ell+1}$ contains that for $\IW_{\ell}$.
Furthermore, as noted in \cite[p.~334]{Jonsson98} (and demonstrated in the proof of Theorem \ref{thm:Convergence} below), elements of $\IL_2(\Gamma)$ can be approximated arbitrarily well by elements of $\IW_\ell$ as $\ell\to\infty$, which implies that
\[\{\psi_0\}\cup \{\psi^m_{\bm}\}_{\bm\in I_\ell,\,\ell\in\N_0,\, m\in\{1,\ldots,M -1\}}\]
is a complete orthonormal set in $\IL_2(\Gamma)$. %
Hence every $f\in \IL_2(\Gamma)$ has a unique representation
\begin{align}
\label{eq:L2Rep}
f = \beta_0  \psi_0 + \sum_{m=1}^{M -1}\sum_{\ell=0}^\infty \sum_{\bm\in I_\ell}\beta^m_{\bm} \psi^m_{\bm},
\end{align}
with %
\begin{align}
 \label{eq:BetaDef1}
 \beta_0  = \beta_0 (f):=(f,\psi_0)_{\IL_2(\Gamma)} \quad \text{and} \quad \beta^m_{\bm} = \beta^m_{\bm}(f) := (f,\psi^m_{\bm})_{\IL_2(\Gamma)}, \,\bm\in I_\ell,\,\ell\in\N_0,\, m\in\{1,\ldots,M -1\},
 \end{align}
and
\[
\|f\|_{\IL_2(\Gamma)} = \Bigg(|\beta_0 |^2 + \sum_{m=1}^{M -1}\sum_{\ell=0}^\infty \sum_{\bm\in I_\ell}|\beta^m_{\bm}|^2\Bigg)^{1/2}.
\]

The next result, which combines Theorems 1 and 2 in \cite{Jonsson98}, provides a characterization of the space $\IH^t(\Gamma)\subset\IL_2(\Gamma)$ (introduced after \rf{eq:st}) for $n-1<d=\dimH(\Gamma)<n$ and $0<t<1$ in terms of the wavelet basis introduced above, under the assumption that
$\Gamma$ is a disjoint IFS attractor.
This assumption ensures that the piecewise-constant spaces $\IW_\ell$ are contained in $\IH^t(\Gamma)$ for all $t>0$ (see \cite[Thm.~2]{Jonsson98}). In the statement of Theorem \ref{thm:Jonsson} the set $J_\nu$ is defined for $\nu\in\mathbb{Z}$ by %
\begin{equation} \label{eq:Jnudef}
J_\nu := \big\{\bm %
\in I_{\N_0} %
:2^{-\nu}\leq \diam(\Gamma_{\bm}) < 2^{-\nu +1}\big\},
\end{equation}
and $\nu_0$ is defined to be the unique integer such that $0\in J_{\nu_0}$, i.e.\ such that $2^{-\nu_0}\leq \diam(\Gamma)<2^{-\nu_0+1}$.
Note that\footnote{As an example of these definitions, suppose that $\Gamma$ is the (disjoint) attractor of the IFS illustrated in Figure \ref{fig:Basis}, in which case $\diam(\Gamma_\bm) = \diam(E_\bm)$, where $E_\bm$ is as defined in Figure \ref{fig:Basis}, in particular $\diam(\Gamma_0)=\diam(\Gamma)=1$. Then $\nu_0=0$, $J_0=\{0\}$, $J_1=\emptyset$, $J_2=\{1,3\}$, and $J_3=\{(1,1),2\}$, so that, since $I_0=\{0\}$ and $I_1=\{1,2,3\}$, $I_0=J_0$ and $J_0\cup J_1 \cup J_2 \subset I_0\cup I_1 \subset J_0\cup J_1 \cup J_2 \cup J_3$.}
\[\bigcup_{\nu=\nu_0}^\infty J_\nu = \bigcup_{\ell=0}^\infty I_\ell,
 \]
with disjoint unions on both sides, so $\sum_{\nu=\nu_0}^\infty\sum_{\bm\in J_\nu}F(\bm) = \sum_{\ell=0}^\infty\sum_{\bm\in I_\ell}F(\bm)$
whenever
the convergence is unconditional, as is the case, for instance, for \eqref{eq:L2Rep}. %
For convenience we introduce in this theorem a norm $\|\cdot\|_t$ that is different, but trivially equivalent to that used in \cite{Jonsson98}, which was $|\beta_0 | + \sum_{m=1}^{M -1}\left(\sum_{\nu=\nu_0}^\infty 2^{2\nu t}\sum_{\bm\in J_\nu}|\beta^m_{\bm}|^2 \right)^{1/2}$.

\begin{thm}[{\cite[Thms 1 \& 2]{Jonsson98}}]
\label{thm:Jonsson}
Let $\Gamma$ be a disjoint IFS attractor with $n-1<d=\dimH(\Gamma)<n$,
and let $0<t<1$. Then
\[\IH^t(\Gamma) = \{f\in \IL_2(\Gamma): \, \|f\|_{t}<\infty\}, \]
with
\[\|f\|_{t} := \Bigg(|\beta_0 |^2 + \sum_{m=1}^{M -1}\sum_{\nu=\nu_0}^\infty 2^{2\nu t}\sum_{\bm\in J_\nu}|\beta^m_{\bm}|^2 \Bigg)^{1/2}, \]
where $\beta_0  $ and $\{\beta^m_{\bm}\}$ are the coefficients from \eqref{eq:L2Rep}. Furthermore, $\|\cdot\|_{t}$ and $\|\cdot\|_{\IH^t(\Gamma)}$ are equivalent.
If $f\in \IH^t(\Gamma)$ then \eqref{eq:L2Rep} converges unconditionally in $\IH^t(\Gamma)$.
\end{thm}

\begin{rem}[Fractional norms in the homogeneous case]
In the homogeneous case where $\rho_m=\rho $ for each $m=1,\ldots,M$,
we have
\[\|f\|_{t} = \Bigg(|\beta_0 |^2 + \sum_{m=1}^{M -1}\sum_{\ell=0}^\infty 2^{2\nu(\ell)t}\sum_{\bm\in I_\ell}|\beta^m_{\bm}|^2 \Bigg)^{1/2}, \]
where $\nu(\ell)=\lceil (\ell\log(1/\rho) - \log(\diam(\Gamma)))/\log 2\rceil$, which
is equivalent to the norm
\[\Bigg(|\beta_0 |^2 + \sum_{m=1}^{M-1}\sum_{\ell=0}^\infty \rho^{-2\ell t}\sum_{\bm\in I_\ell}|\beta^m_{\bm}|^2 \Bigg)^{1/2}.\]
If $\rho\in(0,1/2]$ then the function $\nu(\ell)$ is injective and $I_\ell=J_{\nu(\ell)}$, $\ell\in\N_{0}$.
\end{rem}

Theorem \ref{thm:Jonsson} has the following important corollary, which is obtained by duality. In this corollary and subsequently (see, e.g., \cite[Remark 3.8]{InterpolationCWHM}), given an interval $\mathcal{I}\subset \R$ we will say that a collection of Hilbert spaces $\{H_s:s\in \mathcal{I}\}$, indexed by $\mathcal{I}$, is an {\em interpolation scale} if, for all $s,t\in \mathcal{I}$ and $0<\eta<1$, $(H_s,H_t)$ is a compatible couple (in the standard sense, e.g.\ \cite[\S2.3]{BeLo}) and if the interpolation space
$(H_s,H_t)_{\eta}$\footnote{Here, and subsequently, $(H_s,H_t)_{\eta}$ denotes the standard complex interpolation space, $(H_s,H_t)_{[\eta]}$ in the notation of \cite{BeLo}; equivalently,  the $K$- or $J$-method real interpolation spaces denoted $(H_s,H_t)_{\eta,2}$ in \cite{BeLo,InterpolationCWHM}, which are the same Hilbert spaces $(H_s,H_t)_{[\eta]}$, with equal norms, if the $K$- and $J$-methods are appropriately normalised (see \cite[Remark 3.6]{InterpolationCWHM} and \cite{InterpolationE_CWHM}).} coincides with $H_\theta$, for $\theta = (1-\eta)s+\eta t$, with equivalent norms. We will say that $\{H_s:s\in \mathcal{I}\}$ is an {\em exact} interpolation scale if, moreover, the norms of $(H_s,H_t)_{\eta}$ and $H_\theta$ coincide, for all $s,t\in \mathcal{I}$ and $0<\eta<1$.

\begin{cor}
\label{cor:Wavelets}
Let $\Gamma$ and $t$ satisfy the assumptions of Theorem \ref{thm:Jonsson}.

\begin{enumerate}[(i)]

\item If $\{\beta_0\}\cup\{\beta^m_{\bm}\}_{\bm\in I_\ell,\,\ell\in\N_0,\, m\in\{1,\ldots,M -1\}}\subset\C$ satisfy
\begin{align}
\label{eq:BetaCond}
\Bigg(|\beta_0 |^2 + \sum_{m=1}^{M -1}\sum_{\nu=\nu_0}^\infty 2^{-2\nu t}\sum_{\bm\in J_\nu}|{\beta^m_{\bm}}|^2 \Bigg)^{1/2}<\infty
\end{align}
then
\begin{align}
\label{eq:HMinustRep}
f := \beta_0  \psi_0 + \sum_{m=1}^{M -1}\sum_{\ell=0}^\infty \sum_{\bm\in I_\ell}\beta^m_{\bm} \psi^m_{\bm},
\end{align}
converges in $\IH^{-t}(\Gamma)$.

\item
Each $f\in \IH^{-t}(\Gamma)$ can be written in the form \eqref{eq:HMinustRep} (with convergence in $\IH^{-t}(\Gamma)$), where %
\begin{align}
\label{eq:BetaDef2}
\beta_0  := \langle f,\psi_0\rangle_{\IH^{-t}(\Gamma)\times \IH^{t}(\Gamma)} \quad \text{and} \quad \beta^m_{\bm} := \langle f,\psi^m_{\bm}\rangle_{\IH^{-t}(\Gamma)\times \IH^{t}(\Gamma)}, \,\bm\in I_\ell,\,\ell\in\N_0,\, m\in\{1,\ldots,M -1\}
\end{align}
satisfy \rf{eq:BetaCond}. (By \eqref{eq:L2dualequiv} these definitions coincide with \eqref{eq:BetaDef1} when $f\in \IL_2(\Gamma)$.)

\item
The norms $\|\cdot\|_{\IH^{-t}(\Gamma)}$ and
\[ \|f\|_{-t} :=\Bigg(|\beta_0 |^2 + \sum_{m=1}^{M -1}\sum_{\nu=\nu_0}^\infty 2^{-2\nu t}\sum_{\bm\in J_\nu}|{\beta^m_{\bm}}|^2 \Bigg)^{1/2}, \qquad f\in\IH^{-t}(\Gamma),\]
are equivalent on $\IH^{-t}(\Gamma)$.

\item
The duality pairing $\langle \cdot,\cdot\rangle_{\IH^{-t}(\Gamma)\times \IH^{t}(\Gamma)}$ can be evaluated using the wavelet basis as%
\[ \langle f,g\rangle_{\IH^{-t}(\Gamma)\times \IH^{t}(\Gamma)}  = \beta_0\overline{\beta_0'} + \sum_{m=1}^{M -1}\sum_{\nu=\nu_0}^\infty \sum_{\bm\in J_\nu}{\beta^m_{\bm}}\overline{{\beta^m_{\bm}}'},\]
for $g=\beta_0'  \psi_0 + \sum_{m=1}^{M -1}\sum_{\nu=\nu_0}^\infty \sum_{\bm\in J_\nu}{\beta^m_{\bm}}' \psi^m_{\bm}\in \IH^{t}(\Gamma)$. With this pairing, $\IH^{-t}(\Gamma)$ provides a unitary realisation of $(\IH^{t}(\Gamma))^*$ with respect to the norms $\|\cdot\|_{t}$ and $\|\cdot \|_{-t}$.

\item Equipped with the norm $\|\cdot\|_\tau$, %
$\{\IH^{\tau}(\Gamma)\}_{-1<\tau<1}$ is an exact interpolation scale.%
\end{enumerate}
\end{cor}

\begin{proof}
For $t\in \R$ we can define the weighted $\ell_2$ sequence space
\[
\ih^t := \{ \boldsymbol{\beta} = \{\beta_0\}\cup\{\beta^m_{\bm}\}_{\bm\in I_\ell,\,\ell\in\N_0,\, m\in\{1,\ldots,M -1\}}:\|\boldsymbol{\beta}\|_{\ih^t}<\infty\},
\]
where
\[\|\boldsymbol{\beta}\|_{\ih^t} := \Bigg(|\beta_0 |^2 + \sum_{m=1}^{M -1}\sum_{\nu=\nu_0}^\infty 2^{2\nu t}\sum_{\bm\in J_\nu}|\beta^m_{\bm}|^2 \Bigg)^{1/2}, \quad \boldsymbol{\beta}\in \ih^t,\]
which is a Hilbert space with the obvious inner product. The dual space of $\ih^t$ can be unitarily realised as $\ih^{-t}$, with duality pairing
\[ \langle \boldsymbol{\beta},\boldsymbol{\beta}' \rangle_{\ih^{-t}\times \ih^t} := \beta_0\overline{\beta_0'} + \sum_{m=1}^{M -1}\sum_{\nu=\nu_0}^\infty \sum_{\bm\in J_\nu}{\beta^m_{\bm}}\overline{{\beta^m_{\bm}}'}, \quad \boldsymbol{\beta}\in \ih^{-t}, \boldsymbol{\beta}'\in \ih^t. \]
Furthermore, Theorem \ref{thm:Jonsson} implies that the space $\IH^t(\Gamma)$ is linearly and topologically isomorphic to $\ih^t$ for $0<t<1$ (unitarily if we equip $\IH^t(\Gamma)$ with $\|\cdot\|_t$).
Hence by duality $\IH^{-t}(\Gamma)$ is also linearly and topologically isomorphic to $\ih^{-t}$ for the same range of $t$ (unitarily if we equip $\IH^{-t}(\Gamma)$ with $\|\cdot\|_{-t}$). From these observations parts (i)-(iv) of the result follow, noting, in the case of (iv), that 
$f$ is a continuous antilinear functional on $\IH^t(\Gamma)$.  

For (v) we note that by
\cite[Thm.~3.1]{InterpolationCWHM} $\{\ih^\tau\}_{\tau\in \R}$ is an exact interpolation scale.
The corresponding statement about $\{\IH^{\tau}(\Gamma)\}_{-1<\tau<1}$, equipped with the norm $\|\cdot\|_\tau$, follows from Theorem \ref{thm:Jonsson} and parts (i)-(iii), combined with\footnote{There is an inaccuracy in the statement of \cite[Cor.~3.2]{InterpolationCWHM}; the map $\mathcal{A}:\Sigma(\overline{H})\to \mathcal{Y}$ in that corollary needs to be injective as well as linear for the corollary to hold (see \cite{InterpolationE_CWHM}). This injectivity follows automatically from the other conditions on $\mathcal{A}$  when (as in the application we make here) $H_1\subset H_0$, with continuous embedding, since then $\Sigma(\overline{H})=H_0$.}
\cite[Cor.~3.2]{InterpolationCWHM} and the fact that $\IH^0(\Gamma)=\IL_2(\Gamma)$ is unitarily isomorphic to $\ih^0$. Explicitly, in \cite[Cor.~3.2]{InterpolationCWHM}, given $-1<\tau_0<\tau_1<1$ we take $H_j=\IH^{\tau_j}(\Gamma)$, $j=0,1$, $\mathcal{X}=\{0\}\cup\{(\nu,\bm,m):\bm\in I_\ell,\,\ell\in\N_0,\, m\in\{1,\ldots,M -1\} \}$, $\mu$ to be the counting measure on $\mathcal{X}$, $\mathcal{A}$ to be the map taking $f\in\IH^{\tau_j}(\Gamma)$ to the sequence $\boldsymbol{\beta}$ defined by \rf{eq:BetaDef1} or \rf{eq:BetaDef2} (as appropriate), and $w_j(0)=1$, $w_j((\nu,\bm,m))=2^{2\nu \tau_j}$.
\end{proof}

A basic interpolation result is that  if $X_0\supset X_1$ and $Y_0\supset Y_1$ are Hilbert spaces, with $X_1$ and $Y_1$ continuously embedded in $X_0$ and $Y_0$, respectively, and $\mathcal{A}:X_j\to Y_j$ is a linear and topological isomorphism, for $j=0,1$, then, for $0<\theta<1$, $\mathcal{A}((X_0,X_1)_\theta) = (Y_0,Y_1)_\theta$ and $\mathcal{A}:(X_0,X_1)_\theta \to (Y_0,Y_1)_\theta$ is a linear and topological isomorphism. This is immediate since $((X_0,X_1)_\theta,(Y_0,Y_1)_\theta)$ and $((Y_0,Y_1)_\theta, (X_0,X_1)_\theta)$ are,  in the terminology of \cite[\S2]{InterpolationCWHM}, pairs of interpolation spaces relative to $(\overline{X},\overline{Y})$ and $(\overline{Y},\overline{X})$, respectively, where $\overline{X}=(X_0,X_1)$, $\overline{Y}=(Y_0,Y_1)$ (e.g., \cite[Theorem 4.1.2]{BeLo}).

\begin{rem}
\label{rem:RemarkInterp}
Theorem \ref{thm:Jonsson} and Corollary \ref{cor:Wavelets}, together with the above interpolation result applied with $\mathcal{A}$ taken as the identity operator, imply that, if $\Gamma$ satisfies the assumptions of Theorem \ref{thm:Jonsson},  then $\{\IH^{t}(\Gamma)\}_{-1<t<1}$ is an interpolation scale also when $\IH^t(\Gamma)$ is equipped with the original norm $\|\cdot\|_{\IH^t(\Gamma)}$ (or indeed any other equivalent norm).
\end{rem}

We include the following corollary, although we will not use it subsequently, because it may be of independent interest. %
Note that the range of $s$ does not extend to $s=-(n-d)/2$ since, by \cite[Thm~2.17]{HewMoi:15}, $H^{s}_\Gamma\neq \{0\}$ for $s<-(n-d)/2$, but $H^{-(n-d)/2}_\Gamma=\{0\}$ so that (e.g., \cite[Theorem 2.2(iv)]{InterpolationCWHM}) $\left(H_\Gamma^s,H_\Gamma^{-(n-d)/2}\right)_\theta = \{0\}$ for all $s\in \R$ and $0<\theta<1$. %
\begin{cor}
\label{cor:Interp}
Suppose that $\Gamma$ satisfies the assumptions of Theorem \ref{thm:Jonsson}. Then $\{H^{s}_\Gamma\}_{-(n-d)/2-1<s<-(n-d)/2}$ is an interpolation scale.
\end{cor}
\begin{proof} Apply the basic interpolation result above with $\mathcal{A}=\tr^*$ and $X_0=\IH^{-t}(\Gamma)$, $X_1=\IH^{-t'}(\Gamma)$, $Y_0=H_\Gamma^{-s}$, $Y_1=H_\Gamma^{-s'}$, for some $0<t'<t<1$,
where $s$ and $t$ are related by \eqref{eq:tr*} and similarly $t'=s'-(n-d)/2$; note that $\mathcal{A}:X_j\to Y_j$ is a linear and topological isomorphism, for $j=0,1$, by Theorem \ref{thm:Density}. This gives, for $0<\eta<1$, where $X_\eta:= (X_0,X_1)_\eta$ and $Y_\eta:=(Y_0,Y_1)_\eta$, that $\mathcal{A}(X_\eta)=Y_\eta$ and $\mathcal{A}:X_\eta\to Y_\eta$ is a linear and topological isomorphism. But, by Remark \ref{rem:RemarkInterp}, $X_\eta=\IH^{-t^*}(\Gamma)$, with equivalent norms, where $t^* := (1-\eta)t+\eta t'$,  so that $\mathcal{A}(X_\eta)=\tr^*(\IH^{-t^*}(\Gamma))=H_\Gamma^{-s^*}$, by Theorem \ref{thm:Density}, where $s^*:= t^*+(n-d)/2=(1-\eta)s+\eta s'$. Thus $Y_\eta = H_\Gamma^{-s^*}$; moreover, the norms on $Y_\eta$ and $H_\Gamma^{-s^*}$ are equivalent since the norms on $X_\eta$ and $\IH^{-t^*}(\Gamma)$ are equivalent and the mappings $\tr^*:X_\eta\to Y_\eta$ and $\tr^*:\IH^{-t^*}(\Gamma)\to H_\Gamma^{-s^*}$ (Theorem \ref{thm:Density}) are both linear and topological isomorphisms.
\end{proof}

\section{BVPs and BIEs}
\label{sec:BVPsBIEs}

In this section we state the BVP and BIE that we wish to solve. We consider time-harmonic acoustic scattering of an incident wave $u^i$ propagating in $\R^{n+1}$ ($n=1,2$) by a planar screen $\Gamma$, a subset of the hyperplane $\Gamma_\infty=\R^{n}\times\{0\}$.
We initially consider the case where $\Gamma$ is assumed simply to be non-empty and compact, for which a well-posed BVP/BIE formulation was  presented in \cite[\S3.2]{BEMfract}. We later specialise to the case where $\Gamma$ is a $d$-set for some $n-1<d \leq n$, and then further to the case where $\Gamma$ is a disjoint IFS attractor.

Our BVP, stated as Problem \ref{prob:BVP} below, is for the scattered field $u$, which is assumed to satisfy the Helmholtz equation %
\begin{align}
\label{eqn:HE}
\Delta u + k^2 u = 0, %
\end{align}
in $D:=\R^{n+1}\setminus \Gamma$,
for some wavenumber $k>0$,
and the Sommerfeld radiation condition
\begin{align}
\label{eqn:SRC}
\pdone{u(x)}{r} - \ri k u(x) = o(r^{-n/2}), \qquad r:=|x|\to\infty, \text{ uniformly in } \hat x:=x/|x|.
\end{align}
We assume that the incident wave $u^i$ is an element of $W^{1,{\rm loc}}(\R^{n+1})$ satisfying \rf{eqn:HE}  in some neighbourhood of $\Gamma$ (and hence $C^\infty$ in that neighbourhood by elliptic regularity, see, e.g., \cite[Thm~6.3.1.3]{Evans2010}); for instance, $u^i$ might be the plane wave $u^i(x)=\re^{\ri k \vartheta\cdot x}$ for some $\vartheta\in\R^{n+1}$, $|\vartheta|=1$.
To impose a Dirichlet (sound-soft) boundary condition on $\Gamma$
we stipulate that\footnote{The condition $\sigma(u+u^i)\in W_0^1(D)$, for every $\sigma\in C_{0,\Gamma}^\infty$, is equivalent to the (perhaps more familiar) requirement that $u+u^i\in W_0^{1,{\rm loc}}(D)$. Here $W_0^{1,{\rm loc}}(D)$ is the closure of $C_0^\infty(D)$ in $W^{1,{\rm loc}}(D)$ equipped with its usual topology, so that $W_0^{1,{\rm loc}}(D) = \{v\in W^{1,{\rm loc}}(D):\chi v\in W_0^1(D), \mbox{ for all }\chi\in C_0^\infty(\mathbb{R}^{n+1})\}$. One point of multiplying by the cut-off function $\sigma$ is that, for $u\in W^{1,{\rm loc}}(D)$, $\sigma u\in W^1(D)$, so that, taking traces, $\gamma^\pm (\sigma u)\in H^{1/2}(\Gamma_\infty)$, which is in the domain of the orthogonal projection operator $P$ defined in (29) which plays a key role subsequently.}
$\sigma(u+u^i)\in W^{1}_0(D)$, the closure of $C^\infty_0(D)$ in $W^1(D)$, for every $\sigma\in C^\infty_{0,\Gamma}$.
For the traces on $\Gamma_\infty$ this implies that $\gamma^\pm (\sigma(u+u^i)|_{U^\pm})\in \tH^{1/2}(\Gamma^c)$. (Here, and in what follows, $\Gamma^c$ will denote $\Gamma_\infty\setminus \Gamma$, the complement of $\Gamma$ in $\Gamma_\infty$, rather than its complement in $\R^{n+1}$, which we have denoted by $D$.)
This motivates the following problem statement, in which $P$ denotes the orthogonal projection
\begin{equation} \label{eq:Porth}
P:H^{1/2}(\Gamma_\infty)\to \tH^{1/2}(\Gamma^c)^\perp.
\end{equation}
Note that if $\sigma_1,\sigma_2\in C^\infty_{0,\Gamma}$ then $\sigma_1=\sigma_2$ on some open set $G\supset\Gamma$, so that, for $u\in
W^{1,\mathrm{loc}}(D)$, $\gamma^\pm((\sigma_1- \sigma_2)u|_{U^\pm})\in H^{1/2}_{G^c\cap \Gamma_\infty}\subset
\tilde{H}^{1/2}(\Gamma^c)$ so that $P\gamma^\pm((\sigma_1-\sigma_2)u|_{U^\pm})=0$. %

\begin{prob}%
\label{prob:BVP}
Let $\Gamma\subset\Gamma_\infty$ be non-empty and compact. Given $k>0$ and $g\in \tH^{1/2}(\Gamma^c)^\perp$, find $u\in C^2\left(D\right)\cap  W^{1,\mathrm{loc}}(D)$ satisfying
\rf{eqn:HE} in $D$, \rf{eqn:SRC},
and the boundary condition
\begin{align}\label{a1bc}
P\gamma^\pm(\sigma u|_{U^\pm})&=g,
\end{align}
for some (and hence every) $\sigma\in C^\infty_{0,\Gamma}$.
In the case of scattering of an incident wave $u^i$, $g$ is given specifically as %
\begin{align}
\label{eqn:gDefScatteringProblem}
g=-P \gamma^\pm (\sigma u^i|_{U^\pm}).
\end{align}
\end{prob}

The next result reformulates the BVP as a BIE. In this theorem $\cS:H^{-1/2}_{\Gamma}\to C^2(D)\cap W^{1,{\rm loc}}(\R^{n+1})$ denotes the (acoustic) single-layer potential operator, defined by (e.g., \cite[\S2.2]{ScreenPaper}) %
\begin{equation} \label{eq:SLPdef}
\cS\psi(x) := \langle\gamma^\pm (\sigma\Phi(x,\cdot)|_{U^\pm}), \overline{\psi}\rangle_{H^{1/2}(\Gamma_\infty)\times H^{-1/2}(\Gamma_\infty)}, \quad x\in D,
\end{equation}
where $\overline{\psi}$ denotes the complex conjugate of $\psi$,\footnote{If $\psi\in L_2(\Gamma_\infty)\subset H^{-1/2}(\Gamma_\infty)$, $\bar\psi$ is the usual complex conjugate, and this definition of the complex conjugate is extended to $H^{-1/2}(\Gamma_\infty)$ by density.} $\Phi(\bx,\by):=\re^{\ri k |\bx-\by|}/(4\pi |\bx-\by|)$ ($n=2$),  $\Phi(\bx,\by):=(\ri/4)H^{(1)}_0(k|\bx-\by|)$ ($n=1$), $H_0^{(1)}$ is the Hankel function of the first kind of order zero (e.g., \cite[Equation (9.1.3)]{AbramowitzStegun}),
and $\sigma$ is any element of $C^\infty_{0,\Gamma}$ with $x\not\in \supp{\sigma}$. The $\pm$ in \rf{eq:SLPdef} indicates that either trace can be taken, with the same result.
In the case when $\Gamma$ is the closure of a Lipschitz open subset of $\Gamma_\infty$ and $\psi\in L_2(\Gamma)$ the potential can be expressed as an integral with respect to (Lebesgue) surface measure, namely
 \begin{align}
 \label{eq:SLPrep}
\cS\psi(\bx)=\int_{\Gamma}\Phi(\bx,\by)\psi(\by)\,\rd s(\by), \qquad \bx\in D.
\end{align}
The operator $S:H^{-1/2}_{\Gamma}\to \tH^{1/2}(\Gamma^c)^\perp$ denotes the single-layer boundary integral operator
\begin{equation} \label{eq:Sdef}
S\psi:=  P\gamma^\pm(\sigma\cS\psi|_{U^\pm}), \qquad \psi\in H^{-1/2}_{\Gamma},
\end{equation}
where $\sigma \in C^\infty_{0,\Gamma}$ is arbitrary,
which is continuous and coercive (see Lemma \ref{lem:coer} below).

\begin{thm}[{\cite[Thm.~3.29 and Thm.~6.4]{ScreenPaper}}]
\label{thm:Closed}
Let $\Gamma\subset\Gamma_\infty$ be non-empty and compact.
Then Problem \ref{prob:BVP}
has a unique solution %
satisfying
the representation formula
\begin{align}
\label{eqn:Rep}
u(x )= -\cS\phi(x), \qquad x\in D,
\end{align}
where $\phi = \dn^+(\sigma u|_{U^+})-\dn^-(\sigma u|_{U^-}) \in H^{-1/2}_{\Gamma}$ (with $\sigma \in C^\infty_{0,\Gamma}$ arbitrary) is the unique solution of the BIE
\begin{equation}
\label{eqn:BIE}
S\phi = -g,
\end{equation}
with $g$ given by \eqref{eqn:gDefScatteringProblem} in the case of scattering of an incident wave $u^i$.
\end{thm}

Define the sesquilinear form $a(\cdot,\cdot)$  on $H^{-1/2}_\Gamma\times  H^{-1/2}_\Gamma$ by
\begin{align}
\label{eqn:Sesqui}
a(\phi,\psi):=\langle S\phi,\psi\rangle_{\tH^{1/2}(\Gamma^c)^\perp \times H^{-1/2}_\Gamma}=\langle S\phi,\psi\rangle_{H^{1/2}(\Gamma_\infty)\times H^{-1/2}(\Gamma_\infty)},\qquad \phi, \psi\in H^{-1/2}_{\Gamma}.
\end{align}
Then the BIE \rf{eqn:BIE} can be written equivalently in variational form as: given $g\in (\tH^{1/2}(\Gamma^c))^\perp$, find $\phi\in H^{-1/2}_{\Gamma}$ such that
\begin{align}
\label{eqn:VariationalCts}
a(\phi,\psi)=-\langle g,\psi\rangle_{H^{1/2}(\Gamma_\infty)\times H^{-1/2}(\Gamma_\infty)}, \qquad \mbox{ for all } \psi\in H^{-1/2}_{\Gamma}.
\end{align}
This equation will be the starting point for our Galerkin discretisation in \S\ref{sec:HausdorffBEM}.

The definition, domain and codomain of $S$ may seem exotic.
But, as noted in \cite[\S3.3]{BEMfract}, given any bounded Lipschitz open set $\Omega\subset \Gamma_\infty$ containing $\Gamma$, the sesquilinear form $a(\cdot,\cdot)$ is nothing but the restriction to $H^{-1/2}_\Gamma\times  H^{-1/2}_\Gamma$ of the sesquilinear form $a^\Omega(\cdot,\cdot)$ defined on $\tH^{-1/2}(\Omega)\times \tH^{-1/2}(\Omega)$ by
\begin{align} \label{eqn:Sesquir}
a^\Omega(\phi,\psi):=\langle S^\Omega\phi,\psi\rangle_{H^{1/2}(\Omega) \times \tH^{-1/2}(\Omega)},\qquad \phi, \psi\in \tH^{-1/2}(\Omega),
\end{align}
with $S^\Omega:\tH^{-1/2}(\Omega)\to H^{1/2}(\Omega)$ %
the single-layer boundary integral operator on the Lipschitz screen $\Omega$, defined in the standard way (e.g., \cite[\S2.3]{CoercScreen2}), so that
\begin{equation} \label{eq:SOmega}
S^\Omega\phi(\bx)=\int_{\Omega}\Phi(\bx,\by)\phi(\by)\,\rd s(\by), \qquad \bx\in \Omega,
\end{equation}
for $\phi\in L_2(\Omega)$.  %
The continuity and coercivity of $a^\Omega(\cdot,\cdot)$ (e.g., \cite{CoercScreen2}) therefore implies the continuity and coercivity of $a(\cdot,\cdot)$, as the following lemma (\cite{CoercScreen2}, \cite[\S2.2]{ScreenPaper}) states.

\begin{lem} \label{lem:coer}
The sesquilinear form $a(\cdot,\cdot)$ is continuous and coercive on $H^{-1/2}_\Gamma\times  H^{-1/2}_\Gamma$, specifically, for some constants $C_a, \alpha>0$ (the {\em continuity} and {\em coercivity} constants) depending only on $k$ and $\diam(\Gamma)$,
\begin{equation} \label{eq:ContCoer}
|a(\phi,\psi)| \leq C_a\|\phi\|_{H^{-1/2}_\Gamma}\, \|\psi\|_{H^{-1/2}_\Gamma}, \quad |a(\phi,\phi)|\geq \alpha \|\phi\|_{H^{-1/2}_\Gamma}^2, \quad \phi,\psi\in H^{-1/2}_\Gamma.
\end{equation}
\end{lem}

Having computed $\phi$ by solving a Galerkin discretisation of \eqref{eqn:VariationalCts}, we will also evaluate $u(x)$ at points $x\in D$ using \eqref{eqn:Rep} and \eqref{eq:SLPdef}. Further, recall that (e.g., \cite[Eqn.~(2.23)]{ChGrLaSp:11}, \cite[p.~294]{McLean})
$$
u(x) = \frac{\re^{\ri k|x|}}{|x|^{n/2}}\left(u^\infty(\hat x)+ O(|x|^{-1})\right), \quad \mbox{as} \quad |x|\to \infty,
$$
uniformly in $\hat x:= x/|x|$, where $u^\infty\in C^\infty(\mathds{S}^n)$ is the so-called {\em far-field pattern} of $u$ and $\mathds{S}^n$ is the unit sphere in $\R^{n+1}$. We will also compute this far-field pattern, given explicitly (\cite[Eqn.~(2.23)]{ChGrLaSp:11}, \cite[p.~294]{McLean}) as
\begin{equation} \label{eq:ffpattern}
u^\infty(\hat x) = - \langle \gamma^\pm (\sigma \Phi^\infty(\hat x,\cdot)), \overline{\phi}\rangle_{H^{1/2}(\Gamma_\infty)\times H^{-1/2}(\Gamma_\infty)}, \quad \hat x\in \mathds{S}^n,
\end{equation}
where $\sigma$ is any element of $C^\infty_{0,\Gamma}$ and
\begin{equation} \label{eq:FFps}
\Phi^\infty(\hat x,y):= \frac{\ri k^{(n-2)/2}}{2(2\pi \ri)^{n/2}}\, \exp(-\ri k\hat x\cdot y), \quad \hat x\in \mathds{S}^n, \; y\in \R^{n+1}.
\end{equation}
Note that $\Phi^\infty(\cdot,y)$ is the  far-field pattern of $\Phi(\cdot,y)$, for $y\in \R^{n+1}$.

The following lemma provides conditions under which a compact screen $\Gamma\subset\Gamma_\infty$ produces a non-zero scattered field. We note that a sufficient condition for $H^{-1/2}_\Gamma \neq \{0\}$ is that $\dimH(\Gamma)>n-1$ \cite[Thm~2.12]{HewMoi:15}, and that when $\Gamma$ is a $d$-set this is also a necessary condition \cite[Thm~2.17]{HewMoi:15}.
\begin{lem}[{\cite[Thm~4.6]{ScreenPaper}}]
\label{lem:Nullity}
Suppose that $u^i$, which is $C^\infty$ in a neighbourhood of $\Gamma$, is non-zero on $\Gamma$. Then the solution of Problem \ref{prob:BVP} with $g$ given by \eqref{eqn:gDefScatteringProblem} is zero if and only if $H^{-1/2}_\Gamma = \{0\}$.
\end{lem}

\subsection{The BIE on \texorpdfstring{$d$}{d}-sets in trace spaces}
\label{sec:BIEondsets}

Suppose now that
$\Gamma$ is a compact $d$-set with $n-1<d\leq n$.
The assumption that $d>n-1$ ensures that $\Gamma$ produces a non-zero scattered field under the conditions of Lemma \ref{lem:Nullity}.
Furthermore, the condition $d>n-1$ is equivalent to %
\eqref{eq:st} with $s=1/2$, %
so that the results in \S\ref{sec:FunctionSpaces} apply with $s=1/2$ and
\begin{align}
\label{eq:tdDef}
t=t_d:=\frac{1}{2}-\frac{n-d}{2} \in\Big(0,\frac12\Big].
\end{align}
It then follows from \rf{eq:SLPdef} and \rf{eq:L2dualrep} that for $\Psi\in\IL_2(\Gamma)$ the potential $\cS$ has the following integral representation with respect to Hausdorff measure:
\begin{align}
\label{eq:SLPrepHausdorff}
\cS\tr^*\Psi(\bx)=\int_{\Gamma}\Phi(\bx,\by)\Psi(\by)\,\rd \cH^d(\by), \qquad \bx\in D.
\end{align}

\begin{prop} \label{prop:LPcont} For every $\Psi\in \IL_\infty(\Gamma)$ it holds that $\cS\tr^*\Psi\in C(\R^{n+1})$. Precisely, the function
$$
F(x):=\int_{\Gamma}\Phi(\bx,\by)\Psi(\by)\,\rd \cH^d(\by),
$$
is well-defined for all $x\in \R^{n+1}$ and is continuous, and $\cS\tr^*\Psi(x) = F(x)$ for $x\in D$. %
\end{prop}
\begin{proof} It is clear that $F(x)$ is well-defined for $x\in D$ since the integrand is then in $\IL_\infty(\Gamma)$, as $\Phi(x,y)$ is continuous for $x\neq y$, and $\cH^d(\Gamma)<\infty$ as $\Gamma$ is a bounded $d$-set.  For some constant $C>0$,
$|\Phi(x,y)|\leq Cf(|x-y|)$, for $x,y\in \R^{n+1}$, $x\neq y$, where $f:(0,\infty)\to(0,\infty)$ is decreasing and continuous, given explicitly by  (see \cite[Equations (9.1.3), (9.1.12), (9.1.13), (9.2.3)]{AbramowitzStegun})
\begin{alignat*}{2}
f(r)&:=
\left\lbrace \begin{array}{ll}
1+|\log(r)|, & 0<r<1,\\
r^{-1/2}, & r>1,
\end{array}
\right.
\end{alignat*}
when $n=1$, by $f(r):=r^{-1}$, $r>0$, when $n=2$. Thus $F(x)$ is also well-defined for $x\in \Gamma$ by Corollary \ref{lem:1dIntBd} and since $\Psi\in \IL_\infty(\Gamma)$.
To see that $F$ is continuous, for $\varepsilon>0$ let
$$
\Phi_\varepsilon(x,y) := \left\{\begin{array}{cc}
                                  \Phi(x,y), & |x-y|>\varepsilon, \\
                                  \Phi(0,\varepsilon \hat e), & |x-y|\leq \varepsilon,
                                \end{array}\right.
$$
where $\hat e\in \R^{n+1}$ is any unit vector,
and let $F_\varepsilon(x):= \int_\Gamma \Phi_\varepsilon(x,y)\Psi(y) \, \rd \cH^d(y)$, for $x\in \R^{n+1}$. Then, for every $\varepsilon>0$, $\Phi_\varepsilon \in C(\R^{n+1}\times \R^{n+1})$, so that $F_\varepsilon\in C(\R^{n+1})$. Further, for $x\in \R^{n+1}$, noting Remark \ref{rem:AntonioModified}, we have that
\begin{equation} \label{eq:Feps bound}
|F(x)-F_\varepsilon(x)| \leq 2C \|\Psi\|_{\IL_\infty(\Gamma)} \int_{\Gamma\cap \overline B_\varepsilon(x)}f(|x-y|)\, \rd \cH^d(y) \leq 2C C_2\|\Psi\|_{\IL_\infty(\Gamma)}\int_0^{2\varepsilon} r^{d-1}f(r) \rd r \to 0
\end{equation}
as $\varepsilon\to 0$, uniformly in $x\in \R^{n+1}$, so that also $F\in C(\R^{n+1})$. Thus also $\cS\tr^* \Psi\in W^1_{\mathrm{loc}}(\R^{n+1})$ is continuous, in the (usual) sense that it is equal almost everywhere with respect to $n+1$-dimensional Lebesgue measure to a continuous function (the function $F$), by \eqref{eq:SLPrepHausdorff}.
\end{proof}

Noting that $\tr:\tH^{1/2}(\Gamma^c)^\perp \to \IH^{t_d}(\Gamma)$ and $\tr^*:\IH^{-t_d}(\Gamma)\to H^{-1/2}_\Gamma$ are unitary isomorphisms (see Theorem~\ref{thm:Density}), if we define $\IS:\IH^{-t_d}\GG\to\IH^{t_d}\GG$ by
\begin{align}
\label{eqn:ISDef}
\IS:=\tr\, S\, \tr^* %
\end{align}
then $\IS$ is continuous and coercive, with the same associated constants as $S$ (the constants in Lemma \ref{lem:coer}).
Furthermore, it follows from \eqref{eqn:Sesqui} and \eqref{eq:tracedual} that
\begin{equation} \label{eqn:SesquiS}
a(\tr^*\bbx,\tr^*\bby) = \langle \IS \bbx,\bby\rangle_{\IH^{t_d}\GG\times\IH^{-t_d}\GG}, \quad \bbx,\bby\in \IH^{-t_d}(\Gamma),
\end{equation}
and
the variational problem \rf{eqn:VariationalCts} can be equivalently stated as: given $g\in \tH^{1/2}(\Gamma^c)^\perp$, find $\bbx\in \IH^{-t_d}(\Gamma)$ such that
\begin{equation} \label{eq:varformnew}
\langle \IS \bbx,\bby\rangle_{\IH^{t_d}\GG\times\IH^{-t_d}\GG}=-\langle \tr g,\bby\rangle_{\IH^{t_d}\GG\times \IH^{-t_d}\GG},
\qquad \mbox{ for all } \bby\in \IH^{-t_d}\GG,
\end{equation}
and the solutions of \eqref{eq:varformnew} and \rf{eqn:VariationalCts} are related through $\phi = \tr^*\Psi$.
A schematic showing the relationships between the relevant function spaces and operators is given in Figure \ref{fig:AndreaScheme}.

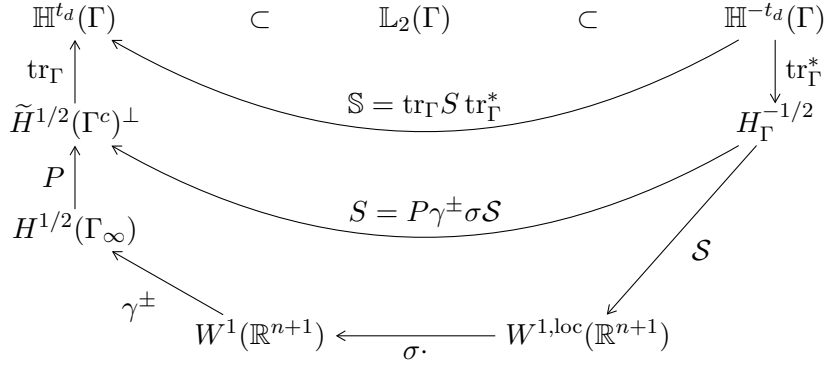
\begin{figure}[t]
\centering
\begin{tikzpicture}
\matrix[matrix of math nodes,
column sep={13pt},
row sep={40pt,between origins},
text height=1.5ex, text depth=0.25ex] (s)
{
|[name=Ht]| \IH^{t_d}\GG &
|[name=]| \subset &
|[name=L2]| \IL_2\GG &
|[name=]| \subset &
|[name=Hmt]| \IH^{-t_d}\GG
\\
|[name=Hscp]| \tH^\half(\Gamma^c)^\perp&
& & &
|[name=HmsG]| H^\mhalf_\Gamma&
\\
|[name=Hs]| H^\half(\Gamma_\infty) &&&&
\\
&
|[name=W]| W^{1}(\R^{n+1})
&&
|[name=Wloc]| W^{1,{\rm loc}}(\R^{n+1})
\\
};
\draw[->,>=angle 60] %
(Hscp) edge node[auto]{\(\tr\)} (Ht)
(Hmt) edge node[auto]{\(\tr^*\)} (HmsG)
(HmsG) edge node[auto]{\(\cS\)} (Wloc)
(Wloc) edge node[auto]{\(\sigma\cdot\)} (W)
(W) edge node[auto]{\(\gamma^\pm\)} (Hs)
(Hs) edge node[auto]{\(P\)} (Hscp)
(HmsG) edge[bend left=30] node[auto,swap]{\( S = P\gamma^\pm\sigma\cS\)} (Hscp)
(Hmt) edge[bend left=30] node[auto,swap]{\( \IS = \tr S\, \tr^*\)} (Ht)
;
\end{tikzpicture}
\caption{Schema of relevant function spaces and operators for $s=1/2$, $t=t_d:=\half-\frac{n-d}2$.}
\label{fig:AndreaScheme}
\end{figure}

Since, %
as an operator on $H^{1/2}(\Gamma_\infty)$, $\ker(\tr)= \tH^{1/2}(\Gamma^c)$ (Theorem \ref{thm:Density}), so that $\tr P\phi=\tr \phi$, $\phi\in H^{1/2}(\Gamma_\infty)$, and recalling \eqref{eq:tr*}, \eqref{eq:Porth} and \eqref{eq:Sdef}, we see that
(with the $\pm$ again indicating that either trace can be taken, with the same result)
\begin{equation} \label{eq:Srep}
\IS\bbx = \tr \gamma^\pm((\sigma\cS\tr^*\bbx)|_{U^\pm}), \quad \bbx \in \IH^{-t_d}(\Gamma),
\end{equation}
with $\sigma \in C^\infty_{0,\Gamma}$ arbitrary.

The following integral representation for $\IS$ will be crucial for our Hausdorff BEM in \S\ref{sec:HausdorffBEM}.%
\begin{thm}\label{lem:ISasHauss}
Let $\Gamma$ be a compact $d$-set with $n-1<d\leq n$. For $\bbx$ in $\IL_\infty(\Gamma)$, %
\begin{equation} \label{eq:ISasHauss}
\IS\bbx(x) = \int_{\Gamma} \Phi(x,y) \bbx(y) \rd \cH^d(y),
\qquad \text{ for }\cH^d\text{-a.e.\ }x\in\Gamma.
\end{equation}
\end{thm}

\begin{proof}
Let $\Psi\in \IL_\infty(\Gamma)\subset \IL_2(\Gamma)\subset \IH^{-t_d}(\Gamma)$, so that $\IS \Psi\in \IH^{t_d}(\Gamma)\subset \IL_2(\Gamma)$.  For arbitrary $\sigma \in C^\infty_{0,\Gamma}$, we have that
\begin{align}
\label{eqn:StrictRep1}
\IS\Psi = \tr\left(\gamma^\pm\left((\sigma f)|_{U^\pm}\right)\right),
\end{align}
by \eqref{eq:Srep}, where $f:=\cS\tr^*\Psi \in C^2(D)\cap W^{1,{\rm loc}}(\R^{n+1})$. Now, where $F$ is defined as in Proposition \ref{prop:LPcont}, $f(x)=F(x)$ for almost all $x\in \R^{n+1}$ (for $x\in D$ by \eqref{eq:SLPrepHausdorff}), and $F\in C(\R^{n+1})$ by Proposition \ref{prop:LPcont}. Further, if $G\in W^1(\R^{n+1})\cap C(\R^{n+1})$ it is easy to see that  $\tr\gamma^\pm(G|_{U^\pm})= G|_\Gamma$. Thus $\IS \Psi = F|_\Gamma$ in $\IL_2(\Gamma)$, and the result follows.
\end{proof}

The definition and mapping properties of $\tr$ and $\tr^*$, noted in \S\ref{sec:FunctionSpaces}, combined with the representation \eqref{eq:Srep}, enable us to extend the domain of $\IS$ to $\IH^{-t}(\Gamma)$, for $t_d<t<2t_d$, or restrict it to $\IH^{-t}(\Gamma)$, for $0<t<t_d$, as stated in %
the following key result.%

\begin{prop}
\label{lem:cont}
Let $\Gamma$ be a compact $d$-set with $n-1<d \leq n$. For $|t|<t_d$, $\IS: \IH^{t-t_d}(\Gamma)\to \IH^{t+t_d}(\Gamma)$ and is continuous. %

\end{prop}
\begin{proof} Let $\Omega\subset \Gamma_\infty$ be any bounded open set containing $\Gamma$, so that (see \S\ref{sec:FunctionSpaces}) $H^s_\Gamma$ is a closed subspace of $\widetilde H^s(\Omega)$ for every $s\in \R$. The claimed mapping property of $\IS$ follows from \rf{eq:Srep} since $\tr:H^s(\R^n)\to \IH^t(\Gamma)$ and $\tr^*:\IH^{-t}(\Gamma)\to H^{-s}_\Gamma\subset \widetilde H^{-s}(\Omega)$ are continuous for $s>(n-d)/2$ (i.e.\ $t>0$), with $s$ and $t$ related by \eqref{eq:st}, and since
the mapping $\phi\mapsto\gamma^\pm((\sigma \cS\phi)|_{U^\pm}):\widetilde H^s(\Omega)\to H^{s+1}(\Gamma_\infty)$ is continuous for $s\in \R$ (e.g., \cite[Thm 1.6]{CoercScreen2}). %
\end{proof}

We now make a conjecture concerning the mapping properties of $\IS^{-1}$. To the best of our knowledge there are no results in this direction in the case $d<n$, but the conjecture can be seen as an extension of known results in the case $d=n$,  since the conjecture is known to be true in the case that $\Gamma = \overline{\Omega}$ for some bounded Lipschitz domain $\Omega\subset\Gamma_\infty$, in which case $d=n$ (so $t_d=1/2$).
For in this case the single layer BIO $S^{\Omega}$, defined below \eqref{eqn:Sesquir}, is invertible as an operator from $\tH^{t-1/2}(\Omega)$ to $H^{t+1/2}(\Omega)$ for $|t|<1/2$
--- see \cite[Thm~1.8]{StWe84} for the case $n=1$ and \cite[Theorem 4.1]{Schneider91}\footnote{\label{footnote:Si} Temporarily denoting $S^\Omega$ by $S_k$ to indicate the dependence on $k$, \cite[Theorem 4.1]{Schneider91} gives that $S_k:\tH^{t-1/2}(\Omega)\to H^{t+1/2}(\Omega)$  is Fredholm of index zero for $|t|<1/2$ when the wavenumber is purely imaginary, say $k=\ri$. (We are defining $S_\ri$ here by \eqref{eq:SOmega} with $k=\ri$; equivalently, arguing, e.g., as in \cite[\S3]{CoercScreen2},  $S_\ri$ is the pseudodifferential operator on $\Omega$ with symbol $\sigma(x,\xi)= 1/(2\sqrt{1+|\xi|^2})$, in the notation of \cite[Theorem 4.1]{Schneider91}.)  Since  $S_\ri-S_k$ is compact as an operator from $\tH^0(\Gamma)=L^2(\Gamma)\to H^1(\Gamma)$ and by duality also from $\tH^{-1}(\Gamma)\to L^2(\Gamma)=H^0(\Gamma)$ (see, e.g., the argument on \cite[p.~122]{ChGrLaSp:11}) and so, by interpolation (see, e.g., \cite[Thm.~10]{CwKa95} and \cite[Cor.~4.7, 4.10]{InterpolationCWHM}) from $\tH^{t-1/2}(\Gamma)\to H^{t+1/2}(\Gamma)$ for $|t|\leq 1/2$, it follows that $S_k=S_\ri + (S_k-S_\ri)$ is Fredholm of index zero as a mapping from $\tH^{t-1/2}(\Gamma)\to H^{t+1/2}(\Gamma)$ for $|t|< 1/2$.
Since $S_k:\tH^{t-1/2}(\Omega)\to H^{t+1/2}(\Omega)$ is coercive and so invertible by Lax-Milgram for $t=0$, $S_k$ is injective and so invertible for $0<t<1/2$; indeed, invertible also for $-1/2<t<0$, since $S_k^*$, the adjoint of $S_k$, is invertible for this range and $S_k^*=JS_kJ$, where $J$ maps $\phi$ to its complex conjugate.}
 for the case $n=2$
--- which implies 
that $\IS: \IH^{t-t_d}(\Gamma)\to \IH^{t+t_d}(\Gamma)$ is invertible, 
by 
the fact that $\IH^s(\Gamma)=H^s(\Gamma)$ for $s>0$ (see \cite{dequalsnpaper}), 
Theorem \ref{thm:Density}, and the fact that $\tH^s(\Omega) = H^s_{\overline\Omega}$ (see, e.g.,\ \cite[Thm 3.29]{McLean}). 
A further motivation for our conjecture is that its  truth  implies convergence rates for our BEM (defined and analysed in \S\ref{sec:HausdorffBEM}) that are evidenced by numerical experiments in \S\ref{sec:NumericalResults} for cases with $n-1<d<n$.

\begin{conj}
\label{ass:Smoothness}
If $\Gamma$ is a compact $d$-set with $n-1<d\leq n$ and $|t|<t_d$, then $\IS: \IH^{t-t_d}(\Gamma)\to \IH^{t+t_d}(\Gamma)$ is invertible, and hence (by Proposition \ref{lem:cont} and the bounded inverse theorem) a linear and topological isomorphism.
\end{conj}

While we are not able to prove Conjecture \ref{ass:Smoothness} in its full generality, in the case where $\Gamma$ is a disjoint IFS attractor, in which case, by Lemma \ref{lem:disconnected}, 
$d<n$,
we can prove that $\IS$ is invertible for a range of $t$,
using Corollary \ref{cor:Wavelets} and results from function space interpolation theory.
\begin{prop}
\label{prop:epsilon}
Let $\Gamma$ be a disjoint IFS attractor with $n-1<d=\dimH(\Gamma)<n$.
Then there exists $0<\epsilon\leq t_d$ such that $\IS: \IH^{t-t_d}(\Gamma)\to \IH^{t+t_d}(\Gamma)$ is invertible for $|t|< \epsilon$. %
\end{prop}
\begin{proof}
We recall the following result from \cite[Prop.~4.7]{mitrea1999boundary}, which quotes \cite{vsneiberg1974spectral}.
Suppose $E_j$ and $F_j$ are Banach spaces, $E_1\subset E_0$ and $F_1\subset F_0$ with continuous embeddings, and $T:E_j\to F_j$ is a bounded linear operator
for $j=0,1$.
Let $E_\theta=(E_0 , E_1)_\theta$ be defined by complex interpolation, and similarly $F_\theta$, so $T: E_\theta\to F_\theta$ and is
bounded, for $0<\theta<1$.
Assume that $T: E_{\theta_0}\to F_{\theta_0}$ is invertible for some $\theta_0\in (0,1)$.
Then $T: E_\theta\to F_\theta$ is invertible for $\theta$ in a neighbourhood of $\theta_0$.

The claimed invertibility of $\IS: \IH^{t-t_d}(\Gamma)\to \IH^{t+t_d}(\Gamma)$ in a neighbourhood of $t=0$ then follows by taking $0<\tau<t_d$ and applying the above result with
\[E_0=\IH^{-\tau-t_d}(\Gamma),\quad E_1 = \IH^{\tau-t_d}(\Gamma),\quad
F_0 = \IH^{-\tau+t_d}(\Gamma),\quad F_1 = \IH^{\tau+t_d}(\Gamma),$$
$$\text{and }\theta_0=1/2, \text{ so that }E_{\theta_0} = \IH^{-t_d}(\Gamma)\text{ and }F_{\theta_0} = \IH^{t_d}(\Gamma),\]
recalling that (i) $\IS: \IH^{t-t_d}(\Gamma)\to \IH^{t+t_d}(\Gamma)$ is bounded for $|t|< t_d$ (Proposition \ref{lem:cont}); (ii) $\IS: \IH^{-t_d}(\Gamma)\to \IH^{t_d}(\Gamma)$ is invertible (Lemma \ref{lem:coer}); (iii) $\{\IH^t(\Gamma)\}_{|t|<1}$ is an interpolation scale (Corollary \ref{cor:Wavelets}); and (iv) $t_d\in(0,1/2)$, so $t-t_d\in(-2t_d,0)\subset(-1,0)$ and $t+t_d\in(0,2t_d)\subset(0,1)$, for $|t|<t_d$.
\end{proof}

\begin{rem}[Solution regularity in the $H_\Gamma^s$ scale] \label{rem:regularity}
The mapping properties of $\tr^*$ in Theorem~\ref{thm:Density} and
the relationship \eqref{eqn:ISDef} %
between $S$ and $\IS$ imply that, if $\phi\in H_\Gamma^{-1/2}$ is the solution of the BIE \rf{eqn:BIE}, then $\IS\Psi = -\tr g$, where $\Psi := (\tr^*)^{-1}\phi\in \IH^{-t_d}(\Gamma)$. Thus
if, for some $0<t<t_d$, $\IS: \IH^{t-t_d}(\Gamma)\to \IH^{t+t_d}(\Gamma)$ is invertible and $\tr g\in \IH^{t+t_d}(\Gamma)$, then, again using the mapping properties of $\tr^*$ from Theorem~\ref{thm:Density},  the solution $\phi=\tr^*\Psi$ of the BIE \rf{eqn:BIE} satisfies
\begin{align}
\label{eq:phiReg}
\phi\in H^{-1/2+t}_\Gamma, \quad \mbox{with} \quad \|\phi\|_{H^{-1/2+t}_\Gamma} \leq C\|\tr g\|_{\IH^{t+t_d}(\Gamma)}
\end{align}
for some constant $C>0$ independent of $\phi$ and $g$.
In the case of scattering of an incident wave $u^i$, in which $g$ is given by \eqref{eqn:gDefScatteringProblem}, we have that $\tr g$ $=-\tr \gamma^\pm((\sigma u^i)|_{U^\pm})$ $\in \IH^{t+t_d}(\Gamma)$ for all $0<t<t_d$, since $u^i$ is $C^\infty$ in a neighbourhood of $\Gamma$.
Hence, in this case, if $\Gamma$ is a disjoint IFS attractor with $n-1<d=\dimH(\Gamma)<n$, then \eqref{eq:phiReg} holds for $0< t<\epsilon$, where $\epsilon$ is as in Proposition \ref{prop:epsilon}, and, if Conjecture \ref{ass:Smoothness} holds, then \eqref{eq:phiReg} holds for $0< t<t_d$ whenever $\Gamma$ is a compact $d$-set with $n-1<d\leq n$.
\end{rem}

\section{The Hausdorff BEM}
\label{sec:HausdorffBEM}
We now define and analyse our Hausdorff BEM. %
To begin with, we assume simply that $\Gamma$ is a compact $d$-set for some $n-1<d\leq n$.
Given $N\in \N$ let $\{T_j\}_{j=1}^N$ be a ``mesh'' of $\Gamma$, by which we mean a collection of $\cH^d$-measurable subsets of $\Gamma$ (the ``elements'') such that $\cH^d(T_j)>0$ for each $j=1,\ldots,N$, $\cH^d(T_j\cap T_{j'})=0$ for $j\neq j'$, and
\[\Gamma = \bigcup_{j=1}^N T_j.\]
Define the $N$-dimensional space of piecewise constants
\begin{equation}\label{eq:IVnDef}
\IV_N:=\{f\in\IL_2(\Gamma):f|_{T_j}=c_j \text{ for some }c_j\in \C, \, j=1,\ldots,N\}\subset\IL_2(\Gamma)\end{equation}
and set
\begin{align}
\label{eq:VnDef}
V_N := \tr^*(\IV_N)\subset H^{-1/2}_\Gamma.
\end{align}

Our proposed BEM for solving the BIE \rf{eqn:BIE}, written in variational form as \eqref{eqn:VariationalCts}, uses $V_N$ as the approximation space in a Galerkin method.
Given $g\in (\tH^{1/2}(\Gamma^c))^\perp$ we seek $\phi_N\in V_N$ such that (with $a$ defined by \rf{eqn:Sesqui})
\begin{align}
\label{eqn:Variational}
a(\phi_N,\psi_N)
=-\langle g,\psi_N\rangle_{H^{1/2}(\Gamma_\infty)\times H^{-1/2}(\Gamma_\infty)}, \qquad \mbox{ for all } \psi_N\in V_N.
\end{align}

Let $\{f^i\}_{i=1}^N$ be a basis for $\IV_N$, and let $\{e^i=\tr^*f^i\}_{i=1}^N$ be the corresponding basis for $V_N$. Then, writing $\phi_N=\sum_{j=1}^N c_j e^j$, \eqref{eqn:Variational} implies that the coefficient vector $\vec{c}=(c_1,\ldots,c_N)^T\in\C^N$ satisfies the system
\begin{equation} \label{eq:cvec}
A \vec{c} = \vec{b},
\end{equation}
where, by \eqref{eqn:SesquiS}, \eqref{eq:L2dualequiv}, and \eqref{eq:ISasHauss}, the matrix $A\in \C^{N\times N}$ has $(i,j)$-entry given by %
\begin{align}
A_{ij}=a(e^j,e^i) = a\left(\tr^*f^j,\tr^* f^i\right)
& = \langle \IS f^j,f^i \rangle_{\IH^{t_d}(\Gamma)\times \IH^{-t_d}(\Gamma)} \notag\\
& = (\IS f^j,f^i)_{\IL_2(\Gamma)},
\notag\\
\label{eq:Galerkin}
&=\int_\Gamma \int_\Gamma \Phi(x,y) f^j(y) \overline{f^i(x)}\, \rd\cH^d(y)\rd\cH^d(x),
\end{align}
and, by \eqref{eq:L2dualrep}, the vector $\vec{b}\in\C^N$ has $i$th entry given by
\begin{align}
b_{i}=-\langle g,e^i \rangle_{H^{1/2}(\Gamma_\infty)\times H^{-1/2}(\Gamma_\infty)}
& = -(\tr g,f^i)_{\IL_2(\Gamma)}
= -\int_\Gamma \tr g(x) \overline{f^i(x)}\, \rd\cH^d(x).
\label{eq:GalerkinRHS}
\end{align}

For the canonical $\IL_2(\Gamma)$-orthonormal basis for $\IV_N$, where $f^j|_{T_j} := (\cH^d(T_j))^{-1/2}$ and $f^j|_{\Gamma\setminus T_j}:=0$, $j=1,\ldots,N$,
\begin{align}
\label{eqn:GalerkinElements}
A_{ij}&=(\cH^d(T_i))^{-1/2}(\cH^d(T_j))^{-1/2}\int_{T_i}\int_{T_j} \Phi(x,y) \, \rd\cH^d(y)\rd\cH^d(x),
\end{align}
\begin{align}
\label{eqn:GalerkinElementsRHS}
b_{i}&=-(\cH^d(T_i))^{-1/2}\int_{T_i} \tr g(x)
\, \rd\cH^d(x).
\end{align}

Once we have computed $\phi_N$ by solving \eqref{eq:cvec} we will compute approximations to $u(x)$ and $u^\infty(x)$, given by \eqref{eqn:Rep}/\eqref{eq:SLPdef} and \eqref{eq:ffpattern}, respectively. Each expression takes the form\footnote{Note that if $\psi\in H^{-1/2}_\Gamma$ and $\varphi^\dag \in H^{1/2}(\Gamma_\infty)$ then $\langle\varphi^\dag,\psi\rangle_{H^{1/2}(\Gamma_\infty)\times H^{-1/2}(\Gamma_\infty)}=\langle\varphi,\psi\rangle_{H^{1/2}(\Gamma_\infty)\times H^{-1/2}(\Gamma_\infty)}$, where $\varphi := P\varphi^\dag\in (\tH^{1/2}(\Gamma^c))^\perp$, since $\langle \phi,\psi\rangle_{H^{1/2}(\Gamma_\infty)\times H^{-1/2}(\Gamma_\infty)}= 0$ for $\psi \in H^{-1/2}_\Gamma$ and $\phi \in \tH^{1/2}(\Gamma^c)$.}
\begin{equation} \label{eq:Jdef}
J(\phi) = \langle \varphi, \overline{\phi}\rangle_{H^{1/2}(\Gamma_\infty)\times H^{-1/2}(\Gamma_\infty)},
\end{equation}
for %
some $\varphi\in (\tH^{1/2}(\Gamma^c))^\perp$. Explicitly, where $\sigma$ is any element of $C^\infty_{0,\Gamma}$ (with $x$ not in the support of $\sigma$ in the case $u(x)$),
\begin{equation} \label{eq:varphi}
\varphi = -P\gamma_\pm\left(\sigma v|_{U_\pm}\right),
\end{equation}
with $v = \Phi(x,\cdot)$ in the case that $J(\phi)=u(x)$, $v=\Phi^\infty(\hat x,\cdot)$ in the case that $J(\phi)=u^\infty(\hat x)$; note that each $v$ is $C^\infty$ in a neighbourhood of $\Gamma$. In each case we approximate $J(\phi)$ by $J(\phi_N)$ which, recalling \eqref{eq:L2dualrep}, is given explicitly by
\begin{eqnarray} \nonumber
J(\phi_N) &=& \langle \varphi, \overline{\phi_N}\rangle_{H^{1/2}(\Gamma_\infty)\times H^{-1/2}(\Gamma_\infty)} = \sum_{j=1}^N c_j\langle \varphi, \overline{e^j}\rangle_{H^{1/2}(\Gamma_\infty)\times H^{-1/2}(\Gamma_\infty)}\\ \label{eq:JphiN}
& =& \sum_{j=1}^Nc_j(\tr \varphi,\overline{f^j})_{\IL_2(\Gamma)}
= \sum_{j=1}^Nc_j\int_\Gamma \tr \varphi \,f^j\, \rd\cH^d.
\end{eqnarray}
For the canonical $\IL_2(\Gamma)$-orthonormal basis for $\IV_N$,
\begin{equation} \label{eq:JphiNCanonical}
J(\phi_N)
= \sum_{j=1}^N\frac{c_j}{(\cH^d(T_j))^{1/2}}\int_{T_j} \tr \varphi \, \rd\cH^d.
\end{equation}

The following is a basic convergence result.%

\begin{thm}
\label{thm:Convergence}
Let $\Gamma$ be a compact $d$-set for some $n-1<d\leq n$. Then for each $N\in\N$ the variational problem \rf{eqn:Variational}
has a unique solution $\phi_N\in V_N$ and $\|\phi - \phi_N\|_{H^{-1/2}(\Gamma_\infty)} \to 0$ as $N\to \infty$, where $\phi\in H^{-1/2}_\Gamma$ denotes the solution of \rf{eqn:BIE}, provided that
$h_N:=\max_{j=1,\ldots,N}\diam(T_j) \to 0$ as $N\to \infty$. Further, where $J(\cdot)$ is given by \eqref{eq:Jdef} for some $\varphi \in  (\tH^{1/2}(\Gamma^c))^\perp$, $J(\phi_N)\to J(\phi)$ as $N\to\infty$.
\end{thm}
\begin{proof}
The well-posedness of \rf{eqn:Variational} follows from the Lax--Milgram lemma and the continuity and coercivity of $S:H^{-1/2}_{\Gamma}\to (\tH^{1/2}(\Gamma^c))^\perp$.
Furthermore, by C\'ea's lemma (e.g., \cite[Theorem 8.1]{Steinbach}) we have the following quasi-optimality estimate, where  $C_a,\alpha>0$ are as in \rf{eq:ContCoer}:
\begin{align}
\label{eq:Quasiopt}
\|\phi - \phi_N\|_{H^{-1/2}(\Gamma_\infty)} \leq \frac{C_a}{\alpha}\inf_{\psi_N\in V_N}\|\phi - \psi_N\|_{H^{-1/2}(\Gamma_\infty)}.
\end{align}
Hence to prove convergence of $\phi_N$ to $\phi$ it suffices to show that $\inf_{\psi_N\in V_N}\|\phi - \psi_N\|_{H^{-1/2}(\Gamma_\infty)}\to 0$ as $N\to \infty$, which, by the definition of $V_N$ and the fact that $\tr^*: \IH^{-t_d}(\Gamma) \to H^{-1/2}_\Gamma$ is a unitary isomorphism, is equivalent to showing that $\inf_{\Psi_N\in \IV_N}\|(\tr^*)^{-1}\phi - \Psi_N\|_{\IH^{-t_d}(\Gamma)}\to 0$ as $N\to \infty$.
Furthermore, since $\IL_2(\Gamma)$ is continuously embedded in $\IH^{-t_d}(\Gamma)$ with dense image it suffices to show that $\inf_{\Psi_N\in \IV_N}\|\Psi - \Psi_N\|_{\IL_2(\Gamma)}\to 0$ as $N\to \infty$ for every fixed $\Psi\in\IL_2(\Gamma)$.
To show the latter
we note that the space $C(\Gamma)$ of continuous functions on $\Gamma$ (equipped with the supremum norm) is continuously embedded in $\IL_2(\Gamma)$ with dense image (continuity is obvious and density follows by the density of $C^\infty_0(\R^n)$ in $H^{1/2}(\R^n)$ and that (see \S\ref{sec:FunctionSpaces}) $\tr:H^{1/2}(\R^n)\to\IL_2(\Gamma)$ is continuous and has dense range).
Then, given $\epsilon>0$ and $\Psi\in\IL_2(\Gamma)$, there exists $\widetilde{\Psi}\in C(\Gamma)$ such that $\|\Psi - \widetilde{\Psi}\|_{\IL_2(\Gamma)}<\epsilon/2$, and by the uniform continuity of $\widetilde{\Psi}$ and the fact that $h_N\to 0$ as $N\to \infty$ there exists $N\in\N$ and $\Psi_N\in\IV_N$ such that $|\widetilde{\Psi}(x)-\Psi_N(x)|<\epsilon/(2\sqrt{\cH^d(\Gamma)})$ for $\cH^d$-a.e.\ $x\in \Gamma$, which implies that $\|\widetilde{\Psi} - \Psi_N\|_{\IL_2(\Gamma)}<\epsilon/2$, from which it follows that $\|\Psi - \Psi_N\|_{\IL_2(\Gamma)}<\epsilon$ by the triangle inequality. That also $J(\phi_N)\to J(\phi)$ is clear since $J(\cdot)$ is a bounded linear functional.
\end{proof}

\subsection{Best approximation error estimates}
\label{sec:BestApprox}
We now assume that $\Gamma$ is the attractor of an IFS of contracting similarities satisfying the OSC, as in \S\ref{sec:IFS}. In this case, one possible choice of BEM approximation space could be $\IV_N=\IW_\ell$, for some $\ell\in\N_0$ (recall the definitions in \S\ref{sec:Wavelets}),
so that $\{T_j\}_{j=1}^N=\{\Gamma_{\bm}\}_{\bm\in I_\ell}$ and $N=M^\ell$. However, since the spaces $\IW_\ell$ are defined by refinement to a certain prefractal level, if the contraction factors $\rho_1,\ldots,\rho_M$ are not all equal the resulting mesh elements may differ significantly in size when $\ell$ is large. This motivates the use of spaces defined by refinement to a certain element size.
Recalling the notations of \S\ref{sec:Wavelets} (in particular, $J_\nu$ is defined in \eqref{eq:Jnudef} and $\nu_0$  below \eqref{eq:Jnudef}), for $\nu\geq \nu_0$  let
\begin{align}
\label{eq:IXdef}
\IX_\nu := {\rm span}\big(\{\psi_0\}\cup \{\psi^m_{\bm}\}_{\bm\in J_{\nu'},\,\nu'\in\{\nu_0,\ldots,\nu\},\, m\in\{1,\ldots,M -1\}}\big)=  {\rm span}\big(\{\chi_{\bm}\}_{\bm\in K_{\nu}}\big),
\end{align}
where %
\begin{equation} \label{eq:Knudef}
K_{\nu}  := \big\{\bm \in I_{\N} %
: \diam(\Gamma_{\bm})< 2^{-\nu} \text{ and } \diam(\Gamma_{\bm_-})\geq 2^{-\nu}\big\}.
\end{equation}
Note that both of these spanning sets for $\IX_\nu$ are orthonormal bases, and that $\IX_\nu\subset \IX_{\nu^\prime}$, for $\nu_0\leq \nu\leq \nu^\prime$, giving that
\begin{equation} \label{eq:Xnugen}
\IX_\nu =  {\rm span}\big(\{\chi_0\}\cup \{\chi_{\bm}:\bm\in I_{\N}, \, \diam(\Gamma_{\bm_-})\geq 2^{-\nu}\}\big)
\end{equation}
(but with this spanning set not linearly independent).
Note also that, for $\nu\geq \nu_0$, $\{\Gamma_\bm:\bm\in K_\nu\}$ is a mesh in the sense introduced at the beginning of this section.
Figure \ref{fig:BasisCont} illustrates the definitions of $\IX_\nu$ and $K_\nu$, showing the meaning of $K_2$ for a particular IFS for which $\diam(\Gamma)=1$ so that $\nu_0=0$; for this same IFS we have that $K_0=K_1=\{1,2,3\}$.

\begin{figure}
\fbox{
\def\xa{0}\def\xb{0.16}\def\xc{0.2}\def\xd{0.26}\def\xe{0.3}\def\xf{0.4}\def\xg{0.5}\def\xh{0.56}\def\xi{0.575}\def\xj{0.5975}\def\xk{0.6125}\def\xl{0.65}\def\xm{0.75}\def\xn{0.85}\def\xo{0.875}\def\xp{0.9125}\def\xq{0.9375}\def\xr{1}
\begin{tikzpicture}[x=16cm]
\def\dx{0.05}
\def\dy{0.35}
\newcommand{\y}{0}
\def\y{0}
\draw (\xa,\y)--(\xr,\y);\node at ({(\xa+\xr)/2},\y-\dy) {$E_{0}$};
\def\y{-1}
\draw (\xa,\y)--(\xf,\y);\node at ({(\xa+\xf)/2},\y-\dy) {$E_{1}$};
\draw[ultra thick] (\xg,\y)--(\xl,\y);\node at ({(\xg+\xl)/2},\y-\dy) {$E_{2}$};
\draw (\xm,\y)--(\xr,\y);\node at ({(\xm+\xr)/2},\y-\dy) {$E_{3}$};
\def\y{-2}
\draw[ultra thick] (\xa,\y)--(\xb,\y);\node at ({(\xa+\xb)/2},\y-\dy) {$E_{(1,1)}$};
\draw[ultra thick] (\xc,\y)--(\xd,\y);\node at ({(\xc+\xd)/2},\y-\dy) {$E_{(1,2)}$};
\draw[ultra thick] (\xe,\y)--(\xf,\y);\node at ({(\xe+\xf)/2},\y-\dy) {$E_{(1,3)}$};
\draw (\xg,\y)--(\xh,\y);\node at ({(\xg+\xh)/2-0.01},\y-\dy) {$E_{(2,1)}$};
\draw (\xi,\y)--(\xj,\y);\node at ({(\xi+\xj)/2},\y-\dy) {$E_{(2,2)}$};
\draw (\xk,\y)--(\xl,\y);\node at ({(\xk+\xl)/2+0.02},\y-\dy) {$E_{(2,3)}$};
\draw[ultra thick] (\xm,\y)--(\xn,\y);\node at ({(\xm+\xn)/2},\y-\dy) {$E_{(3,1)}$};
\draw[ultra thick] (\xo,\y)--(\xp,\y);\node at ({(\xo+\xp)/2},\y-\dy) {$E_{(3,2)}$};
\draw[ultra thick] (\xq,\y)--(\xr,\y);\node at ({(\xq+\xr)/2},\y-\dy) {$E_{(3,3)}$};
\draw[white](0,.5)--(1,.5); %
\end{tikzpicture}}
\caption{
To illustrate the definition $K_\nu$, we show, as in Figure \ref{fig:Basis},  the sets $E_{\bm}$, $\bm\in I_\ell$, $\ell=0,1,2$, with $E=[0,1]$, for the IFS $s_1(x)=0.4x$, $s_2(x)=0.15x+0.5$, $s_3(x)=0.25x+0.75$. Let $\Gamma$ be the attractor of this IFS. Then $\Gamma_\bm=E_\bm\cap \Gamma$ so that $\diam(\Gamma_\bm)=\diam(E_\bm)$ for each $\bm$. The sets $E_\bm$ highlighted with thick lines are those with indices $\bm\in K_2$; for example, $\diam(\Gamma_3)=\diam(E_3)=1/4$, so that $(3,m)\in K_2$, $m=1,2,3$, and $\diam(\Gamma_2)=\diam(E_2)<1/4$, so that also $2\in K_2$.
}
\label{fig:BasisCont}
\end{figure}
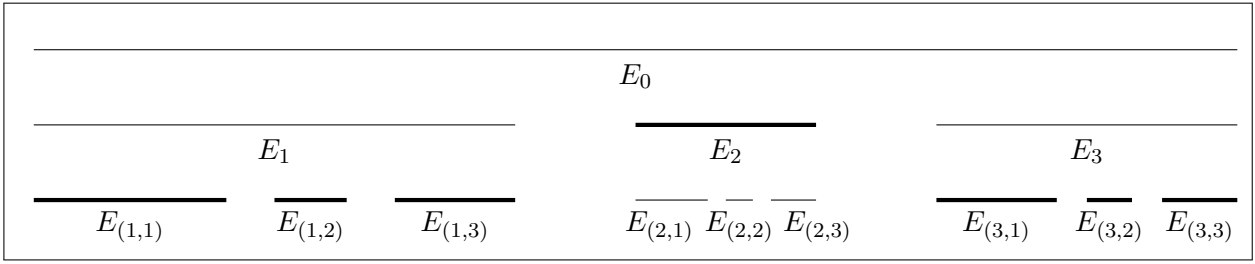
\begin{prop} \label{prop:convrateHausdorff}
Let $\Gamma$ be a disjoint IFS attractor with $n-1<d=\dimH(\Gamma)<n$.
Then for $-1<t<1$ and $\nu\geq \nu_0$ the orthogonal projection operator $\IP_\nu:\IL_2(\Gamma)\to \IX_\nu\subset \IL_2(\Gamma)$ defined by
\begin{align}
\label{eq:PnuDef}
\IP_\nu f = \beta_0  \psi_0 + \sum_{m=1}^{M-1}\sum_{\nu'=\nu_0}^{\nu}\sum_{\bm\in J_{\nu'}}\beta^m_{\bm} \psi^m_{\bm},
\end{align}
with the coefficients given by \eqref{eq:BetaDef1}, extends/restricts to a bounded linear operator $\IP_\nu:\IH^t(\Gamma)\to\IX_\nu \subset \IH^t(\Gamma)$,
and there exists $C>0$, independent of $\nu$ and $f$, such that
\begin{align}
\label{eq:HtboundPnu}
\| \IP_\nu  f\|_{\IH^t(\Gamma)}\leq C\| f\|_{\IH^t(\Gamma)}, \qquad \nu\geq \nu_0, \quad f\in \IH^t(\Gamma).
\end{align}
Furthermore, for $-1<t_{1}<t_{2}<1$ there exists $c>0$, independent of $\nu$ and $f$, such that%
\begin{align}
\label{eq:Htbound}
\| f-\IP_\nu  f\|_{\IH^{t_{1}}(\Gamma)}\leq c\,2^{-\nu(t_{2}-t_{1})}\| f\|_{\IH^{t_{2}}(\Gamma)}, \qquad \nu\geq \nu_0, \quad f\in \IH^{t_{2}}(\Gamma).
\end{align}
\end{prop}
\begin{proof}

The boundedness result follows from Theorem \ref{thm:Jonsson} and Corollary \ref{cor:Wavelets}, noting that
\[ \|\IP_\nu f\|_{t} = \Bigg(|\beta_0 |^2 + \sum_{m=1}^{M -1}\sum_{\nu'=\nu_0}^{\nu} 2^{2\nu' t}\sum_{\bm\in J_{\nu'}}|\beta^m_{\bm}|^2 \Bigg)^{1/2}\leq \|f\|_{t},\]
which gives \eqref{eq:HtboundPnu} with $C=\sup_{\Psi\in \IH^{t}(\Gamma)\setminus\{0\}}\frac{\|\Psi\|_{\IH^t(\Gamma)}}{\|\Psi\|_{t}}\times \sup_{\Psi\in \IH^{t}(\Gamma)\setminus\{0\}}\frac{\|\Psi\|_{t}}{\|\Psi\|_{\IH^{t}(\Gamma)}}$.

For the approximation result, we note, where $c>0$ denotes a constant independent of $\nu$ and $f$, not necessarily the same at each occurrence, that %
\begin{align*}
\| f-\IP_\nu  f\|_{\IH^{t_{1}}(\Gamma)} &= \Bigg\|\sum_{m=1}^{M-1}\sum_{\nu'=\nu+1}^{\infty}\sum_{\bm\in J_{\nu'}}\beta^m_{\bm}\psi^m_{\bm}\Bigg\|_{\IH^{t_{1}}(\Gamma)}\\
& \leq c \Bigg(\sum_{m=1}^{M -1}\sum_{\nu'=\nu+1}^\infty 2^{2\nu' t_1}\sum_{\bm\in J_{\nu'}}|\beta^m_{\bm}|^2 \Bigg)^{1/2}\\
&=c\Bigg(\sum_{m=1}^{M -1}\sum_{\nu'=\nu+1}^\infty 2^{-2\nu'(t_2-t_1)}2^{2\nu't_2}\sum_{\bm\in J_{\nu'}}|\beta^m_{\bm}|^2 \Bigg)^{1/2}\\
& \leq c 2^{-\nu(t_2-t_1)} \Bigg(\sum_{m=1}^{M -1}\sum_{\nu'=\nu+1}^\infty 2^{2\nu't_2}\sum_{\bm\in J_{\nu'}}|\beta^m_{\bm}|^2 \Bigg)^{1/2}%
\leq c 2^{-\nu(t_2-t_1)}\| f\|_{\IH^{t_{2}}(\Gamma)}.
\end{align*}
\end{proof}

Let $X_\nu := \tr^*(\IX_\nu)$. (Recall that the choice of $s>(n-d)/2$ in the definition of $\tr^*$ makes no difference to the definition of $X_\nu$.) Since $\IX_\nu\subset\IL_2(\Gamma)$ we have $X_\nu\subset H^{s}_{\Gamma}$ for all $s<-(n-d)/2$.
Then the following best approximation error estimate follows immediately, on application of $\tr^*$, from Theorem \ref{thm:Density} and Proposition \ref{prop:convrateHausdorff} applied with $-1<t_{1}<t_{2}<0$.
\begin{cor}
\label{cor:conv}
Let $\Gamma$ be a disjoint IFS attractor with $n-1<d=\dimH(\Gamma)<n$,
and suppose that $-(n-d)/2-1<s_{1}<s_{2}<-(n-d)/2$.
Then %
there exists $c>0$, independent of $\nu$ and $\psi$, such that %
\begin{align}
\label{eq:Hsbound}
\inf_{\psi_\nu\in X_\nu}\|\psi-\psi_\nu\|_{H^{s_{1}}_\Gamma}\leq c 2^{-\nu(s_{2}-s_{1})}\|\psi\|_{H^{s_{2}}_\Gamma},\qquad \nu\geq \nu_0,\,\, \psi\in H^{s_{2}}_\Gamma.
\end{align}
\end{cor}

\begin{rem} The above corollary holds also for larger values of $s_2$, but is then a trivial result since, as noted above Corollary \ref{cor:Interp}, $H_\Gamma^{s_2}=\{0\}$ for $s_2\geq -(n-d)/2$.
\end{rem}

Corollary \ref{cor:conv} can be rephrased with $2^{-\nu}$ replaced by a more general element size $h$, and $X_{\nu}$ replaced by
\begin{align}
\label{eq:YHDef}
Y_h := \tr^*(\IY_h),
\end{align}
where
\begin{align} \label{eq:YHDef2}
\IY_h := {\rm span}\big(\{\chi_{\bm}\}_{\bm\in L_{h}}\big),
\end{align}
and %
$L_h:=\{0\}$ for $h=\diam(\Gamma)$, while, for $h<\diam(\Gamma)$,
\begin{align}
\label{eq:LhDef}
L_{h} := \big\{\bm \in I_{\N} %
: \diam(\Gamma_{\bm})\leq h \text{ and } \diam(\Gamma_{\bm_-})>h\big\}.
\end{align}
This framework is a natural one for numerics --- we specify $h>0$ and consider all those components which have diameter less than or equal to $h$ but whose ``parent'' (in the IFS structure) has diameter greater than $h$. The change from $<$ to $\leq$ and $\geq$ to $>$ in moving from $K_\nu$ to $L_h$ is intentional, so that the definition of $L_h$ allows components that are equal to $h$ in diameter; $h$ is an upper bound for, and may be equal to, the maximum element diameter $h_N = \max_{\bm\in L_h}\diam(\Gamma_\bm)$. %
Note that the set $\{\chi_{\bm}\}_{\bm\in L_{h}}$ is an orthonormal basis for $\IY_h$ and, if $\chi_\bm\in \IY_h$, then $\chi_{\bm_-}\in \IY_h$, so that
\begin{equation} \label{eq:Yhgen}
\IY_h =  {\rm span}\big(\{\chi_0\}\cup\{\chi_{\bm}: \bm\in I_\N, \, \diam(\Gamma_{\bm_-})> h\}\big)
\end{equation}
(this spanning set not linearly independent).
We assume in the following corollary and subsequently that $h$, an upper bound for $h_N\leq \diam(\Gamma)$, satisfies
\begin{equation} \label{eq:hrange}
0<h\leq \diam(\Gamma).
\end{equation}
Except where indicated explicitly otherwise, results through the rest of the paper hold for all $h$ in this range. The following corollary, a best approximation result for the spaces $Y_h$, follows from the inclusion $\IX_\nu\subset \IY_h$, for some suitable $\nu$, and Corollary \ref{cor:conv}.
\begin{cor}
\label{cor:convh}
Under the assumptions of Corollary \ref{cor:conv} we have that
\begin{align}
\label{eq:Hsboundh}
\inf_{\psi_h\in Y_h}\|\psi-\psi_h\|_{H^{s_{1}}_\Gamma}\leq c  h^{s_{2}-s_{1}}\|\psi\|_{H^{s_{2}}_\Gamma},\qquad %
\psi\in H^{s_{2}}_\Gamma,
\end{align}
for some constant $c>0$ independent of $h$ and $\psi$.
\end{cor}
\begin{proof}
It is enough to show \eqref{eq:Hsboundh} for $0<h\leq\diam(\Gamma)/2$, since the result trivially holds for larger $h$ (just take $\psi_h=0$ on the left hand side). Thus, given $0<h\leq\diam(\Gamma)/2$,  let $\nu=\lceil \log(1/h)/\log(2) \rceil-1$, %
so that $2^{-\nu-1}\leq h<2^{-\nu}$ and $\nu\geq \nu_0$. %
Then, if $\bm\in K_\nu$, which implies that
$\diam(\Gamma_{\bm_-})\geq 2^{-\nu}$, it follows that $\diam(\Gamma_{\bm_-})>h$, so that $\chi_\bm\in \IY_h$. Thus  %
$\IX_\nu\subset\IY_h$, which implies that $ X_\nu \subset Y_h$, so that, where $c>0$ denotes a constant independent of $h$ and $\psi$, not necessarily the same at each occurrence,
\[
\inf_{\psi_h\in Y_h}\|\psi-\psi_h\|_{H^{s_{1}}_\Gamma}\leq \inf_{\psi_\nu\in X_v}\|\psi-\psi_\nu\|_{H^{s_{1}}_\Gamma}\leq c 2^{-\nu(s_{2}-s_{1})}\|\psi\|_{H^{s_{2}}_\Gamma}\leq c h^{s_{2}-s_{1}}\|\psi\|_{H^{s_{2}}_\Gamma}.
\]
\end{proof}

\subsection{Galerkin error estimates}
\label{sec:GalerkinBounds}
We now use the best approximation error results proved above to give an error bound for the Galerkin approximation to the BIE \eqref{eqn:BIE} when the solution
$\phi$ is sufficiently smooth.  Note that, combining the results of this theorem with inverse estimates that we prove as Theorem \ref{thm:inverse} below, we extend \eqref{eq:GalerkinBound} to a bound on $\|\phi-\phi_N\|_{H_\Gamma^{s_1}}$ for a range of $s_1$ in Corollary \ref{cor:GalExt} below.

\begin{thm}
\label{thm:BEMConvergence}
Let $\Gamma$ be a disjoint IFS attractor with $n-1<d=\dimH(\Gamma)<n$.
Let $\phi$ be the unique solution of \rf{eqn:BIE} and %
let $\phi_N$ be the unique solution of \rf{eqn:Variational} with $V_N=Y_h$, with $Y_h$ defined as in \rf{eq:YHDef}. Suppose that $\phi\in H^s_\Gamma$ for some $-1/2<s<-(n-d)/2$. Then, for some constant $c>0$ independent of $h$ and $\phi$,
\begin{align}
\label{eq:GalerkinBound}
\|\phi-\phi_N\|_{H^{-1/2}_\Gamma}\leq c  h^{s+1/2}\|\phi\|_{H^{s}_\Gamma}. %
\end{align}
Furthermore, let $\varphi\in (\tH^{1/2}(\Gamma^c))^\perp$ be such that the solution
$\zeta\in H^{-1/2}_\Gamma$ of \rf{eqn:VariationalCts}, with $-g$ replaced by $\varphi\in (\tH^{1/2}(\Gamma^c))^\perp$, also lies in the space $H^s_\Gamma$. Then the linear functional $J(\cdot)$ on $H^{-1/2}_\Gamma$, defined by $J(\psi):=\langle \varphi, \overline{\psi}\rangle_{H^{1/2}(\Gamma_\infty)\times H^{-1/2}(\Gamma_\infty)}$, $\psi\in H^{-1/2}_\Gamma$, satisfies
\begin{align}
\label{eq:GalerkinBoundJ}
|J(\phi)-J(\phi_N)|\leq c h^{2s+1}\|\phi\|_{H^{s}_\Gamma}\|\zeta\|_{H^{s}_\Gamma},%
\end{align}
for some constant $c>0$ independent of $h$, $\phi$, and $\zeta$.
\end{thm}
\begin{proof}
We first note that, in the general notation introduced at the start of this section, to get $V_N=Y_h$ we can take $\{T_j\}_{j=1}^N=\{\Gamma_{\bm}\}_{\bm\in L_h}$.
Then to obtain \rf{eq:GalerkinBound} we simply combine the best approximation error bound \rf{eq:Hsboundh} (for $s_1=-1/2$ and $s_2=s$) with the quasioptimality estimate \rf{eq:Quasiopt}.

The bound \rf{eq:GalerkinBoundJ} follows by a standard superconvergence argument, as used in e.g.\ \cite{HsiaoWendland81,SloanSpence:88}.
Since the kernel $\Phi(x,y)$ of the integral operator $S^\Omega$, densely defined on $H^{-1/2}(\Omega)$ by \eqref{eq:SOmega}, satisfies $\Phi(x,y)=\Phi(y,x)$ for $x\neq y$,
it follows from \eqref{eqn:Sesquir} and since $a(\cdot,\cdot)$ is the restriction of $a^\Omega(\cdot,\cdot)$ to $H_\Gamma^{-1/2}\times H_\Gamma^{-1/2}$,
that
\begin{equation} \label{eq:symm}
a(\upsilon,\psi)=a(\bar \psi,\bar \upsilon), \quad \upsilon,\psi\in H^{-1/2}_\Gamma,
\end{equation}
(recall that the overline denotes complex conjugation).
Now $\zeta\in H^{-1/2}_\Gamma$
is the solution of \rf{eqn:VariationalCts} with $-g$ replaced by $\varphi$, i.e.
$$
a(\zeta,\psi) = \langle \varphi,\psi\rangle_{H^{1/2}(\Gamma_\infty)\times H^{-1/2}(\Gamma_\infty)}, \quad \psi \in H^{-1/2}_\Gamma.
$$
Then, by \eqref{eq:symm}, Galerkin orthogonality (that $a(\phi-\phi_N,\psi_N)=0$, for all $\psi_N\in V_N$ by \eqref{eqn:VariationalCts} and \eqref{eqn:Variational}), and %
Lemma \ref{lem:coer},
for any $\zeta_N\in V_N$ we have, where $c>0$ denotes some constant independent of $h$, $\phi$, and $\zeta$, not necessarily the same at each occurrence,
\begin{align}
\nonumber
|J(\phi)-J(\phi_N)| & = |\langle \varphi,\overline{\phi-\phi_N}\rangle_{H^{1/2}(\Gamma_\infty)\times H^{-1/2}(\Gamma_\infty)}|\\ \nonumber
&= |a(\zeta,\overline{\phi-\phi_N})|=|a(\phi-\phi_N,\overline{\zeta})|\\ \nonumber
&=|a(\phi-\phi_N,\overline{\zeta-\zeta_N})|\\ \label{eq:basic}
&\leq c \|\phi-\phi_N\|_{H^{-1/2}_\Gamma}\|\zeta-\zeta_N\|_{H^{-1/2}_\Gamma}.
\end{align}
Choosing $\zeta_N$ to be the solution of \rf{eqn:Variational} with $-g$ replaced by $\varphi$,
it follows by \rf{eq:GalerkinBound} that
\[
|J(\phi)-J(\phi_N)| \leq c h^{s+1/2}\|\phi\|_{H^s_\Gamma}h^{s+1/2}\|\zeta\|_{H^s_\Gamma} = c h^{2s+1}\|\phi\|_{H^s_\Gamma}\|\zeta\|_{H^s_\Gamma},
\]
proving \rf{eq:GalerkinBoundJ}.
\end{proof}

Importantly, as observed above \eqref{eq:Jdef}, both $u(x)$ and $u^\infty(\hat x)$ can be written as $J(\phi)$, where $J(\cdot)$ is a bounded linear functional of the form treated in Theorem \ref{thm:BEMConvergence}, explicitly with $\varphi$ given by \eqref{eq:varphi}, where $v=\Phi(x,\cdot)$ in the case that $J(\phi)=u(x)$ and $v=\Phi^\infty(\hat x,\cdot)$ in the case that $J(\phi)=u^\infty(\hat x)$; in each case  $v$ is $C^\infty$ in a neighbourhood of $\Gamma$.

\begin{rem}[Range of exponent $s$ in the convergence rates]
\label{rem:extra_smoothness}
By Proposition \ref{prop:epsilon} and Remark \ref{rem:regularity}, there exists $\epsilon\in (0,t_d]$ (with $t_d=1/2-(n-d)/2$) such that, if $0<t< \epsilon$ and $\tr g\in \IH^{t+t_d}(\Gamma)$, then, where $s=-1/2+t$, $\phi\in H_\Gamma^s$ and \eqref{eq:GalerkinBound} holds with the same value of $s$. If also $\tr \varphi\in \IH^{t+t_d}(\Gamma)$, then $\zeta\in H_\Gamma^{s}$ and \eqref{eq:GalerkinBoundJ} holds with $s=-1/2+t$. In the case of scattering when $g$ is given by \eqref{eqn:gDefScatteringProblem}, $\tr g\in \IH^{t+t_d}(\Gamma)$ for all $0<t< \epsilon$, and also $\tr \varphi\in \IH^{t+t_d}(\Gamma)$ for all $0<t< \epsilon$, if $\varphi= P\gamma^\pm (\sigma v|_{U^\pm})$, where $\sigma\in C^\infty_{0,\Gamma}$ and $v$ is $C^\infty$ in a neighbourhood of $\Gamma$.
(As noted above,
$\varphi$ has this form
for the linear functionals needed to compute $u(x)$, for $x\in D$, and the far-field $u^\infty(\hat x)$, for $\hat x\in \mathds{S}^n$.)
If Conjecture \ref{ass:Smoothness} holds then we may take $\epsilon = t_d$. Thus, in the scattering case, if Conjecture \ref{ass:Smoothness} holds, then \eqref{eq:GalerkinBound} holds for all $-1/2<s<-(n-d)/2$, and the same is true for \eqref{eq:GalerkinBoundJ} if $\varphi= P\gamma^\pm (\sigma v|_{U^\pm})$ and $v$ is $C^\infty$ in a neighbourhood of $\Gamma$.
\end{rem}

\begin{rem}[Convergence rates when $\Gamma$ is a disjoint homogeneous IFS]
\label{rem:ConvRates}
Suppose that, in addition to the assumptions of Theorem \ref{thm:BEMConvergence}, $\Gamma$ is homogeneous, with $\rho_m=\rho $ for $m=1,\ldots,M$, for some $0<\rho<1$. Then, taking $h=\rho ^\ell\diam{\Gamma}$,
i.e.\ using the approximation space $V_N=\tr^*({\rm span}\left(\{\chi_{\bm}\}_{\bm\in I_\ell}\right))$, Theorem \ref{thm:BEMConvergence} implies that %
\[
\|\phi-\phi_N\|_{H^{-1/2}_\Gamma}\leq c \rho^{\ell(s+1/2)}\|\phi\|_{H^{s}_\Gamma},
\quad |J(\phi)-J(\phi_N)| \leq c \rho^{\ell(2s+1)}\|\phi\|_{H^{s}_\Gamma}\|\zeta\|_{H^{s}_\Gamma}.
\]
(Here and below $c>0$ denotes some constant independent of $\ell$, $\phi$, and $\zeta$, not necessarily the same at each occurrence.)
If $\phi\in H^s_{\Gamma}$ for all $s<-(n-d)/2$
 we can take $s=-(n-d)/2-\epsilon$ for arbitrarily small $\epsilon$ in the first of the above estimates, and the same holds for the second of the above estimates if also $\zeta$ satisfies $\zeta \in H^s_{\Gamma}$ for the same range of $s$. (If Conjecture~\ref{ass:Smoothness} holds this smoothness of $\phi$ and $\zeta$ is guaranteed for all scattering problems by Remark \ref{rem:extra_smoothness} if $\varphi$ satisfies the conditions in that remark.) Then, recalling that $d=\log(1/M)/\log \rho $, we get $s+1/2=(d+1-n)/2-\epsilon=(1/2)(\log(1/M)/\log \rho  + 1-n-2\epsilon)$, so that, for each $\epsilon>0$, %
\begin{align}
\label{eq:ConvRates}
\begin{aligned}
\|\phi-\phi_N\|_{H^{-1/2}_\Gamma}& \leq c \left(\sqrt{\frac{\rho ^{1-n-2\epsilon}}{M}}\right)^\ell\|\phi\|_{H^{-(n-d)/2-\epsilon}_\Gamma},\\
\quad |J(\phi)-J(\phi_N)| &\leq c \left(\frac{\rho ^{1-n-2\epsilon}}{M}\right)^\ell\|\phi\|_{H^{-(n-d)/2-\epsilon}_\Gamma}\|\zeta\|_{H^{-(n-d)/2-\epsilon}_\Gamma}.
\end{aligned}
\end{align}
In the case $n=1$ %
(e.g.\ $\Gamma$ a Cantor set, see \eqref{eq:CS_IFS}, for which $M=2$), the fact that $\epsilon>0$ can be taken arbitrarily small means that in numerical experiments we expect to see errors in computing $\phi_N$ and $J(\phi_N)$ that tend to zero roughly like $M^{-\ell/2}$ and $M^{-\ell}$ respectively, independent of the parameter $\rho$, if the above bounds are sharp.%

In the case $n=2$ %
(e.g.\ $\Gamma$ a Cantor dust, see \eqref{eq:CD_IFS}, for which $M=4$), we expect errors in computing $\phi_N$ and $J(\phi_N)$ that tend to zero roughly like $(M\rho )^{-\ell/2}$ and $(M\rho )^{-\ell}$ respectively, if the above bounds are sharp. Note that in this case we need $\rho >1/M$ to ensure $d=\log(1/M)/\log(\rho)>1=n-1$, and that these predicted convergence rates,
as a function of $\ell$, %
decrease as $d$ approaches $1$ with $M$ fixed.
For the Cantor dust with $\rho =1/3$ (the ``middle-third'' case) we predict convergence rates of roughly $(3/4)^{\ell/2}$ and $(3/4)^{\ell}$, respectively.
\end{rem}

\begin{rem}[Connection to standard BEM convergence results] \label{rem:classicalBEMconv}
The results of Theorem \ref{thm:BEMConvergence}, because they require that $d<n$, do not apply when $\Gamma$ is 
the closure of  
a Lipschitz domain (so that $d=\dim_H(\Gamma)=n$), for which case standard regularity and convergence results (e.g., \cite{StWe84,Ste:87,ErStEl90} and see the discussion above Conjecture \ref{ass:Smoothness}) predict that $\phi\in H^s_\Gamma$ and that the bounds \eqref{eq:GalerkinBound} and \eqref{eq:GalerkinBoundJ} hold for all $s<0$. Note that  these convergence rates for standard BEM are those predicted by taking the formal limit as $d\to n^-$ in the results of Theorem \ref{thm:BEMConvergence}. In \S\ref{sec:NumericalResults} we will compare results for the cases where $\Gamma$ is a Cantor set or Cantor dust, the attractor of the IFS \eqref{eq:CS_IFS} or \eqref{eq:CD_IFS}, respectively, with results for standard BEM. If we take the parameter $\rho=0.5$ in each of \eqref{eq:CS_IFS} and \eqref{eq:CD_IFS} then the attractor is just $\Gamma=[0,1]^n$, with $n=1$ or $2$. As we note in \S\ref{s:exp:Cantor} and \S\ref{s:exp:Dust}, the relative errors $\|\phi-\phi_N\|_{H^{-1/2}_\Gamma}/\|\phi\|_{H^{-1/2}_\Gamma}$ for our new Hausdorff-measure BEM for the Cantor set and Cantor dust with $\rho=0.49$ are almost the same as relative errors for the limiting case $\rho=0.5$, when $\Gamma=[0,1]^n$, for standard BEM with a uniform mesh and the same number of degrees of freedom.
\end{rem}

\subsection{Inverse estimates and conditioning}
\label{sec:InverseEstimates}

The following inverse estimate follows almost immediately from the wavelet characterisation of the spaces $\IH^{-t}(\Gamma)$ for $-1<t<1$ in Theorem \ref{thm:Jonsson} and Corollary \ref{cor:Wavelets}. In Appendix \ref{app:Inverse} we show, by an alternative, lengthier argument, closer to standard
arguments  based on ``bubble functions'' (e.g.~\cite{Dahmen04}), that the estimate \eqref{eq:InvEst} holds in fact for all $t>0$, and for the full range $0<d<n$; moreover, for any $T>0$, the constant $c_t$ in the estimate can be chosen independently of $t$ for $0<t\leq T$.
\begin{thm} \label{thm:inverse}
Let $\Gamma$ be a disjoint IFS attractor with $n-1<d=\dimH(\Gamma)<n$, and suppose that $-1<t_1<t_2<1$. Then, for some constant $c>0$ independent of $h$ and $\Psi_{h}$,
\begin{equation}\label{eq:InvEstGen}
\|\Psi_h\|_{\IH^{t_2}(\Gamma)} \leq c h^{t_1-t_2}\|\Psi_h\|_{\IH^{t_1}(\Gamma)}, \quad %
\Psi_h\in \IY_h.
\end{equation}
In particular, if
$0<t<1$,
then,
for some constant $c_t>0$ independent of $h$ and $\Psi_{h}$, %
\begin{equation}\label{eq:InvEst}
\|\Psi_h\|_{\IL_2(\Gamma)} \leq c_t h^{-t}\|\Psi_h\|_{\IH^{-t}(\Gamma)}, \quad %
\Psi_h\in \IY_h.
\end{equation}
\end{thm}
\begin{proof}
Since $\IY_h\subset \IY_{h^\prime}$ for $0<h^\prime \leq h\leq \diam(\Gamma)$, it is enough to show this result for $0<h\leq \diam(\Gamma)/2$. So suppose $0<h\leq \diam(\Gamma)/2$ and $\Psi_h\in \IY_h$. Then, arguing as in the proof of Corollary \ref{cor:convh}, we have that $\IY_h\subset \IX_{\nu+1}$, with $\nu\geq \nu_0$ such that $2^{-\nu -1}\leq h<2^{-\nu}$.
Then, since $\Psi_h = \IP_{\nu+1}\Psi_h$, and where the
coefficients are given by \eqref{eq:BetaDef1}  with $f$ replaced by $\Psi_h$,
\begin{align*}
\|\Psi_h\|_{t_1} &=\Bigg(|\beta_0 |^2 + \sum_{m=1}^{M -1}\sum_{\nu'=\nu_0}^{\nu+1} 2^{2\nu' t_1}\sum_{\bm\in J_{\nu'}}|{\beta^m_{\bm}}|^2\Bigg)^{1/2} \\
&\geq \min(1,2^{-(\nu+1)(t_2-t_1)})\Bigg(|\beta_0 |^2 + \sum_{m=1}^{M -1}\sum_{\nu'=\nu_0}^{\nu+1} 2^{2\nu' t_2}\sum_{\bm\in J_{\nu'}}|{\beta^m_{\bm}}|^2\Bigg)^{1/2}\\
&=\min(1,2^{-(\nu+1)(t_2-t_1)})\|\Psi_h\|_{t_2}.
\end{align*}
Noting that, where $s=t_2-t_1\in (0,2)$,
\begin{eqnarray*}
\min(1,2^{-(\nu+1)s}) = 2^{-(\nu+1)s}\min(1,2^{(\nu+1)s})
\geq 2^{-(\nu+1)s}\min(1,2^{2(\nu_0+1)})>2^{-s}h^{s}\min(1,2^{2(\nu_0+1)}),
\end{eqnarray*}
\eqref{eq:InvEstGen} follows by the equivalence of the norms $\|\cdot\|_{s}$ and $\|\cdot\|_{\IH^{s}(\Gamma)}$ for $-1<s<1$. The bound \eqref{eq:InvEst} is the special case of \eqref{eq:InvEstGen} with $t_2=0$ and $t_1=-t$.
\end{proof}

One application of the above result is to extend the bound \eqref{eq:GalerkinBound} on $\|\phi-\phi_N\|_{H_\Gamma^s}$ for $s=-1/2$ to a larger range $-1/2\leq s< -(n-d)/2$, by applying the following proposition in the case that $c_\psi=\|\psi\|_{H_\Gamma^{s_2}}$.
\begin{prop} \label{prop:InvApp1} Suppose that $-(n-d)/2-1<r<s_1\leq s_2<-(n-d)/2$, $\psi\in H_\Gamma^{s_2}$, $\psi_h\in Y_h$, and $c_\psi>0$, and that
$$
\|\psi-\psi_h\|_{H_\Gamma^{r}} \leq c^* h^{s_2-r} c_\psi, %
$$
for some constant $c^*>0$ independent of $h$, $\psi$, and $\psi_h$. Then
\begin{equation} \label{eq:ErrExt}
\|\psi-\psi_h\|_{H_\Gamma^{s_1}} \leq C h^{s_2-s_1} \left({c_\psi+}\|\psi\|_{H^{s_2}_\Gamma}\right) \quad \mbox{and} \quad  \|\psi_h\|_{H_\Gamma^{s_2}} \leq C\left(c_\psi+\|\psi\|_{H^{s_2}_\Gamma}\right),
\end{equation}
for some constant $C>0$ independent of $h$, $\psi$, and $\psi_h$.
\end{prop}
\begin{proof} As in the proof of Theorem \ref{thm:inverse}, since $Y_h\subset Y_{h^\prime}$ for $0<h^\prime \leq h\leq \diam(\Gamma)$, it is enough to show this result for $0<h\leq \diam(\Gamma)/2$. So suppose $0<h\leq \diam(\Gamma)/2$ and $\psi_h\in Y_h$. Recalling \eqref{eq:tr*}, let $\Psi:=(\tr^*)^{-1}\psi\in \IH^{s_2+(n-d)/2}(\Gamma)$ and $\Psi_h:=(\tr^*)^{-1}\psi_h\in \IY_h$. Then, for every $\Xi_h\in \IY_h$, using \eqref{eq:tr*} and the inverse inequality \eqref{eq:InvEstGen},
\begin{eqnarray} \nonumber
\|\psi-\psi_h\|_{H_\Gamma^{s_1}} &= &\|\Psi-\Psi_h\|_{\IH^{s_1+(n-d)/2}(\Gamma)} \leq  \|\Psi-\Xi_h\|_{\IH^{s_1+(n-d)/2}(\Gamma)} + \|\Xi_h-\Psi_h\|_{\IH^{s_1+(n-d)/2}(\Gamma)}\\ \nonumber
& \leq &  \|\Psi-\Xi_h\|_{\IH^{s_1+(n-d)/2}(\Gamma)} + ch^{r-s_1}\|\Xi_h-\Psi_h\|_{\IH^{r+(n-d)/2}(\Gamma)}\\ \nonumber
 & \leq & \|\Psi-\Xi_h\|_{\IH^{s_1+(n-d)/2}(\Gamma)} + ch^{r-s_1}\big(\|\Psi-\Xi_h\|_{\IH^{r+(n-d)/2}(\Gamma)} + \|\psi-\psi_h\|_{H_\Gamma^{r}}\big)\\ \label{eq:ErrEst2}
  & \leq & \|\Psi-\Xi_h\|_{\IH^{s_1+(n-d)/2}(\Gamma)} + ch^{r-s_1}\|\Psi-\Xi_h\|_{\IH^{r+(n-d)/2}(\Gamma)} +cc^*h^{s_2-s_1}c_\psi. %
\end{eqnarray}
Let $\nu\geq \nu_0$ be such that $2^{-\nu -1}\leq h<2^{-\nu}$, so that, by \eqref{eq:Xnugen} and \eqref{eq:Yhgen}, $\IX_\nu\subset\IY_{ h}$, and, where $\IP_\nu$ is as defined in Proposition \ref{prop:convrateHausdorff},
let $\Xi_h:= \IP_\nu\Psi\in \IX_\nu$. Then, for $r\leq s\leq s_2$, by Proposition \ref{prop:convrateHausdorff} (specifically by \eqref{eq:HtboundPnu} when $s=s_2$, \eqref{eq:Htbound}, applied with $t_1=s+(n-d)/2$ and $t_2=s_2+(n-d)/2$, when $r\leq s<s_2$), and again using \eqref{eq:tr*},
$$
\|\Psi-\Xi_h\|_{\IH^{s+(n-d)/2}(\Gamma)} \leq C_s 2^{-\nu(s_2-s)} \|\Psi\|_{\IH^{s_2+(n-d)/2}(\Gamma)} \leq C_s 2^{s_2-s}h^{s_2-s} \|\psi\|_{H^{s_2}_\Gamma}\leq 2C_s h^{s_2-s} \|\psi\|_{H^{s_2}_\Gamma}.
$$
This bound, applied with $s=r$ and $s=s_1$, combined with \eqref{eq:ErrEst2},
gives the first bound in \eqref{eq:ErrExt}. The second bound follows on taking $s_1=s_2$.
\end{proof}

The first claim of the following corollary follows immediately from the above result. The second (cf.~\cite[Thm.~1.4]{ErStEl90}, \cite{CoSt88}, \cite[Thm.~3.2.4]{Ciarlet78}) combines our earlier results with standard Aubin-Nitsche lemma arguments.
\begin{cor} \label{cor:GalExt} Suppose that $\Gamma$ satisfies the conditions of Theorem \ref{thm:BEMConvergence} and that $\phi$ and $\phi_N$ are as defined in Theorem \ref{thm:BEMConvergence}. Suppose also that $-1/2 \leq s_1 < s_2<-(n-d)/2$ and that $\phi\in H^{s_2}_\Gamma$. Then, for some constant $c>0$ independent of $h$ and $\phi$,
\begin{equation} \label{eq:GalerkinBoundG}
\|\phi-\phi_N\|_{H_\Gamma^{s_1}} \leq c h^{s_2-s_1} \|\phi\|_{H^{s_2}_\Gamma} \quad \mbox{and} \quad \|\phi_N\|_{H_\Gamma^{s_2}} \leq c \|\phi\|_{H^{s_2}_\Gamma}.
\end{equation}
If Conjecture \ref{ass:Smoothness} holds, then \eqref{eq:GalerkinBoundG} holds also for $-1+(n-d)/2< s_1<-1/2 \leq s_2 <-(n-d)/2$.
\end{cor}
\begin{proof} The first claim follows from Theorem \ref{thm:BEMConvergence} and Proposition \ref{prop:InvApp1} applied with $r=-1/2$ and $c_\psi = \|\psi\|_{H_\Gamma^{s_2}}$. Suppose now that Conjecture \ref{ass:Smoothness} holds and $-1+(n-d)/2< s_1 <-1/2 \leq s_2 <-(n-d)/2$.  For $\varphi\in \tH^{-s_1}(\Gamma^c)^\perp\subset H^{1/2}(\Gamma^\infty)$, define the bounded linear functional $J_\varphi(\cdot)$ on $H_\Gamma^{-1/2}$ by
$$
J_\varphi(\psi) = \langle \varphi,\overline{\psi}\rangle_{H^{1/2}(\Gamma_\infty)\times H^{-1/2}(\Gamma_\infty)} = \langle \varphi,\overline{\psi}\rangle_{H^{-s_1}(\Gamma_\infty)\times H^{s_1}(\Gamma_\infty)}, \quad \psi\in H^{-1/2}(\Gamma).
$$
Recalling from \S\ref{sec:FunctionSpaces} that $\tH^{-s_1}(\Gamma^c)^\perp$ is a unitary realisation of the dual space of $H_\Gamma^{s_1}$ with respect to the duality pairing $\langle \cdot,\cdot\rangle_{H^{-s_1}(\Gamma_\infty)\times H^{s_1}(\Gamma_\infty)}$, we have that
$$
\|\phi-\phi_N\|_{H_\Gamma^{s_1}} = \sup_{\varphi\in \tH^{-s_1}(\Gamma^c)^\perp \setminus\{0\}} \frac{|J_\varphi(\phi-\phi_N)|}{\|\varphi\|_{H^{-s_1}(\Gamma_\infty)}}.
$$
Given $\varphi\in \tH^{-s_1}(\Gamma^c)^\perp$, let
$\zeta\in H^{-1/2}_\Gamma$ denote the solution of \rf{eqn:VariationalCts} with $-g$ replaced by $\varphi$. Noting that, by Theorem \ref{thm:Density}, $\tr \varphi\in \IH^{-s_1-(n-d)/2}(\Gamma)$  with $\|\tr\varphi\|_{\IH^{-s_1-(n-d)/2}(\Gamma)}= \|\varphi\|_{H^{-s_1}(\Gamma_\infty)}$, it follows from
Remark \ref{rem:regularity} that $\zeta\in H^{-s_1-1}_\Gamma$, with $\|\zeta\|_{H^{-s_1-1}(\Gamma_\infty)} \leq C\|\varphi\|_{H^{-s_1}(\Gamma_\infty)}$, for some constant $C>0$ independent of $\varphi$. It follows from \eqref{eq:GalerkinBound} and \eqref{eq:basic} that, for some constant $c'>0$ independent of $\phi$, $h$, and $\zeta$,
$$
|J_\varphi(\phi-\phi_N)| \leq c'h^{s_2-s_1}\|\phi\|_{H^{s_2}_\Gamma}\|\zeta\|_{H^{-1-s_1}_\Gamma} \leq c'Ch^{s_2-s_1}\|\phi\|_{H^{s_2}_\Gamma}\|\varphi\|_{H^{-s_1}(\Gamma_\infty)},
$$
so that $\|\phi-\phi_N\|_{H_\Gamma^{s_1}}\leq c'Ch^{s_2-s_1}\|\phi\|_{H^{s_2}_\Gamma}$.
\end{proof}

Another application of Theorem \ref{thm:inverse}, specifically \eqref{eq:InvEst}, is to prove bounds on the condition number of the matrix in our Galerkin BEM. Let %
$N:=\# L_h$ and suppose that $\{f^i\}_{i=1}^N = \{\chi_{\bm}:\bm\in L_h\}$, i.e.\
$f^1,f^2,...,f^N$ is a particular ordering of the $\IL_2(\Gamma)$-orthonormal basis $\{\chi_{\bm}:\bm\in L_h\}$ of $\IY_h$, and let $\{e^i=\tr^*f^i\}_{i=1}^N$ be the corresponding basis for $Y_h$.
Then the Galerkin method using the $N$-dimensional space $V_N=Y_h$ leads to the Galerkin matrix $A\in \C^{N\times N}$ given by \eqref{eqn:GalerkinElements}. The following theorem bounds the $2$-norm, $\|\cdot\|_2$, of this matrix and its inverse. The bound for $\|A\|_2$ is in terms of $\|\IS\|_2$, the norm of $\IS$ as an operator on $\IL_2(\Gamma)$, and recall that $t_d$ is defined in \eqref{eq:tdDef}. The numerical results reported at the end of \S\ref{s:exp:Dust} suggest that these bounds are sharp in their dependence on $h$ and $d$.

\begin{thm}\label{thm:MatrixBounds}
Let $\Gamma$ be an IFS attractor satisfying the OSC with $n-1<d=\dimH(\Gamma)<n$. Then, with $V_N=Y_h$ as described above,
the Galerkin matrix $A\in \C^{N\times N}$ defined in \eqref{eqn:GalerkinElements} satisfies
$$
\|A\|_2 \leq \|\IS\|_2.
$$
If also
$\Gamma$ is disjoint
then, for some $c>0$, %
\begin{equation} \label{eq:DiscreteC}
|\va^H A \va| \geq c h^{2t_d}\|\va\|_2^2, \quad \mbox{for all } \va=(a_1,...,a_N)^T\in \C^N,
\end{equation}
so that also
\begin{equation} \label{eq:InvBound}
\|A^{-1}\|_2 \leq c^{-1} h^{-2t_d}.
\end{equation}
In particular, \rf{eq:DiscreteC} and \rf{eq:InvBound} hold with $c = \alpha c_{t_d}^{-2}$, where $\alpha$ is the coercivity constant from \eqref{eq:ContCoer} and $c_{t_d}$ is the constant from \eqref{eq:InvEst} when $t=t_d$.
\end{thm}
\begin{proof}
For $\va=(a_1,...,a_N)^T\in \C^N$ and $\vb=(b_1,...,b_N)^T\in \C^N$ it holds that
$$
\vb^HA\va = (\IS \Psi_h,\widetilde{\Psi}_h)_{\IL_2(\Gamma)},
$$
where $\Psi_h=\sum_{i=1}^N a_i f^i$, $\widetilde{\Psi}_h=\sum_{i=1}^N b_if^i$, so that
$$
\|A\|_2 = \sup_{\va,\vb\in \C^N\setminus\{0\}} \frac{|\vb^HA\va|}{\|\va\|_2\, \|\vb\|_2} = \sup_{\va,\vb\in \C^N\setminus\{0\}} \frac{|(\IS \Psi_h,\widetilde{\Psi}_h)_{\IL_2(\Gamma)}|}{\|\va\|_2\, \|\vb\|_2} \leq \|\IS\|_2,
$$
as $\|\Psi_h\|_{\IL_2(\Gamma)}=\|\va\|_2$ and $\|\widetilde{\Psi}_h\|_{\IL_2(\Gamma)}=\|\vb\|_2$, since $\{f^1,...,f^N\}$ is $\IL_2(\Gamma)$-orthonormal. Similarly,
it follows from
\eqref{eq:L2dualequiv},  \eqref{eqn:SesquiS}, Lemma \ref{lem:coer}, and Theorem \ref{thm:Density} that
\begin{equation} \label{eq:disc_coer}
|\va^HA\va| = |(\IS \Psi_h,\Psi_h)_{\IL_2(\Gamma)}| = |a(\tr^*\Psi_h,\tr^*\Psi_h)| \geq \alpha \, \|\tr^*\Psi_h\|_{H^{-1/2}_\Gamma}^2 = \alpha \, \|\Psi_h\|_{\IH^{-t_d}(\Gamma)}^2.
\end{equation}
The bound \eqref{eq:DiscreteC}, with $c = \alpha c_{t_d}^{-2}$, follows, if $\Gamma$ is disjoint, by applying Theorem \ref{thm:inverse} with $t=t_d$. The bound \eqref{eq:InvBound} then follows from \eqref{eq:DiscreteC} in the standard way, using the fact that $|\va^HA\va|\leq \|\va\|_2\, \|A\va\|_2$.
\end{proof}

\subsection{Numerical quadrature and fully discrete error estimates}
\label{sec:Quadrature}

To evaluate the Galerkin matrix \eqref{eqn:GalerkinElements} and right-hand side entries \eqref{eqn:GalerkinElementsRHS} we use the quadrature routines from \cite{HausdorffQuadrature}. These are applicable when $\Gamma$ is a disjoint IFS attractor with $n-1<d=\dimH(\Gamma)<n$, and we assume throughout this section that $\Gamma$ is of this form.

As in the previous two sections we focus on the case where $V_N=Y_{h}$ (with $Y_{h}$ defined as in \rf{eq:YHDef}), for some $h\in(0,\diam(\Gamma)]$, and we use the canonical orthonormal basis for $V_N$ 
(i.e.,\ $\tr^*$ applied to the functions $\chi_m$ from \eqref{eq:YHDef2} and \eqref{eq:VkBasisDefn})  %
so that $A$ and $\vb$ are given by \eqref{eqn:GalerkinElements} and \eqref{eqn:GalerkinElementsRHS}.
It follows from \rf{eqn:GalerkinElements} that the Galerkin matrix entries are the double integrals
\begin{align}
\label{eq:GalInt}
A_{ij} = \mu_{\bm(i)}^{-1/2}\mu_{\bm(j)}^{-1/2} \int_{\Gamma_{\bm(i)}}\int_{\Gamma_{\bm(j)}}\Phi(x,y)\,\rd\cH^d(y)\rd\cH^d(x), \qquad i,j\in {1,\ldots,N},
\end{align}
where $\bm(1),\ldots,\bm(N)$ is an ordering of the elements of $L_{h}$ (corresponding to the order of $f^1,\ldots,f^N$ in \S\ref{sec:InverseEstimates}), and for brevity we have written $\mu_{\bm}:=\cH^d(\Gamma_{\bm})$ for $\bm\in I_{\N_0}$. Recall from \S\ref{sec:Wavelets} that $\bm\in I_{\N_0}$ means that either $\bm=0$, in which case $\Gamma_{\bm}=\Gamma_0=\Gamma$ and $\mu_\bm=\cH^d(\Gamma)$, or $\bm=(m_1,\ldots,m_\ell)$ for some $\ell\in\N$, in which case $\mu_{\bm}=\rho_{m_1}^d\cdots \rho_{m_\ell}^d\cH^d\GG$.

For $i\neq j$ the integrand in \rf{eq:GalInt} is continuous,
because, by assumption, the components of $\Gamma$ are disjoint.
In that case, to evaluate \rf{eq:GalInt} we use the composite barycentre rule of \cite[Defn~3.5]{HausdorffQuadrature}.
In our context, this means partitioning the BEM elements $\Gamma_{\bm(i)}$ and $\Gamma_{\bm(j)}$, which have diameter approximately equal to $h$, into a union of self-similar subsets of a possibly smaller approximate diameter $0< h_Q\leq h$, writing the integral in \rf{eq:GalInt} as a sum of integrals over the Cartesian products of these subsets, and approximating these integrals by a one-point barycentre rule.
Specifically, we approximate
\begin{align}
\label{eq:OffDiagQuad}
A_{ij}\approx A_{ij}^Q:= \mu_{\bm(i)}^{-1/2}\mu_{\bm(j)}^{-1/2} \sum_{\bn\in L_{h_Q}^{\bm(i)}}\sum_{\bn'\in L_{h_Q}^{\bm(j)}}\mu_{\bn}\mu_{\bn'}\Phi(x_{\bn},x_{\bn'}),\qquad i,j\in {1,\ldots,N},\,i\neq j,
\end{align}
where, for $i\in {1,\ldots,N}$, %
\begin{align*}
\label{}
L_{h_Q}^{\bm(i)}:=
\begin{cases}
\{\bm(i)\}, & h_Q\geq\diam(\Gamma_{\bm(i)}),\\
\big\{\bn\in L_{h_Q}:\, \Gamma_{\bn}\subset\Gamma_{\bm(i)}\}, &  h_Q<\diam(\Gamma_{\bm(i)}),
\end{cases}
\end{align*}
describes the partitioning of $\Gamma_{\bm(i)}$
and, for $\bn\in I_{\N_0}$,
\[x_{\bn} := \frac{\int_{\Gamma_{\bn}}x\,\rd \cH^d(x)}{\int_{\Gamma_{\bn}}\,\rd \cH^d(x)} = \mu_{\bn}^{-1}\int_{\Gamma_{\bn}}x\,\rd \cH^d(x) \]
is the barycentre of $\Gamma_{\bn}$ with respect to the measure $\cH^d$.
(Note that $x_{\bn}$ is not necessarily an element of $\Gamma_{\bn}$.)
The similarities can be written as $s_m(x)=\rho_mA_mx+v_m$,
for some $v_m\in \R^n$ and some orthogonal $A_m\in\R^{n\times n}$ ($A_m=\pm 1$ in the case $n=1$), and then, writing $\bn=(n_1,\ldots,n_\ell)$, we find (see \cite[Prop.~3.3]{HausdorffQuadrature}) that
\[
x_{\bn}=s_{n_1}\circ s_{n_2}\circ \ldots \circ s_{n_\ell}
\Bigg(\bigg[I-\sum_{m=1}^M\rho_m^{d+1}A_m\bigg]^{-1}\bigg(\sum_{m=1}^M\rho_m^dv_m\bigg)\Bigg),
\]
the formula that we use for calculation of $x_{\bn}$ in \S\ref{sec:NumericalResults}.

For $i=j$ the integral \rf{eq:GalInt} is singular. To evaluate it we adopt the singularity subtraction approach of \cite[\S5]{HausdorffQuadrature}, writing
\begin{align}
 \label{eq:Aii}
  A_{ii} =  \mu_{\bm(i)}^{-1} \int_{\Gamma_{\bm(i)}}\int_{\Gamma_{\bm(i)}}\left(\Phi_{\rm sing}(x,y)+\Phi_{\rm reg}(x,y)\right)\,\rd\cH^d(y)\rd\cH^d(x),
 \end{align}
where
\[
\Phi_{\rm sing}(x,y) :=
\begin{cases}
-\frac{1}{2\pi}\log{|x-y|}, & \quad n=1,\\
\frac{1}{4\pi|x-y|}, & \quad n=2,
\end{cases}
\]
and $\Phi_{\rm reg}:=\Phi -\Phi_{\rm sing}$. The integral of $\Phi_{\rm reg}$
has a continuous integrand
and can be evaluated using the composite barycentre rule,
viz.
\begin{align}
\label{eq:AiiReg}
\int_{\Gamma_{\bm(i)}}\int_{\Gamma_{\bm(i)}}\Phi_{\rm reg}(x,y)\,\rd\cH^d(y)\rd\cH^d(x)\approx \sum_{\bn\in L_{h_Q}^{\bm(i)}}\sum_{\bn'\in L_{h_Q}^{\bm(i)}}\mu_{\bn}\mu_{\bn'}\Phi_{\rm reg}(x_{\bn},x_{\bn'}).
\end{align}
The integral of $\Phi_{\rm sing}$ over $\Gamma_{\bm(i)}\times\Gamma_{\bm(i)}$ is singular, but, using the self-similarity of $\Gamma$ and the symmetry and homogeneity properties of $\Phi_{\rm sing}$, namely the fact that,
for $m=1,\ldots,M$,
$$
\Phi_{\rm sing}( s_m(x),s_m(y)) = \left\{\begin{array}{cc}\Phi_{\rm sing}(x,y) -\frac{1}{2\pi} \log \rho_m, & n=1,\\ \rho_m^{-1}\Phi_{\rm sing}(x,y), & n=2,\end{array}\right.
$$
 it can be written (see \cite[Thm~4.6]{HausdorffQuadrature}) as
\begin{align}
\notag
&\int_{\Gamma_{\bm(i)}}\int_{\Gamma_{\bm(i)}}\Phi_{\rm sing}(x,y)\,\rd\cH^d(y)\rd\cH^d(x)
 \\
\label{eqn:singsub}
& =
\begin{cases}
\displaystyle\bigg(1-\sum_{m=1}^M \rho_m^{2d}\bigg)^{-1}\sum_{m=1}^M\Bigg(-\frac{1}{2\pi} \mu_{\bm(i)}^2\rho_m^{2d}\log( \rho_m) &\\
\displaystyle \hspace{35mm} + \sum_{\substack{m'=1\\m'\neq m}}^M \int_{\Gamma_{(\bm(i),m)}}\int_{\Gamma_{(\bm(i),m')}}\!\!\!\!\!\!\Phi_{\rm sing}(x,y)\,\rd\cH^d(y)\rd\cH^d(x)\Bigg),& n=1,\\
\displaystyle\bigg(1-\sum_{m=1}^M \rho_m^{2d-1}\bigg)^{-1}\sum_{m=1}^M\sum_{\substack{m'=1\\m'\neq m}}^M \int_{\Gamma_{(\bm(i),m)}}\int_{\Gamma_{(\bm(i),m')}}\Phi_{\rm sing}(x,y)\,\rd\cH^d(y)\rd\cH^d(x),& n=2.
\end{cases}
\end{align}
The integrals appearing in \rf{eqn:singsub} all have continous integrands and can be evaluated using the composite barycentre rule (as in \cite[Eqn (47)]{HausdorffQuadrature}),
viz.
\begin{align}
\label{eq:AiiSing}
\int_{\Gamma_{(\bm(i),m)}}\int_{\Gamma_{(\bm(i),m')}}\Phi_{\rm sing}(x,y)\,\rd\cH^d(y)\rd\cH^d(x)\approx \sum_{\bn\in L_{h_Q}^{(\bm(i),m)}}\sum_{\bn'\in L_{h_Q}^{(\bm(i),m')}}\mu_{\bn}\mu_{\bn'}\Phi_{\rm sing}(x_{\bn},x_{\bn'}).
\end{align}
Let $A_{ii}^Q$ denote the resulting approximation of $A_{ii}$, obtained by combining \eqref{eq:Aii}-\eqref{eq:AiiSing}.

Assuming the data $g$ is sufficiently smooth, the right-hand side entries \eqref{eqn:GalerkinElementsRHS} are regular single integrals and can be evaluated by a single integral version of the composite barycentre rules described for double integrals above, as in \cite[Defn~3.1]{HausdorffQuadrature}). Explicitly,
we use the approximation
\begin{align}
\label{eq:bi}
b_{i} = -\mu_{\bm(i)}^{-1/2} \int_{\Gamma_{\bm(i)}}%
\tr g(x)
\,\rd\cH^d(x)
\approx b_{i}^Q:=-\mu_{\bm(i)}^{-1/2}  \sum_{\bn\in L_{h_Q}^{\bm(i)}}\mu_{\bn}
\tr g(x_{\bn})
.
\end{align}
Having computed $A^Q=[A_{ij}^Q]_{i,j=1}^N\approx A$ and $\vb^Q=(b_1^Q,\ldots,b_N^Q)^T\approx \vb$, we solve, instead of \eqref{eq:cvec}, the perturbed linear system $A^Q\vc^{\,Q}=\vb^Q$. Provided this is uniquely solvable (for which see Corollary \ref{cor:fullydiscrete} below), our fully discrete approximation to $\phi$ is $\phi_N^Q := \sum_{j=1}^N c_j^Q e^j$.

Provided $\varphi\in (\tH^{1/2}(\Gamma^c))^\perp$ is sufficiently smooth, we approximate $J(\phi)$, given by \eqref{eq:Jdef}, by an approximation $J^Q(\phi_N^Q)$ defined in a similar way to \eqref{eq:bi}. Using \eqref{eq:JphiNCanonical}, we have that
\begin{equation} \label{eq;JQdef}
J(\phi)\approx J(\phi_N) = \sum_{j=1}^Nc_j \mu_{\bm(j)}^{-1/2}\int_{\Gamma_{\bm(j)}} \tr \varphi \, \rd\cH^d
 \approx J^Q(\phi^Q_N)= \sum_{j=1}^N c_j^Q\mu_{\bm(j)}^{-1/2}\sum_{\bn\in L_{h_Q}^{\bm(j)}}\mu_{\bn}\tr \varphi(x_{\bn}),
\end{equation}
where, for every $\psi_N = \sum_{j=1}^N d_je^j\in V_N$,
\begin{equation} \label{eq:JQdef2}
J^Q(\psi_N) := \sum_{j=1}^N d_j\mu_{\bm(j)}^{-1/2}\sum_{\bn\in L_{h_Q}^{\bm(j)}}\mu_{\bn}\tr \varphi(x_{\bn}).
\end{equation}

The following quadrature error estimates follow from results in \cite{HausdorffQuadrature}, specifically {\cite[Prop.~5.2, Thms~3.6(iii), 5.7 \& 5.11]{HausdorffQuadrature}}.
Following the terminology of \cite{HausdorffQuadrature}, we say an IFS attractor $\Gamma$ is \textit{hull-disjoint} if the convex hulls $\Hull(\Gamma_1), \ldots, \Hull(\Gamma_M)$ are disjoint, which holds if and only if the OSC holds for some open set $O\supset \Hull(\Gamma)$ (cf.\ Lemma \ref{lem:DisjointOSC}, the proof of which works in the same way for hull-disjointness as for disjointness). 
The assumption of hull-disjointness permits analysis of the singular quadrature rules in \cite{HausdorffQuadrature} by Taylor expansion. (The point here is that while the barycentre $x_{\bm}$ of a component $\Gamma_{\bm}$ may not lie in $\Gamma_{\bm}$, it must lie in $\Hull(\Gamma_{\bm}$).) However, numerical experiments suggest that hull-disjointness is not essential for the applicability of the quadrature rules in \cite{HausdorffQuadrature} (see \cite[\S6]{HausdorffQuadrature}). We also suspect that the estimates \rf{eq:AquadEst1} and \rf{eq:AquadEst2} may not be sharp in their $h$-dependence --- see \cite[Rem.~5.10]{HausdorffQuadrature} and the discussion around \cite[Fig.~7]{HausdorffQuadrature}.
In Remark \ref{rem:reducedquad} we describe modifications to our quadrature rule which may improve its efficiency. %
\begin{thm}
\label{thm:Quadrature}
Let $\Gamma$ be an IFS attractor satisfying the OSC with $n-1<d=\dimH(\Gamma)<n$.
Suppose that $0<h_Q\leq h$, and let $A_{ij}$, $A_{ij}^Q$, $b_i$ and $b_i^Q$ be as in \eqref{eq:GalInt}-\eqref{eq:bi}, and $J(\cdot)$ and $J^Q(\cdot)$ be as defined by \eqref{eq:Jdef} and \eqref{eq:JQdef2}, respectively, for some $\varphi\in (\tH^{1/2}(\Gamma^c))^\perp$.

\begin{enumerate}[(i)]
\item
Suppose that $\tr g = G|_\Gamma$, for some $G$ that is twice boundedly differentiable in some open neighbourhood of $\Hull(\Gamma)$.
(For the scattering problem this holds with $G=-u^i$, if $u^i$ satisfies the Helmholtz equation in a neighbourhood of $\Hull(\Gamma)$.)
Then, for $i=1,\ldots,N$,
\begin{align}
\label{eq:bquadEst}
|b_{i}-b^Q_{i}| &\leq h_Q^2
|G|_{2,\Hull(\Gamma)}
\mu_{\bm(i)}^{1/2}
\qquad \text{ and hence } \qquad
\|\vb-\vb^Q\|_2 \leq h_Q^2
|G|_{2,\Hull(\Gamma)}\cH^d(\Gamma)^{1/2},
\end{align}
where $|G|_{2,\Hull(\Gamma)}:=\max_{x\in\Hull(\Gamma)}\max_{\substack{\alpha\in \N_0^n\\|\alpha|=2}}|D^\alpha G(x)|$.
\item
Suppose that $\tr\varphi = V|_\Gamma$, where $V$ is twice boundedly differentiable in some open neighbourhood of $\Hull(\Gamma)$.
(For $\varphi$ given by \eqref{eq:varphi} this holds with $V=-v$ if $v$ is $C^\infty$ in a neighbourhood of $\Hull(\Gamma)$.) Then, for $-1/2\leq s<-(n-d)/2$, there exists a constant $C>0$, independent of $h$, $h_Q$, and $V$, such that, for all $\psi_N\in V_N=Y_h$,
\begin{align} \label{eq:JJQ}
|J(\psi_N)-J^Q(\psi_N)| &\leq C h_Q^2 h^{s+(n-d)/2} |V|_{2,\Hull(\Gamma)} \|\psi_N\|_{H_\Gamma^{s}}.
\end{align}
\item
Suppose that $\Gamma$ is hull-disjoint. Then there exists a constant $C>0$, independent of $h$ and $h_Q$, such that, for $i,j=1,\ldots,N$,
\begin{align}
\label{eq:AquadEst2}
|A_{ij}-A^Q_{ij}| &\leq
C h_{Q} h^{-n}\mu_{\bm(i)}^{1/2}\mu_{\bm(j)}^{1/2},
\qquad \text{ and hence } \qquad
\|A-A^Q\|_2 \leq Ch_Q h^{-n}
\cH^d(\Gamma).
\end{align}
If, further, $\Gamma$ is homogeneous, then there exists a constant $C>0$, independent of $h$ and $h_Q$, such that
\begin{align}
\label{eq:AquadEst1}
|A_{ij}-A^Q_{ij}| &\leq
C h^2_{Q} h^{-(n+1)}\mu_{\bm(i)}^{1/2}\mu_{\bm(j)}^{1/2},
\qquad \text{ and hence } \qquad
\|A-A^Q\|_2 \leq Ch_Q^2 h^{-(n+1)}
\cH^d(\Gamma).
\end{align}
\end{enumerate}
\end{thm}

\begin{proof}
In what follows, when applying results from \cite{HausdorffQuadrature} we are taking $h$, $\Gamma$ and $\Gamma'$ in \cite{HausdorffQuadrature} to be respectively $h_Q$, $\Gamma_{\bm(i)}$, and $\Gamma_{\bm(j)}$ from the current paper.
For (i), noting that
\[ |b_{i}-b^Q_{i}| = \mu_{\bm(i)}^{-1/2}\bigg|\int_{\Gamma_{\bm(i)}} G(x) \, \rd\cH^d(x) -\sum_{\bn\in L_{h_Q}^{\bm(i)}}\mu_{\bn} G(x_{\bn})\bigg|,\]
the first estimate in \rf{eq:bquadEst} follows from \cite[Thm 3.6(iii)]{HausdorffQuadrature}, and the second then follows from the fact that $\sum_{i=1}^N\mu_{\bm(i)}=\cH^d(\Gamma)$.
For (ii), given $\psi_N\in V_N$ let $\vpsi=((\vpsi)_1,\ldots,(\vpsi)_N)^T\in\IC^N$ denote the coefficient vector of its expansion with respect to the basis $(e^1,\ldots,e^N)$.
By the orthonormality of the basis $(f^1,\ldots, f^N)$ in $\IL_2\GG$, the inverse estimate \eqref{eq:InvEst} and the isometry property \eqref{eq:tr*}, given $0<t<1$ we can bound
\begin{align}
\|\vpsi\|_2
=\|(\tr^*)^{-1}\psi_{N}\|_{\IL_2\GG}
\overset{\eqref{eq:InvEst}}\le {c_{t}}  h^{-t}\|(\tr^*)^{-1}\psi_{N}\|_{\IH^{-t}\GG}
\overset{\eqref{eq:tr*}}={c_{t}}  h^{-t}\|\psi_{N}\|_{H^{-t-\frac{n-d}{2}}_\Gamma}.
\label{eq:CoeffBound}
\end{align}
Then, to prove \eqref{eq:JJQ}, using \eqref{eq:JphiNCanonical}, \eqref{eq:JQdef2}, %
\cite[Thm~3.6(iii)]{HausdorffQuadrature}, \eqref{eq:CoeffBound}, and the Cauchy-Schwarz inequality,
\begin{eqnarray*}
|J(\psi_N) - J^Q(\psi_N)| &= & \Bigg|\sum_{j=1}^N (\vpsi)_j\mu_{\bm(j)}^{-1/2}\bigg(\int_{\Gamma_{\bm(j)}} V(x) \, \rd\cH^d(x) -\sum_{\bn\in L_{h_Q}^{\bm(j)}}\mu_{\bn} V(x_{\bn})\bigg)\Bigg|\\
& \leq &  h_Q^2|V|_{2,\Hull(\Gamma)}(\cH^d(\Gamma))^{1/2}\|\vpsi\|_2\\
&\leq & c_{t} h_Q^2h^{-t}|V|_{2,\Hull(\Gamma)}(\cH^d(\Gamma))^{1/2}\|\psi_N\|_{H_\Gamma^{-t-\frac{n-d}{2}}}.
\end{eqnarray*}
For (iii), the first estimates in \rf{eq:AquadEst2} and \rf{eq:AquadEst1} follow from \cite[Prop 5.2]{HausdorffQuadrature} for the off-diagonal terms and \cite[Thms~5.7 \& 5.11]{HausdorffQuadrature} for the diagonal terms.
The factors of $h^{-(n+1)}$ and $h^{-n}$ come from the fact that, in the notation of \cite[Thms~5.7 \& 5.11]{HausdorffQuadrature}, $R_{\Gamma,\Hull}$ equals $\diam(\Gamma)$ times a constant independent of $\diam(\Gamma)$. We note that, since we are allowing $C$ to be $k$-dependent, we can use the ``$k\diam(\Gamma)\leq c_{\rm osc}$'' estimates in \cite[Thms~5.7 \& 5.11]{HausdorffQuadrature}, since, given $k>0$, the constant $c_{\rm osc}$ can be chosen as large as is required.
To derive the second estimate in \rf{eq:AquadEst2} (a similar argument gives the second estimate in \rf{eq:AquadEst1}), note that, with $\|\cdot\|_{\mathrm{F}}$ denoting the Frobenius norm,
\[ \|A-A^Q\|_2\leq \|A-A^Q\|_{\mathrm{F}} \leq Ch_Qh^{-n}\bigg(\sum_{i=1}^N\sum_{j=1}^N \mu_{\bm(i)}\mu_{\bm(j)} \bigg)^{1/2} = Ch_Qh^{-n}\cH^d(\Gamma).\]
\end{proof}

\begin{rem}[Value of $\cH^d(\Gamma)$]
\label{rem:HausdorffMeasure}
The proposed quadrature formulas require the values of $\mu_{\bm}=\cH^d(\Gamma_\bm)$, which are easily computed in terms of $\cH^d\GG$ as $\mu_{\bm}=\rho_{m_1}^d\cdots \rho_{m_\ell}^d\cH^d\GG$, for all $\bm\in I_\N$.
For $n=1$, the Hausdorff measure of a class of Cantor sets is shown to be $\cH^d\GG=1$ in \cite[Thm.~1.14--1.15]{Fal85}; see, e.g., \cite{Zuberman2019} for more recent related results.
However, for $n>1$ the exact value of the Hausdorff measure of even the simplest IFS attractors is known
only for $d\le1$ (see, e.g., \cite{XiongZhou2005}), i.e.\ for the cases that are not relevant for scattering problems (recall, as discussed above Lemma \ref{lem:Nullity}, that $H^{-1/2}_\Gamma=\{0\}$ and $\phi=0$ for $d\le n-1$).
A simple implementation technique which avoids working with an unknown value for $\cH^d(\Gamma)$ and produces the correct solution $\phi_N$ is just to set $\cH^d(\Gamma)=1$ in all calculations. This is equivalent to
introducing a ``normalised Hausdorff measure'' $\cH^d_\star(\cdot):=\cH^d(\cdot)/\cH^d\GG$, so that $\cH^d_\star\GG=1$, and using it throughout in the BEM in place of $\cH^d(\cdot)$.
\end{rem}

To study the influence of the quadrature error on the BEM solution we first adapt the first Strang lemma to our setting.
\begin{prop}\label{prop:Strang}
Let $\Gamma$ be a disjoint IFS attractor with $n-1<d=\dimH(\Gamma)<n$. Assume that
the unique solution $\phi$ of \eqref{eqn:BIE} belongs to $H^s_\Gamma$ for some $-1/2<s<-(n-d)/2$, so that $0<s+1/2<t_d$.
Let
$V_N=Y_{ h}$ so that $N=\dim(V_N)$.
Let $\alpha$ be the coercivity constant of \eqref{eq:ContCoer} and $c_{t_d}$ the constant in the inverse inequality \eqref{eq:InvEst} when $t=t_d$.
Assume that $A^Q\in\C^{N\times N}$ and $\vb^Q\in\C^N$ are approximations of the Galerkin matrix $A$ and the right-hand side vector $\vb$ in \eqref{eqn:GalerkinElements}--\eqref{eqn:GalerkinElementsRHS}
satisfying
\begin{equation}\label{eq:AmA}
\|A-A^Q\|_2\le   h^{2t_d}E_A\qquad \text{and} \qquad \|\vb-\vb^Q\|_2\le  h^{t_d+s+1/2} E_b,
\end{equation}
for some $0\leq E_A<\alpha/c_{t_d}^2$, and $E_b\geq 0$.
Then %
the perturbed linear system $A^Q\vc^{\,Q}=\vb^Q$ is invertible and the corresponding solution
$\phi_{N}^Q:=\sum_{j=1}^N c_j^Qe^j\in V_N$ %
satisfies the %
error bound
\begin{align}\label{eq:PerturbError}
\|\phi_{N}-\phi_{N}^Q\|_{H^\mhalf_\Gamma}
&\leq \frac{C}{\alpha_Q}\Big(E_A\|\phi\|_{H^s_\Gamma} + E_b\Big)  h^{s+1/2},
\end{align}
where $\phi_N=\sum_{j=1}^Nc_je^j\in V_N$ is the Galerkin solution given by \eqref{eqn:Variational}, $\alpha_Q:= \alpha - E_Ac_{t_d}^2$,
 and $C>0$ is a constant independent of $h$, $\phi$, $E_a$, $E_b$, and $N$.
\end{prop}
\begin{proof}
As in the proof of Theorem \ref{thm:Quadrature}, given $\psi_N\in V_N$ let $\vpsi\in\IC^N$ denote the coefficient vector of its expansion with respect to the basis $(e^1,\ldots,e^N)$.
Let $B(\psi_{N}):=-\langle g,\psi_{N}\rangle_{H^\half(\Gamma_\infty)\times H^\mhalf(\Gamma_\infty)} =\vpsi^H\vec b$ and denote the perturbed sesquilinear form and antilinear functional by $a^Q(\xi_{N},\psi_{N}):=\vpsi^H A^Q \vxi$ and $B^Q(\psi_{N}):=\vpsi^H\vec b^Q$, for $\xi_{N},\psi_{N}\in V_N$.
The first bound in \eqref{eq:AmA} gives %
\begin{align}
\nonumber
|a(\xi_{N},\psi_{N})-a^Q(\xi_{N},\psi_{N})|
={|\vpsi^H (A-A^Q)\vxi|}
&\le\|A-A^Q\|_2 \|\vxi\|_2 \|\vpsi\|_2\nonumber\\
&\le h^{t_d}E_A{c_{t_d}} \|\vxi\|_2  \|\psi_{N}\|_{H^\mhalf_\Gamma}   \label{eq:a-aQ}\\
&\le E_A {c_{t_d}^2} \|\xi_{N}\|_{H^\mhalf_\Gamma} \|\psi_{N}\|_{H^\mhalf_\Gamma},
\qquad  \forall \xi_{N},\psi_{N}\in V_N.
\nonumber
\end{align}
From this and the coercivity of $a(\cdot,\cdot)$ in \eqref{eq:ContCoer} follows the coercivity of the perturbed form, and hence the invertibility of $A^Q$:
$$
|a^Q(\psi_{N},\psi_{N})|
\ge |a(\psi_{N},\psi_{N})| - |a(\psi_{N},\psi_{N})-a^Q(\psi_{N},\psi_{N})|
\ge %
\alpha_Q\|\psi_{N}\|_{H^\mhalf_\Gamma}^2,
\quad \forall\psi_h\in V_N.
$$
The second bound in \eqref{eq:AmA}, combined with \eqref{eq:CoeffBound} for $t=t_d$, gives
\begin{align}
\label{eq:B-BQ}
|B(\psi_{N})-B^Q(\psi_{N})|
= |\vpsi^H(\vb-\vb^Q)|
\leq \|\vb-\vb^Q\|_2\|\vpsi\|_2
\leq E_b c_{t_d}h^{s+1/2}\|\psi_{N}\|_{H^\mhalf_\Gamma}^2,
\quad \forall\psi_h\in V_N.
\end{align}
Now $\phi^Q_{N}\in V_N$ is the solution of $a^Q(\phi^Q_{N},\psi_{N})=B^Q(\psi_{N})$, $\forall\psi_{N}\in V_N$.
Strang's first lemma (e.g.\ \cite[III.1.1]{Braess2007}), applied in the special case that $V=S_h$ (in the notation of \cite[III.1.1]{Braess2007}), combined with \eqref{eq:a-aQ} and
\eqref{eq:B-BQ},
gives
\begin{align*}
\|\phi_{N}-\phi^Q_{N}\|_{H^\mhalf_\Gamma}
&\le
\frac1{\alpha_Q}
\bigg(\sup_{\psi_{N}\in V_N} \frac{|a(\phi_{N},\psi_{N})-a^Q(\phi_{N},\psi_{N})|}{\|\psi_{N}\|_{H^\mhalf_\Gamma}}
+\sup_{\psi_{N}\in V_N} \frac{|B(\psi_{N})-B^Q(\psi_{N})|}{\|\psi_{N}\|_{H^\mhalf_\Gamma}}\bigg)
\\
&\le
\frac{c_{t_d}}{\alpha_Q}
\bigg(
E_A h^{t_d}\|\vc\|_2+E_b h^{s+1/2} \bigg).
\end{align*}
Applying \eqref{eq:CoeffBound} with
$t=-s-(n-d)/2$, we obtain that
\begin{align*}
\|\phi_{N}-\phi^Q_{N}\|_{H^\mhalf_\Gamma}
\leq
\frac{c_{t_d}}{\alpha_Q}\bigg(c_{-s-\frac{n-d}2} E_A \|\phi_N\|_{H^s_\Gamma}+E_b\bigg) h^{s+1/2},
\end{align*}
and the bound
\eqref{eq:PerturbError}
follows on applying Corollary \ref{cor:GalExt} with $s_2=s$.
\end{proof}
By combining Theorem \ref{thm:BEMConvergence}, Proposition~\ref{prop:Strang} and the quadrature error bounds in Theorem \ref{thm:Quadrature}, we can complete the convergence analysis of our fully discrete Hausdorff BEM for a hull-disjoint IFS attractor.

\begin{cor}
\label{cor:fullydiscrete}
Let $\Gamma$ be a hull-disjoint IFS attractor with $n-1<d=\dimH(\Gamma)<n$.
Assume that the unique solution $\phi$ of \eqref{eqn:BIE} belongs to $H^s_\Gamma$ for some $-1/2<s<-(n-d)/2$.
Let $V_N=Y_{ h}$.
Let the entries of the Galerkin matrix $A$ \eqref{eqn:GalerkinElements} and right-hand side $\vb$ \eqref{eqn:GalerkinElementsRHS} be approximated with the quadrature formulas outlined in \eqref{eq:GalInt}-\eqref{eq:bi}.
Let $g$ satisfy the assumptions of Theorem \ref{thm:Quadrature}(i),
and suppose that
\begin{align}
\label{eqn:hQuadCondGeneral}
h_Q \leq C_Q h^{d+1},
\end{align}
for some sufficiently small $C_Q>0$, independent of $h_Q$ and $h$.
Then the approximated linear system is invertible and the fully discrete solution $\phi_{N}^Q$ satisfies the %
error bound
\begin{align}\label{eq:PerturbErrorUniform}
\|\phi-\phi_{N}^Q\|_{H^\mhalf_\Gamma}
&\leq C
\big(\|\phi\|_{H^s_\Gamma} + |G|_{2,\Hull(\Gamma)}\big)
 h^{s+1/2},
\end{align}
for some $C>0$ independent of $h$, $\phi$, and $G$.

If, further, $\Gamma$ is homogeneous the above result holds with \eqref{eqn:hQuadCondGeneral} replaced by the weaker condition
\begin{align}
\label{eqn:hQuadCondUniform}
h_Q \leq C_Q h^{d/2+1}.
\end{align}
\end{cor}

\begin{proof}
The two conditions \eqref{eqn:hQuadCondGeneral} and \eqref{eqn:hQuadCondUniform} can be written as
\begin{align}
\label{eqn:hQuadCondBoth}
h_Q\leq C_Q h^{\frac{d}{1+p}+1},
\end{align}
where $p=0$ in the general case, and $p=1$ when $\Gamma$ is also homogeneous.
By Theorem \ref{thm:Quadrature}(i) and (iii) our quadrature formulas
achieve \eqref{eq:AmA} with
\begin{equation} \label{eq:EAbForm}
E_A = C_* h_Q^{1+p} h^{-d-1-p},
\qquad E_b = C_* h_Q^2 h^{-s-1+\frac{n-d}{2}}|G|_{2,\Hull(\Gamma)},
\end{equation}
for some constant $C_*$ independent of both $h$ and $h_Q$.
Hence \eqref{eqn:hQuadCondBoth} implies that $E_A$ is bounded independently of $h$, and, furthermore, by choosing the constant $C_Q$ in \eqref{eqn:hQuadCondBoth} to be sufficiently small, one can ensure that
$E_A\leq \alpha/(2c_{t_d}^2)$, so that the discrete system is invertible by Proposition \ref{prop:Strang}, with
$\alpha_Q := \alpha -E_Ac_{t_d}^2\geq \alpha/2$.
The assumption of \eqref{eqn:hQuadCondBoth} also ensures that $E_b$ is bounded independently of $h$, with
\[E_b\leq C_*C_Q^2h^{
\frac{2d}{1+p}
-s+\frac{n-d}{2}+1}|G|_{2,\Hull(\Gamma)} \leq C_*C_Q^2(\diam(\Gamma))^{
\frac{2d}{1+p}
-s+\frac{n-d}{2}+1}|G|_{2,\Hull(\Gamma)},\]
since $h\leq \diam(\Gamma)$ and
$\frac{2d}{1+p}-s+\frac{n-d}{2}+1>0$
for both $p=0$ and $p=1$. 
Then, by \eqref{eq:PerturbError},
\begin{align}\label{eq:PerturbErrorUniformDiscrete}
\|\phi_N-\phi_{N}^Q\|_{H^\mhalf_\Gamma}
&\leq C
\big(\|\phi\|_{H^s_\Gamma} + |G|_{2,\Hull(\Gamma)}\big)
 h^{s+1/2},
\end{align}
for some $C>0$ independent of $h$, $\phi$, and $G$, and by combining this with \eqref{eq:GalerkinBound}
we deduce \eqref{eq:PerturbErrorUniform}.
\end{proof}

Under a stronger condition on $h_Q$ we can also prove a superconvergence result for the fully discrete approximation $J^Q(\phi_N^Q)$, given by \eqref{eq:JQdef2}, to the linear functional $J(\phi)$ given by \eqref{eq:Jdef} (cf.~\cite[Thm.~4.2.18]{sauter-schwab11}).
The significance of this result for the computation of $u(x)$ and $u^\infty(\hat x)$ is as spelled out in and above Remark \ref{rem:extra_smoothness}.

\begin{cor} \label{cor:fullydiscrete2}
Suppose that $\Gamma$, $\phi$, $V_N=Y_h$, and $g$ are defined, and satisfy the same assumptions, as in Corollary \ref{cor:fullydiscrete}.
Let $J$ and $J^Q$ be defined by \eqref{eq:Jdef}
and \eqref{eq:JQdef2} and suppose that $\varphi$ satisfies the assumptions of Theorem \ref{thm:Quadrature}(ii).
Suppose also that $\zeta\in H_\Gamma^s$, where $\zeta$ is the solution of \eqref{eqn:Variational} with $-g$ replaced by $\varphi$, and suppose that
\begin{equation} \label{eqn:hQuadSuperGeneral}
h_Q \leq C'_Q h^{d+s+3/2},
\end{equation}
for some sufficiently small $C'_Q > 0$, 
independent of $h$ and $h_Q$.  
Then the approximated linear system is invertible and
\begin{align}\nonumber
& |J(\phi)-J^Q(\phi_N^Q)| \leq \\ \label{eq:superFD}
& \hspace{3ex} C\Big(\|\phi\|_{H^s_\Gamma}(\|\zeta\|_{H^s_\Gamma} + \|\varphi\|_{H^{1/2}(\Gamma_\infty)}+|V|_{2,\Hull(\Gamma)}) +(\|\varphi\|_{H^{1/2}(\Gamma_\infty)}+|V|_{2,\Hull(\Gamma)})|G|_{2,\Hull(\Gamma)}\Big)h^{2s+1},
\end{align}
for some constant $C>0$ independent of $h$, $\phi$, $\zeta$, $\varphi$, $G$, and $V$.

If, further, $\Gamma$ is homogeneous the above result holds with \eqref{eqn:hQuadSuperGeneral} replaced by the weaker condition
\begin{align}
\label{eqn:hQuadSuperUniform}
h_Q\leq C'_Q h^{d/2+s/2+5/4}.
\end{align}
\end{cor}

\begin{proof}
The two conditions \eqref{eqn:hQuadSuperGeneral} and \eqref{eqn:hQuadSuperUniform} can be written as
\begin{align}
\label{eqn:hQuadSuperBoth}
h_Q\leq C'_Q h^{\frac{d+s+1/2}{1+p}+1},
\end{align}
where $p$ is as in the proof of Corollary \ref{cor:fullydiscrete}.
We first note that, for every $C_Q>0$, since $s> -1/2$, \eqref{eqn:hQuadSuperBoth} implies \eqref{eqn:hQuadCondBoth} provided $C'_Q$ is sufficiently small. Thus, if \eqref{eqn:hQuadSuperBoth} holds with $C'_Q$ sufficiently small, then the approximated linear system is invertible by Corollary \ref{cor:fullydiscrete}. Next we estimate
\begin{align*}
&|J(\phi)-J^Q(\phi_N^Q)|\\
 &\leq  |J(\phi)-J(\phi_N)| + |J(\phi_N)-J(\phi_N^Q)| + |J(\phi_N^Q)-J^Q(\phi_N^Q)|\\
& \leq  c h^{2s+1}\|\phi\|_{H^{s}_\Gamma}\|\zeta\|_{H^{s}_\Gamma} + \|\varphi\|_{H^{1/2}(\Gamma_\infty)}\|\phi_N-\phi_N^Q\|_{H^{-1/2}_\Gamma} +C h_Q^2 h^{s+(n-d)/2} |V|_{2,\Hull(\Gamma)} \|\phi^Q_N\|_{H_\Gamma^{s}},
\end{align*}
by \eqref{eq:GalerkinBoundJ} and Theorem \ref{thm:Quadrature}(ii). We note moreover that $\|\phi^Q_N\|_{H_\Gamma^{s}}\leq C_1 (\|\phi\|_{H_\Gamma^{s}}+|G|_{2,\Hull(\Gamma)})$, for some constant $C_1>0$ independent of $\phi$, $G$, and $h$,
 as a consequence of Corollary \ref{cor:fullydiscrete} and Proposition \ref{prop:InvApp1}, so that \eqref{eqn:hQuadSuperBoth} implies that the last term in the above equation is
\begin{eqnarray*}
&\leq &CC_1(C'_Q)^2 h^{\frac{2}{1+p}(d+s+1/2)+2+s+(n-d)/2} |V|_{2,\Hull(\Gamma)} (\|\phi\|_{H_\Gamma^{s}}+|G|_{2,\Hull(\Gamma)})\\
& \leq &CC_1(C'_Q)^2 (\diam(\Gamma))^{s(\frac{2}{1+p}-1)+1+\frac{2}{1+p}(d+1/2)+(n-d)/2 } |V|_{2,\Hull(\Gamma)} (\|\phi\|_{H_\Gamma^{s}}+|G|_{2,\Hull(\Gamma)})h^{2s+1},
\end{eqnarray*}
recalling that $s>-1/2$ so that $s(\frac{2}{1+p}-1)+1+\frac{2}{1+p}(d+1/2)+(n-d)/2>0$ for both $p=0$ and $p=1$.
Further, arguing as in the proof of Corollary \ref{cor:fullydiscrete}, the conditions \eqref{eq:AmA} are satisfied with $E_A$ and $E_b$ given by \eqref{eq:EAbForm}, for some constant $C_*$ independent of $h$ and $h_Q$. If also \eqref{eqn:hQuadSuperBoth} holds then %
$$
E_A
\leq C_*(C'_Q)^{1+p}h^{s+1/2}
$$
and
\begin{align}
\label{}
E_b
&\leq C_*(C'_Q)^2 h^{\frac{2}{1+p}(d+s+1/2)+1-s+(n-d)/2} |G|_{2,\Hull(\Gamma)}\\\notag
& \leq C^*(C'_Q)^2 (\diam(\Gamma))^{(2-\frac{2}{1+p})(-s)+\frac{2}{1+p}(d+1/2)+1/2+(n-d)/2} |G|_{2,\Hull(\Gamma)} h^{s+1/2},
\end{align}
recalling that $s<-(n-d)/2$, so that   
$(2-\frac{2}{1+p})(-s)+\frac{2}{1+p}(d+1/2)+1/2+(n-d)/2>0$ for $p=0,1$. 
Thus, if $C'_Q$ is chosen sufficiently small so that $E_A\leq \alpha/(2c_{t_d}^2)$ for $0<h\leq \diam(\Gamma)$, it follows from Proposition \ref{prop:Strang} that
$$
\|\phi_N-\phi_N^Q\|_{H^{-1/2}_\Gamma} \leq C_2(\|\phi\|_{H_\Gamma^s} + |G|_{2,\Hull(\Gamma)})h^{2s+1},
$$
for some constant $C_2>0$ independent of $h$, $\phi$ and $G$. The bound
\eqref{eq:superFD} follows by combining the above inequalities.
\end{proof}

\begin{rem}[Reduced quadrature]
\label{rem:reducedquad}
The quadrature rule defined by \eqref{eq:OffDiagQuad} uses the same maximum mesh width $h_Q$ for all off-diagonal elements $A_{ij}$, $i\neq j$. Since the magnitude of the integrand $\Phi(x,y)$ and its derivatives blows up as one approaches the diagonal $x=y$, and decays away from it, it is possible to save computational effort, while maintaining the error bounds \eqref{eq:AquadEst2} and \eqref{eq:AquadEst1},
and hence the error estimates \eqref{eq:PerturbErrorUniform} and \eqref{eq:superFD}.
To achieve this we increase
the quadrature mesh width (and hence potentially %
decrease the number of quadrature points)
for the computation of $A_{ij}$ when the elements $\Gamma_{\bm(i)}$ and $\Gamma_{\bm(j)}$ are sufficiently well-separated spatially.
In more detail, from \cite[Prop 5.2]{HausdorffQuadrature} we have that, for $i\neq j$, the error in the quadrature approximation \eqref{eq:OffDiagQuad}, with $h_Q$ replaced by a local quadrature mesh width $h_{Q,i,j}$, satisfies
\[|A_{ij}-A_{ij}^Q|\leq C\mu_{\bm(i)}\mu_{\bm(j)}h_{Q,i,j}^2\Upsilon(R_{ij}),\]
for some constant $C$ independent of $i$ and $j$, where
\[\Upsilon(R):=\frac{1+(kR)^{n/2+1}}{R^{n+1}} \qquad \text{and} \qquad R_{ij}:= \dist(\Hull(\Gamma_{\bm(i)}),\Hull(\Gamma_{\bm(j)})).\]
Given $h_Q$, one can therefore maintain the bounds \eqref{eq:AquadEst2} and \eqref{eq:AquadEst1} by replacing $h_Q$ in \eqref{eq:OffDiagQuad} by
\[h_{Q,i,j} = h_Q\left(\frac{\max_{p\neq q}
\Upsilon(R_{pq})}{
\Upsilon(R_{ij})}\right)^{1/2} \geq h_Q. \]
If the quantities $R_{ij}$ are not known exactly, but satisfy
\begin{align}
\label{eq:deltabounds}
R^-_{ij}\leq R_{ij}\leq R^+_{ij},\qquad i\neq j,
\end{align}
for some known quantities $R^-_{ij}\geq 0$ and $R^+_{ij}<\infty$, then one can maintain \eqref{eq:AquadEst2} and \eqref{eq:AquadEst1} by replacing $h_Q$ in \eqref{eq:OffDiagQuad} by
\begin{align}
\label{eq:ReducedhQij}
h_{Q,i,j} = h_Q\max\left(\left(\frac{\max_{p\neq q}%
\Upsilon(R^+_{pq})}{
\Upsilon(R^-_{ij})}\right)^{1/2},1\right) \geq h_Q. 
\end{align}
Here we are using the fact that $\Upsilon(R)$ is a monotonically decreasing function of $R$, and the $\max(\cdot,1)$ is needed to ensure that $h_{Q,i,j}\geq h_Q$ (so that we are not increasing the computational effort unnecessarily) because the quantity inside the square root may be smaller than 1. (This holds in particular if $R^-_{ij}=0$, in which case we are interpreting $1/\Upsilon(R^-_{ij})$ as $1/\Upsilon(0)=1/\infty=0$.)

In our numerical results in \S\ref{sec:NumericalResults} we will compute $A_{ij}^Q$ as described above, choosing $h_{Q,i,j}$ according to \eqref{eq:ReducedhQij} with%
\begin{align}
\label{eq:ReducedRijminus}
R^-_{ij}:=\max(\dist(x_{\bm(i)},x_{\bm(j)})-\diam(\Gamma_{\bm(i)})-\diam(\Gamma_{\bm(j)}),0)
\end{align}
and
\begin{align}
\label{eq:ReducedRijplus}
R^+_{ij}:= \dist(x_{\bm(i)},x_{\bm(j)}),
\end{align}
which satisfy \eqref{eq:deltabounds}.
\end{rem}

\section{Numerical results}
\label{sec:NumericalResults}

In this section we present numerical results that illustrate our methods and assess the sharpness of our theoretical predictions.
The Hausdorff BEM \eqref{eq:cvec} has
been implemented in the Julia language \cite{julia} and the code is available at \url{https://github.com/AndrewGibbs/IFSintegrals}.
In all examples described in this section, we consider the scattering of plane waves, i.e.\ the datum $g$ is as in \eqref{eqn:gDefScatteringProblem} with $u^i(x)=e^{\ri k\vartheta\cdot x}$ and $|\vartheta|=1$. 
We validate our implementation against a different method in \S\ref{sec:lebesgue}, but in the rest of our experiments we use as a reference solution a more accurate Hausdorff-BEM solution with a large number of degrees of freedom $\Nref$. Most of our experiments are for homogeneous attractors, in which case $\Nref=M^{\ellref}$ for some $\ellref\in\N$. For Cantor sets $(M=2)$ we choose $\ellref=15$, so that $\Nref=32768$. For Cantor Dusts ($M=4$) we choose $\ellref=8$, so that $\Nref=65536$.
Since we do not use any matrix compression, the memory required to store the Galerkin matrices grows like $N^2$, with  $N = M^{\ell}$ when the attractor is homogeneous. Thus, while many of our experiments were run on a standard laptop, some (in particular, the calculation of reference solutions) 
required the use of the Myriad High Performance Computing Facility available at University College London, which has computing nodes available with 1.5TB of RAM.

Our implementation uses the quadrature rules described in \S\ref{sec:Quadrature}. Precisely, we approximate the right-hand side in the linear system and linear functionals of the solution (the scattered field and far field) using the quadrature rules \eqref{eq:bi} and \eqref{eq;JQdef}, respectively. We approximate the diagonal matrix elements $A_{ii}$ by \eqref{eq:Aii}-\eqref{eq:AiiSing}. To approximate $A_{ij}$ with $i\neq j$ we use \eqref{eq:OffDiagQuad} with $h_Q$ replaced by $h_{Q,i,j}\geq h_Q$ given (in terms of $h_Q$) by \eqref{eq:ReducedhQij}-\eqref{eq:ReducedRijplus}. All these quadrature rules depend on the parameter $h_Q$. We choose $h_Q=C_Qh$, 
where $C_Q>0$ is a constant independent of $h$. 
While the requirement $h_Q=C_Q h$ is weaker (i.e.\ it requires fewer quadrature points) than the conditions $h_Q\le C_Q h^{1+d/2}$ (see\ \eqref{eqn:hQuadCondUniform}) and $h_Q\le C_Q h^{1+d}$ (see \eqref{eqn:hQuadCondGeneral}) required by our theory, 
our numerical experiments suggest that, in practice, $h_Q=C_Q h$ is sufficient to achieve 
our theoretical convergence rates,  
even when using the reduced quadrature of Remark \ref{rem:reducedquad}. Except where indicated otherwise, in the simulations reported below for homogeneous attractors we use $C_Q=\rho^2$ for cases with $k<20$,  $C_Q=\rho^4$ when $20\leq k\leq 50$.

We measure the accuracy of our BEM solutions in the $H^{-1/2}_\Gamma$ norm. These $H^{-1/2}_\Gamma$ norms are computed by expressing them in terms of a single-layer BIO with wavenumber $k=\ri$, which we denote 
$S^\ri$, as in  
\cite[Table 1]{BEMfract}. Practically, we achieve this by assembling an approximation 
$A^{\ri,Q}$ to the Galerkin matrix $A^\ri$ for the operator $S^\ri$, analogous to \eqref{eq:GalInt}, with $k=\ri$ and $\Nref$ degrees of freedom, using quadrature approximations analogous to 
\eqref{eq:Aii}-\eqref{eq:AiiSing} and \eqref{eq:OffDiagQuad}, choosing $h_Q$ the same as for the reference solution, with the reduced quadrature formulae \eqref{eq:ReducedhQij}-\eqref{eq:ReducedRijplus} applied as if $k=1$. (While the quadrature convergence analysis in \cite[\S5]{HausdorffQuadrature} was presented only for $k>0$, one can check that the relevant results also extend, mutatis mutandis, to the case $k=\ri$.) 
Next, we 
view $\phi_N^Q$ as an element of the larger space $V_{\Nref}$  
and define $\vv$ as the coefficient vector of $\phi_N^Q-\phi_{\Nref}^Q$ in $V_{\Nref}$. Then, arguing as in \cite[Table 1]{BEMfract} (and see Footnote \ref{footnote:Si}), it follows that
	\[
	\|\phi_N^Q-\phi_{\Nref}^Q\|_{H^{-1/2}_\Gamma}^2
	=2\left\langle S^\ri(\phi_N^Q-\phi_{\Nref}^Q),\phi_N^Q-\phi_{\Nref}^Q\right\rangle_{H^{1/2}(\Gamma_\infty) \times H^{-1/2}(\Gamma_\infty)}
	= 2\vv^HA^\ri \vv\approx 2\vv^HA^{\ri,Q} \vv.
	\]

\subsection{\texorpdfstring{$H^{-1/2}_\Gamma$}{Fractional}-norm convergence, Cantor set (\texorpdfstring{$n=1$}{n=1})}
\label{s:exp:Cantor}

\begin{figure}
\includegraphics[width=.48\linewidth]{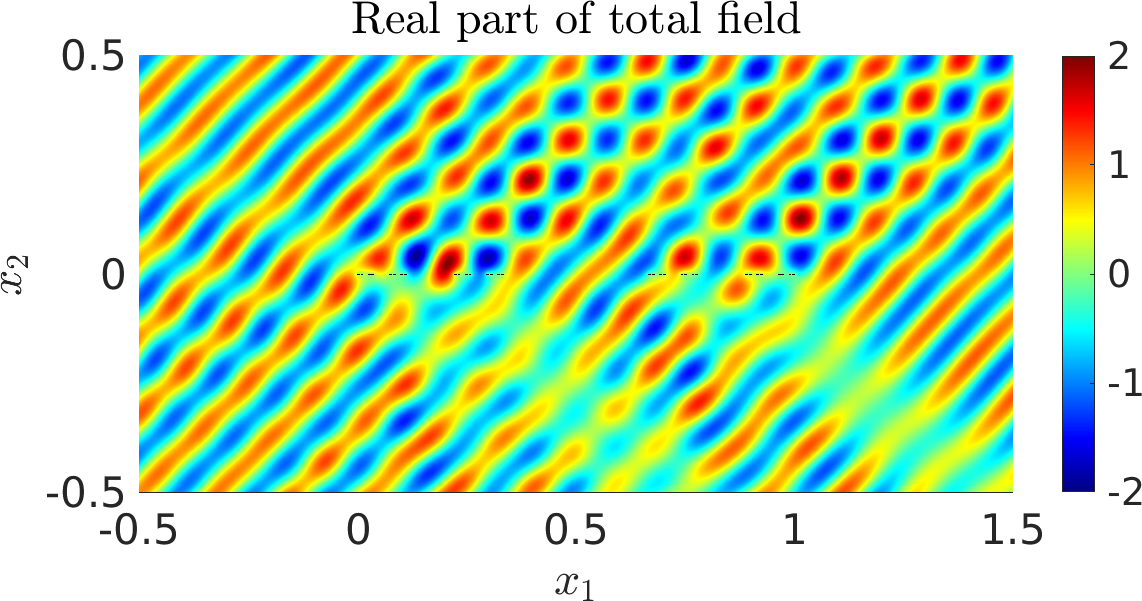}
\hspace{.04\linewidth}
\includegraphics[width=.48\linewidth]{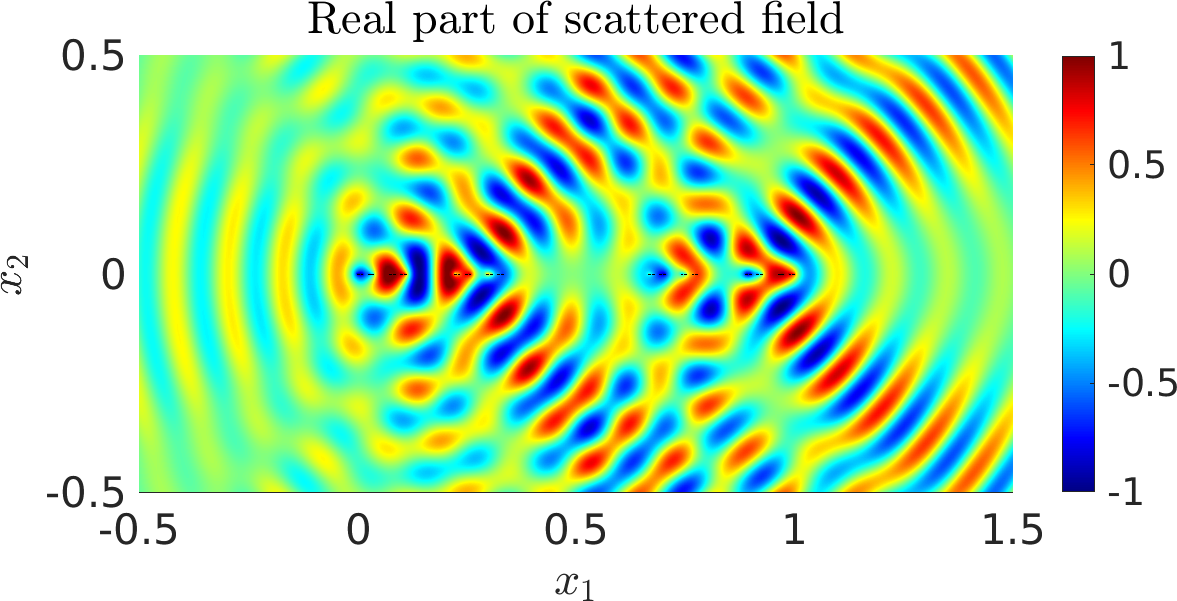}
\caption{Scattering in $\R^2$ by a middle-third Cantor set screen (a subset of the line segment $[0,1]\times\{0\}$) %
with wavenumber $k=50$ and incident plane wave direction $\vartheta=(1,-1)/\sqrt2$, showing the total (left panel) and scattered (right panel) fields computed using our Hausdorff BEM with $\ell=10$ and quadrature parameters as in \S\ref{s:exp:Cantor}. 
}
\label{fig:CantorSetFields}
\end{figure}

We consider first the case where $n=1$ (corresponding to scattering in $\R^2$) and $\Gamma\subset\Gamma_\infty\cong \R$ is a ``middle-$(1-2\rho)$ Cantor set'' screen (e.g., \cite[p.~71, eqn.~(91)]{Fal}), for some $0<\rho<1/2$. Precisely, $\Gamma$ is the attractor of the disjoint homogeneous IFS with $M=2$,
\begin{equation} \label{eq:CS_IFS}
s_1(x)=\rho x, \qquad s_2(x) = 1-\rho +\rho x, \qquad x\in\R,
\end{equation}
a $d$-set with $d=\dimH(\Gamma)=\log{2}/\log{(1/\rho)}$.
We denote by $\phi_\ell$ the Hausdorff-BEM solution ``at level $\ell$'', for $\ell\in\N_0$, i.e.\ with $h=\rho^\ell$ and $N=2^\ell$.

In Figure~\ref{fig:CantorSetFields} we show an example of the total and scattered fields for the middle-third case ($\rho=1/3$). In Figure~\ref{fig:CantorSetH12} we present plots of the relative errors in the Hausdorff-BEM solution for $\ell=0,1,\ldots,14$, measured in $H^\mhalf_\Gamma$ norm, using $\phi_{15}$ as reference solution, for various $\rho$ and $k$ values.
In all cases the incident wave is the plane wave $u^i(x)=\re^{\ri k\vartheta\cdot x}$ with $\vartheta=(1,-1)/\sqrt2$.
The experiments in Figures~\ref{fig:CantorSetFields} and \ref{fig:CantorSetH12} are carried out, as indicated in the second paragraph of this section,  using the quadrature rules described in \S\ref{sec:Quadrature}, with the reduced quadrature of Remark \ref{rem:reducedquad}, and with $h_Q=\rho^2 h$ for $k=0.1$ and $h_Q=\rho^4 h$ for $k=50$.
This means that the off-diagonal entries of the BEM matrix require up to $16$ and $256$ evaluations of the Helmholtz fundamental solution %
for $k=0.1$ and $k=50$, respectively.

In Remark \ref{rem:ConvRates} we noted that our theoretical analysis suggests we should expect convergence approximately like $2^{-\ell/2}$ for the BEM solution, %
which is what we observe in the numerical results
(for each error convergence plot we also plot $C2^{-\ell/2}$ -- the black dashed lines -- choosing the constant $C$ so that the two curves coincide for the largest value of $\ell$).
Higher wavenumbers give different pre-asymptotic behaviour (in particular, the approximation is not accurate for small values of $\ell$) but do not affect the asymptotic rates for large $\ell$, see Figure \ref{fig:CantorSetH12}.

\begin{figure}[tbp!]
\includegraphics[width=0.49\textwidth]%
{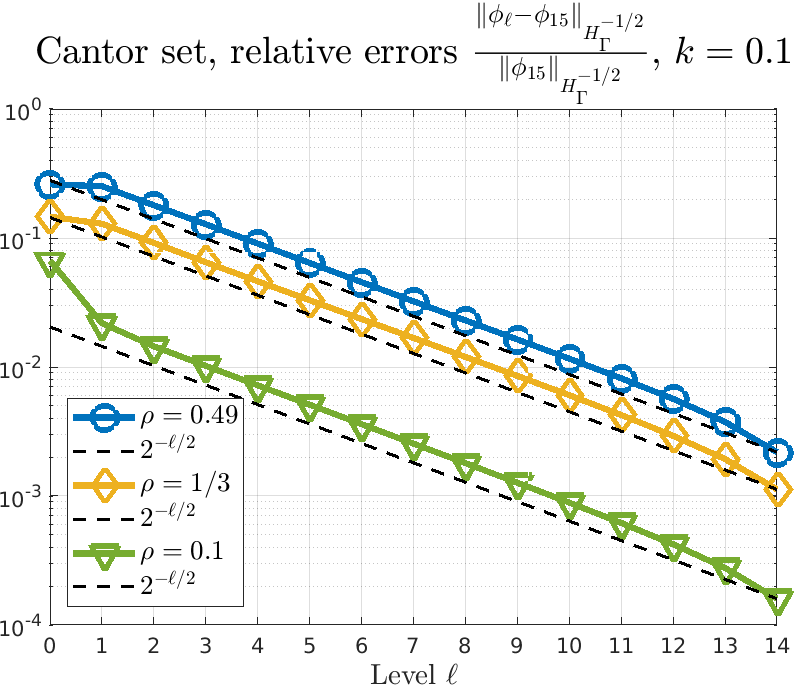}
\includegraphics[width=0.49\textwidth]%
{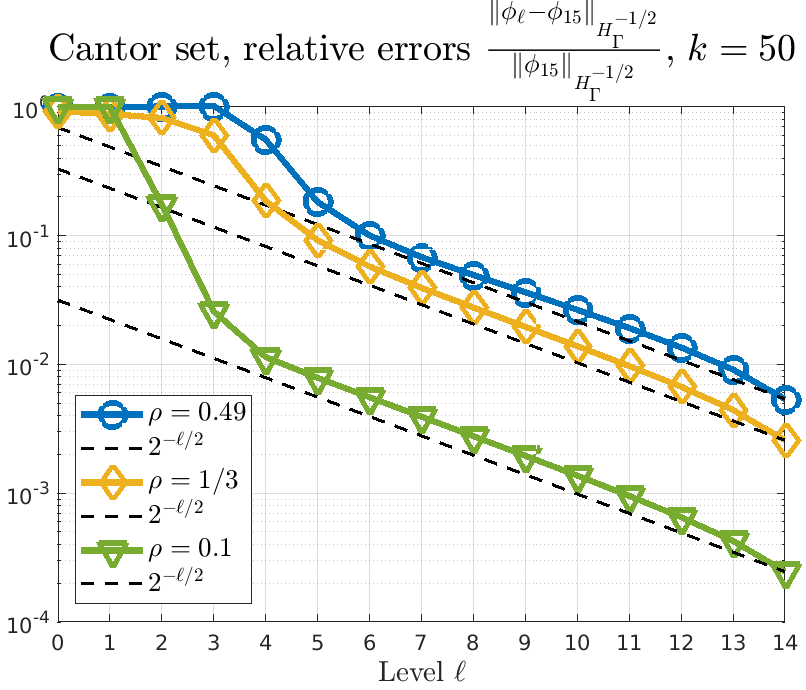}
\caption{$H^{-1/2}_\Gamma$ relative errors for Hausdorff-BEM solutions on Cantor sets.}
\label{fig:CantorSetH12}
\end{figure}

For $\rho=0.5$ our  Hausdorff BEM coincides with classical piecewise-constant BEM on a uniform mesh applied to the  Lipschitz screen that is the unit interval $[0,1]$; our implementation uses  the special choice  \eqref{eq:Aii}-\eqref{eq:AiiSing} of quadrature rules to evaluate the matrix entries (which is simply the composite midpoint rule for the off-diagonal  entries).
Running our Hausdorff-BEM code with $\rho=0.5$ we observe error curves (not reported here) that are almost identical to those for $\rho=0.49$. Hence, at least in this case, the piecewise-constant Hausdorff-BEM approximation of the integral equation on a fractal (the Cantor set with $\rho=0.49$) is no less accurate, for the same number of degrees of freedom, than a classical piecewise-constant BEM approximation on an adjacent, more regular set (the interval $[0,1]$).

\subsection{\texorpdfstring{$H^{-1/2}_\Gamma$}{Fractional}-norm convergence, Cantor dust (\texorpdfstring{$n=2$}{n=2})}
\label{s:exp:Dust}

We consider next the case where $n=2$ (corresponding to scattering in $\R^3$) and $\Gamma\subset\Gamma_\infty\cong \R^2$ is a ``middle-$(1-2\rho)$ Cantor dust'' for some $0<\rho<1/2$,
defined by the disjoint homogeneous IFS with $M=4$,
\begin{align}
\label{eq:CD_IFS}
\begin{split}
s_1(x,y)=\rho (x,y), \qquad &s_2(x,y) = (1-\rho,0) +\rho (x,y),\\
s_3(x,y) = (1-\rho,1-\rho) +\rho (x,y), \qquad &s_4(x,y) = (0,1-\rho) +\rho (x,y), \qquad (x,y)\in\R^2,
\end{split}
\end{align}
a $d$-set with $d=\dimH(\Gamma)=\log{4}/\log{(1/\rho)}$, see Figure~\ref{Fig:Dust}.
Such a screen generates a non-zero scattered field if and only if $\rho>1/4$ (see \cite[Example 8.2]{ScreenPaper} and the discussion before Lemma \ref{lem:Nullity}).
In Figure~\ref{fig:CantorDustH12} we present results similar to those in Figure \ref{fig:CantorSetH12}, albeit for a more restricted range of $\ell$ and a lower value of $k$ in the right-hand panel, due to the increased computational cost associated with the change from $M=2$ (Cantor set) to $M=4$ (Cantor dust).
The Hausdorff-BEM solution $\phi_\ell$ corresponds to a mesh of $N=4^\ell$ elements of diameter $h=\sqrt2\rho^\ell$.
The incident wave is the plane wave $u^i(x)=\re^{\ri k\vartheta\cdot x}$ with $\vartheta=(0,1,-1)/\sqrt2$.
As indicated at the beginning of this section, we use the quadrature rules described in \S\ref{sec:Quadrature}, with the reduced quadrature of Remark \ref{rem:reducedquad}, and with $h_Q=\rho^2 h$ for both $k=0.1$ and $k=5$. As a result, the off-diagonal entries of the BEM matrix require up to $256$ evaluations of $\Phi$.

\begin{figure}[tbp!]
\includegraphics[width=0.49\textwidth]%
{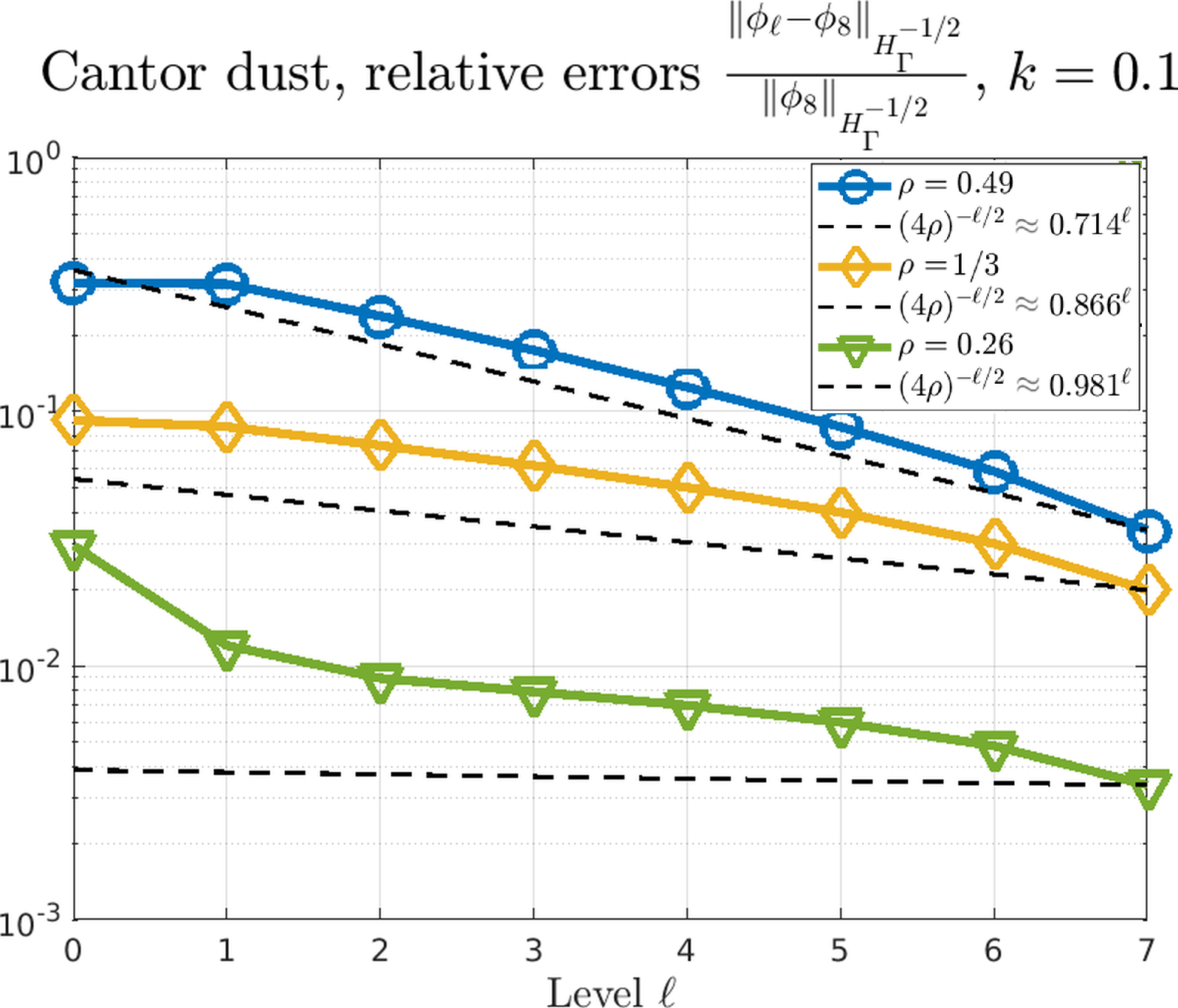}
\hspace{2mm}
\includegraphics[width=0.49\textwidth]%
{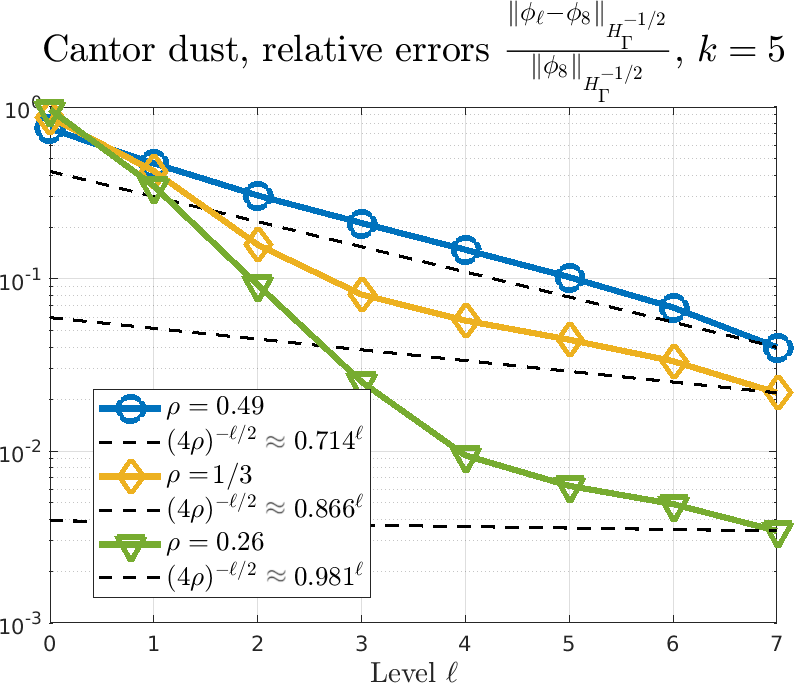}
\caption{$H^{-1/2}_\Gamma$ relative errors for Hausdorff-BEM solutions on Cantor dusts.
}
\label{fig:CantorDustH12}
\end{figure}

In Remark \ref{rem:ConvRates} we noted that our theoretical analysis suggests we should expect convergence approximately like $(4\rho)^{-\ell/2}$, whereas the convergence we observe in Figure~\ref{fig:CantorDustH12} appears to be faster than this, especially for lower values of $\rho$.
In fact, these numerical results are consistent with the theoretical bounds being sharp, 
as we now explain.
The issue is that the convergence rate $(4\rho)^{-\ell/2}$ is slow, so the reference solution (with $\ell=8$) used in Figure~\ref{fig:CantorDustH12} is still quite far from the exact solution $\phi$ of the integral equation, which affects the shape of the error plot.
Let us assume that the Galerkin error bound \eqref{eq:ConvRates} holds with $\epsilon=0$ and in fact is exact, i.e.\
\begin{equation} \label{eq:ErrBEx}
\|\phi-\phi_\ell\|_{H^\mhalf_\Gamma}=C \mathfrak{a}^\ell, \quad \ell\in \N,
\end{equation}
for $\mathfrak{a}:=(4\rho)^{-1/2}\in(1/\sqrt2,1)$ and some $C>0$.
Then, by the triangle inequality, %
\begin{equation}\label{eq:NumericsIncrement}
C \mathfrak{a}^{\ell} (1-\mathfrak{a}^{\ell'-\ell})
\le\|\phi_\ell-\phi_{\ell'}\|_{H^\mhalf_\Gamma}
\le C \mathfrak{a}^{\ell} (1+\mathfrak{a}^{\ell'-\ell}),
\qquad  \ell,\ell'\in\N_0 \;\text{with}\; \ell<\ell',
\end{equation}
and, for any given $\ell'\in \N$, \eqref{eq:ErrBEx} imposes no constraint on $\|\phi_\ell-\phi_{\ell'}\|_{H^\mhalf_\Gamma}$, for $\ell=0,\ldots,\ell'-1$, beyond  \eqref{eq:NumericsIncrement}; in particular, it may be that either the lower bound or the upper bound in \eqref{eq:NumericsIncrement} is attained.
In Figure~\ref{fig:CantorDustBounds} we plot $\|\phi_\ell-\phi_8\|_{H^\mhalf_\Gamma}$, for $\ell=1,\ldots,7$, together with the upper and lower bounds of \eqref{eq:NumericsIncrement} for a suitable $C$ and $\ell'=8$. This plot makes clear that the values we compute for $\|\phi_\ell-\phi_8\|_{H^\mhalf_\Gamma}$ are consistent with \eqref{eq:NumericsIncrement} and so with \eqref{eq:ErrBEx}. %
(This reasoning, i.e. the bounds \eqref{eq:NumericsIncrement}, also explains the little ``dip'' at the right endpoints of the error curves for the Cantor set experiments in Figure~\ref{fig:CantorSetH12}.)

In Figure~\ref{fig:CantorDustIncrements} we plot, as a better test of the sharpness of our theory, the norm $\|\phi_\ell-\phi_{\ell+1}\|_{H^\mhalf_\Gamma}$, for $\ell=0,\ldots,7$, of the difference between Galerkin solutions at consecutive levels $\ell$. We observe, for each choice of $\rho$ and $k$, that, for large enough $\ell$ (how large depending on $\rho$ and $k$), $\|\phi_\ell-\phi_{\ell+1}\|_{H^\mhalf_\Gamma}= c\mathfrak{a}^\ell=c(4\rho)^{-\ell/2}$, for some constant $c>0$, also depending on $\rho$ and $k$, in agreement with \eqref{eq:NumericsIncrement} with $\ell'=\ell+1$. Moreover, since $\|\phi-\phi_\ell\|_{H^\mhalf_\Gamma}\to 0$ as $\ell\to\infty$ by Theorem \ref{thm:BEMConvergence}, this observed convergence implies that
$$
\|\phi-\phi_\ell\|_{H^\mhalf_\Gamma} \leq \sum_{m=\ell}^\infty \|\phi_m-\phi_{m+1}\|_{H^\mhalf_\Gamma} =
\frac{c}{1-\mathfrak{a}} \, \mathfrak{a}^\ell,
$$
confirming, experimentally, the convergence rates for the Cantor dust case predicted in Remark \ref{rem:ConvRates}, and confirming the predicted dependence of these convergence rates on $\rho$.

\begin{figure}[tb!]
\includegraphics[width=0.33\textwidth,clip,trim=30 0 30 0]{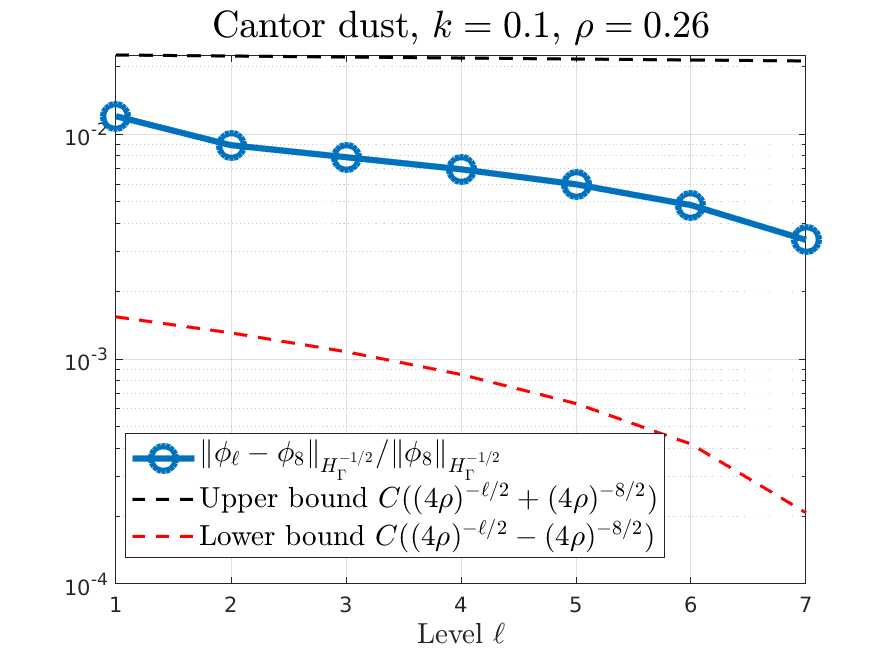}
\includegraphics[width=0.33\textwidth,clip,trim=30 0 30 0]{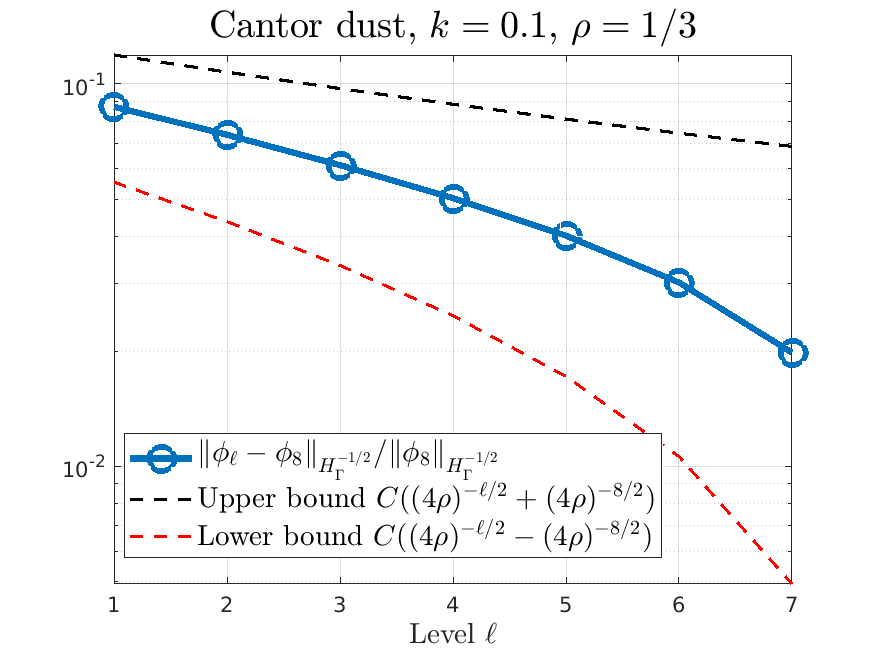}
\includegraphics[width=0.33\textwidth,clip,trim=30 0 30 0]{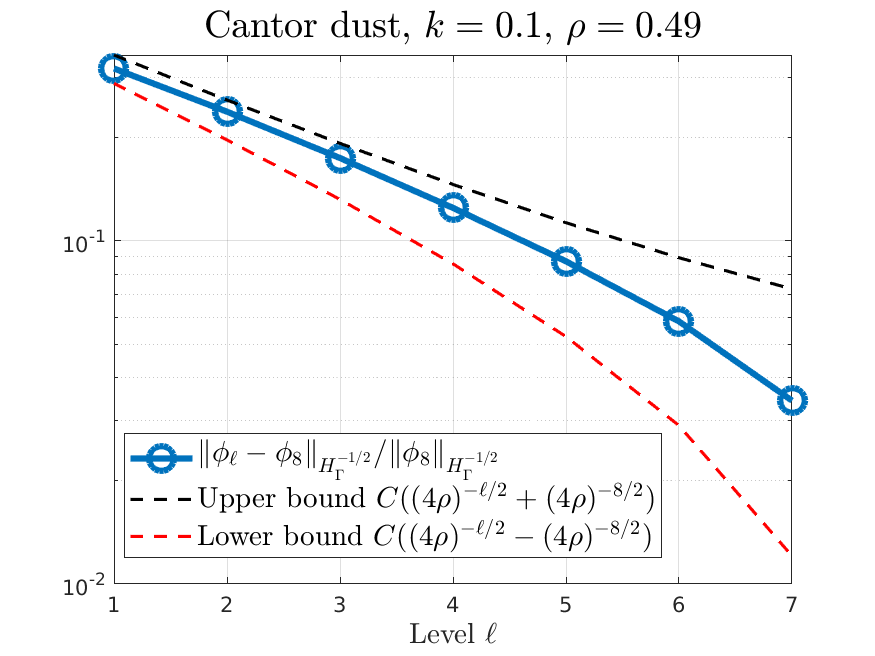}
\caption{Relative $H^{-1/2}_\Gamma$ errors for Hausdorff-BEM solutions on Cantor dusts  using $\phi_8$ as the reference solution, and the  bounds from \eqref{eq:NumericsIncrement} with $\ell'=8$, for $k=0.1$ and different contraction factors $\rho$.}
\label{fig:CantorDustBounds}
\end{figure}

\begin{figure}[tb!]
\includegraphics[width=0.49\textwidth]%
{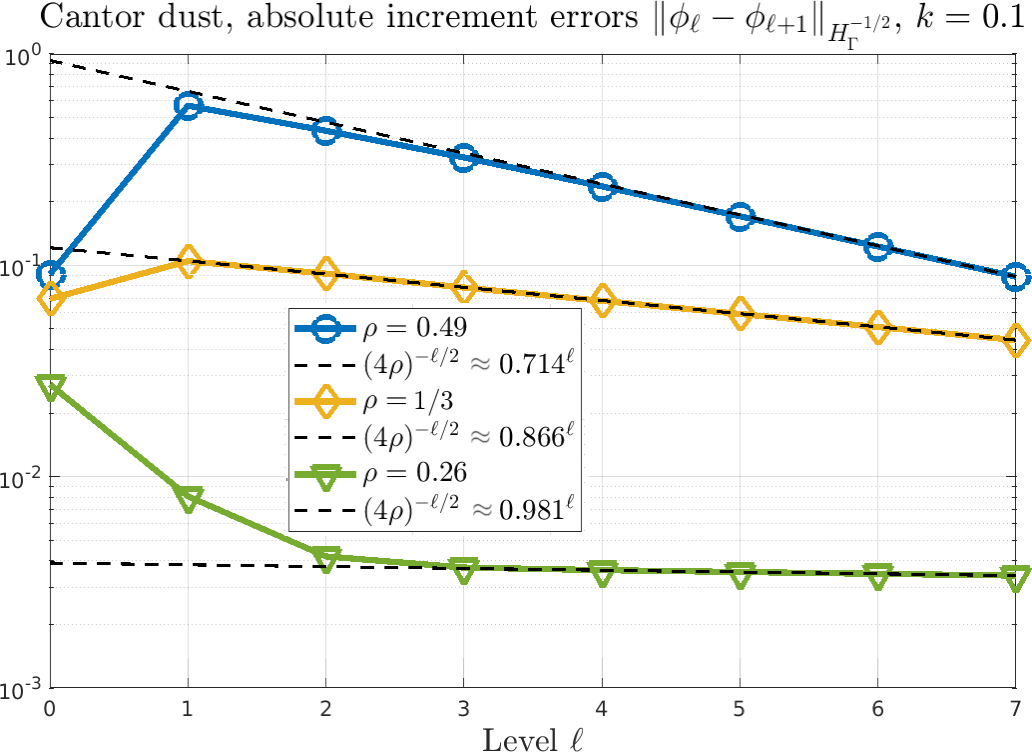}
\hspace{2mm}
\includegraphics[width=0.49\textwidth]%
{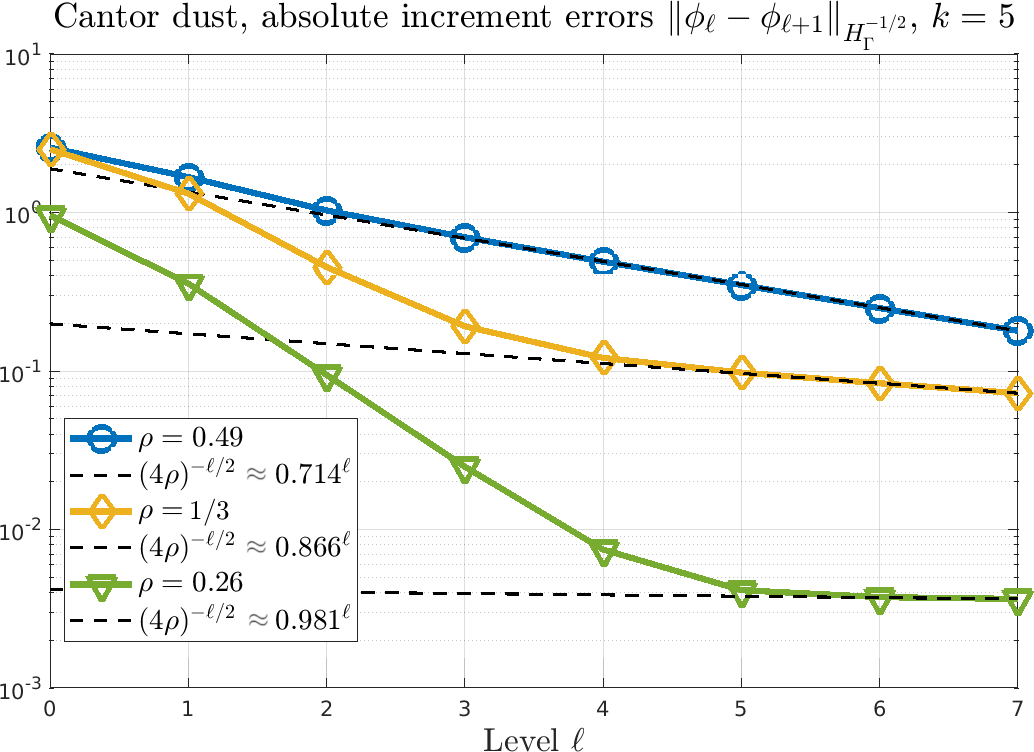}
\caption{$H^{-1/2}_\Gamma$ errors for Hausdorff-BEM solutions on Cantor dusts: the norm of the difference $\phi_\ell-\phi_{\ell+1}$ between consecutive levels decays like $(4\rho)^{-\ell/2}$.}
\label{fig:CantorDustIncrements}
\end{figure}

For $\rho=0.5$ our  Hausdorff BEM coincides with a piecewise-constant BEM on a uniform mesh applied to the screen $[0,1]^2$; our implementation uses the quadrature rules \eqref{eq:Aii}-\eqref{eq:AiiSing} (which reduce to a composite product midpoint rule for the off-diagonal  entries).
As similarly reported for the Cantor set in \S\ref{s:exp:Cantor}, running our Hausdorff-BEM code with $\rho=0.5$ we observe error curves (not reported here) that are almost identical to those for $\rho=0.49$. %

For this Cantor dust example we also compare numerical results with the theoretical predictions of Theorem~\ref{thm:MatrixBounds}. In Figure \ref{fig:MatrixNorms} we plot the norms of the Hausdorff-BEM Galerkin matrix $A$ and its inverse against $\ell$ (the matrix dimension is $N=4^\ell$) for $\rho=1/3$ and different values of $k$.
We observe that $\|A\|_2$ is approximately independent of $\ell$ and $h=\sqrt{2}\rho^\ell$, and that, for $\ell\geq 2$, $\|A^{-1}\|_2\approx (\rho M)^\ell=(4\rho)^\ell=\rho^{(1-d)\ell}=c^* h^{1-d}=c^*h^{-2t_d}$, where $c^*=2^{(d-1)/2}$ and $t_d$ is defined by \eqref{eq:tdDef}, in agreement with Theorem~\ref{thm:MatrixBounds}. (We note that $-2t_d = 1-d =1-\log4/\log3 \approx -0.262$.)

\begin{figure}[tbp!]
\centering
\includegraphics[width=0.95\textwidth,clip,trim=50 0 50 0]{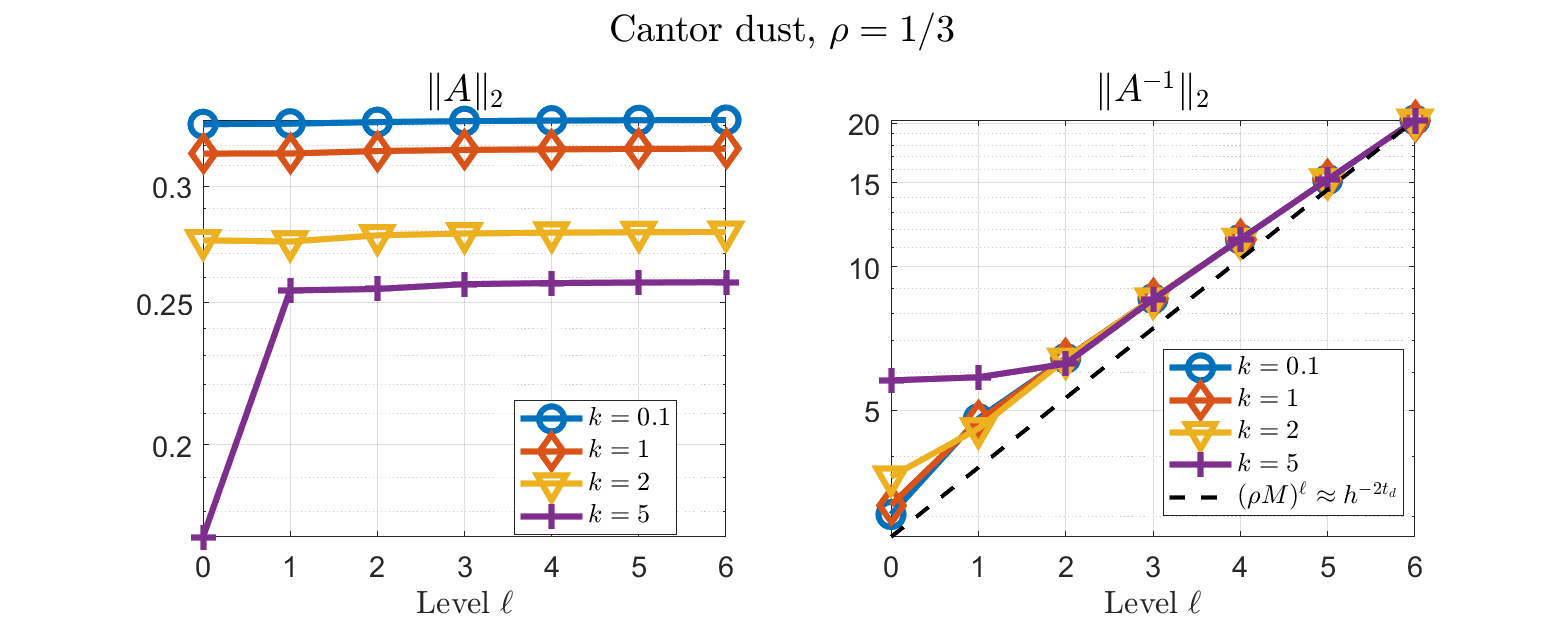}
\caption{The norm of the Hausdorff-BEM Galerkin matrix and its inverse. Here $-2t_d \approx -0.262$.}
\label{fig:MatrixNorms}
\end{figure}

\subsection{Convergence of the scattered field}
\label{s:exp:NearFar}

\begin{figure}[tbp!]\centering
\includegraphics[width=0.47\textwidth,clip,trim=30 0 30 0]{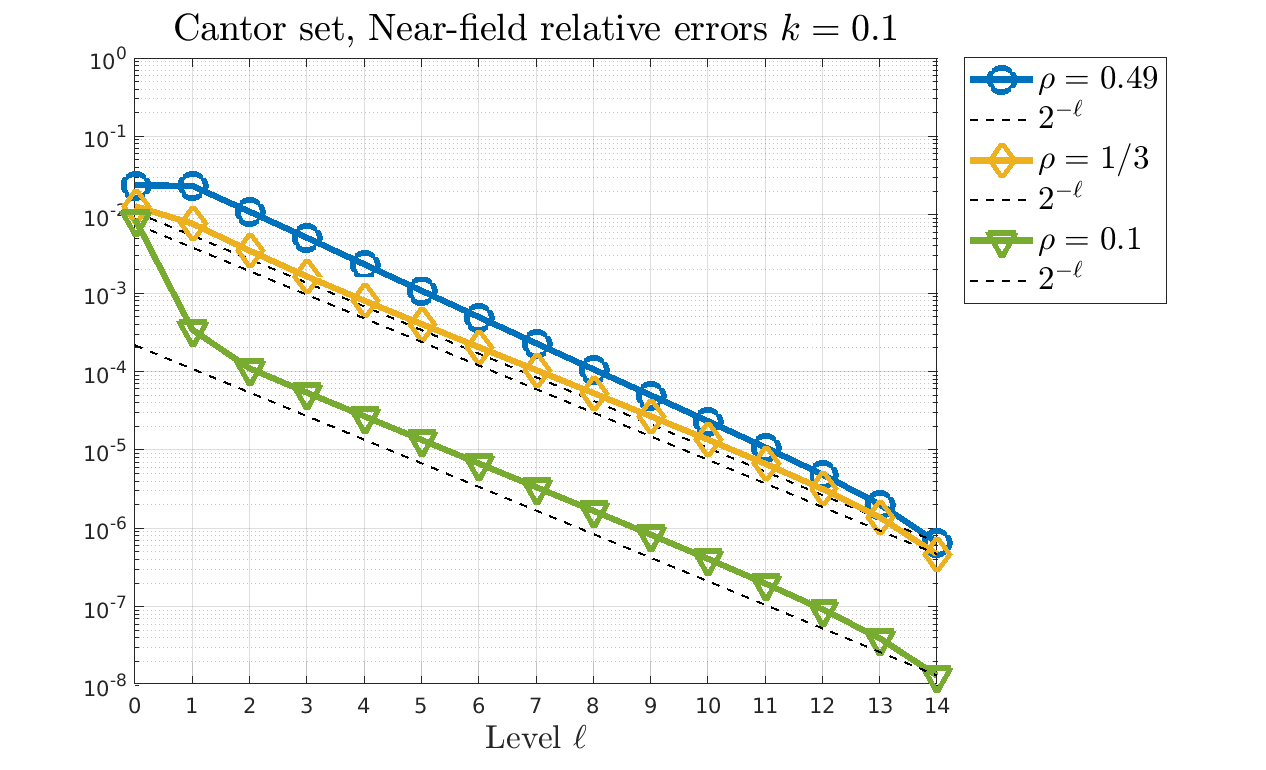}
\includegraphics[width=0.47\textwidth,clip,trim=30 0 30 0]{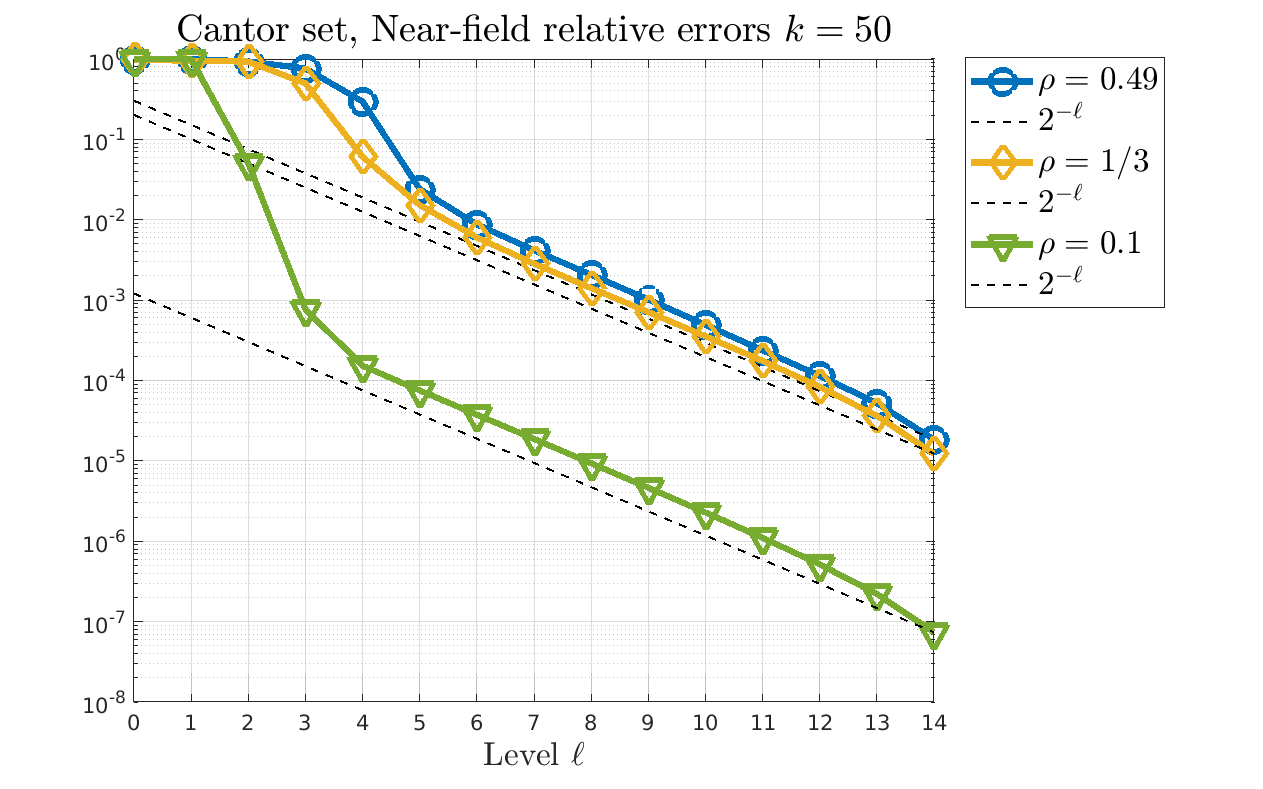}
\\
\includegraphics[width=0.47\textwidth,clip,trim=30 0 30 0]{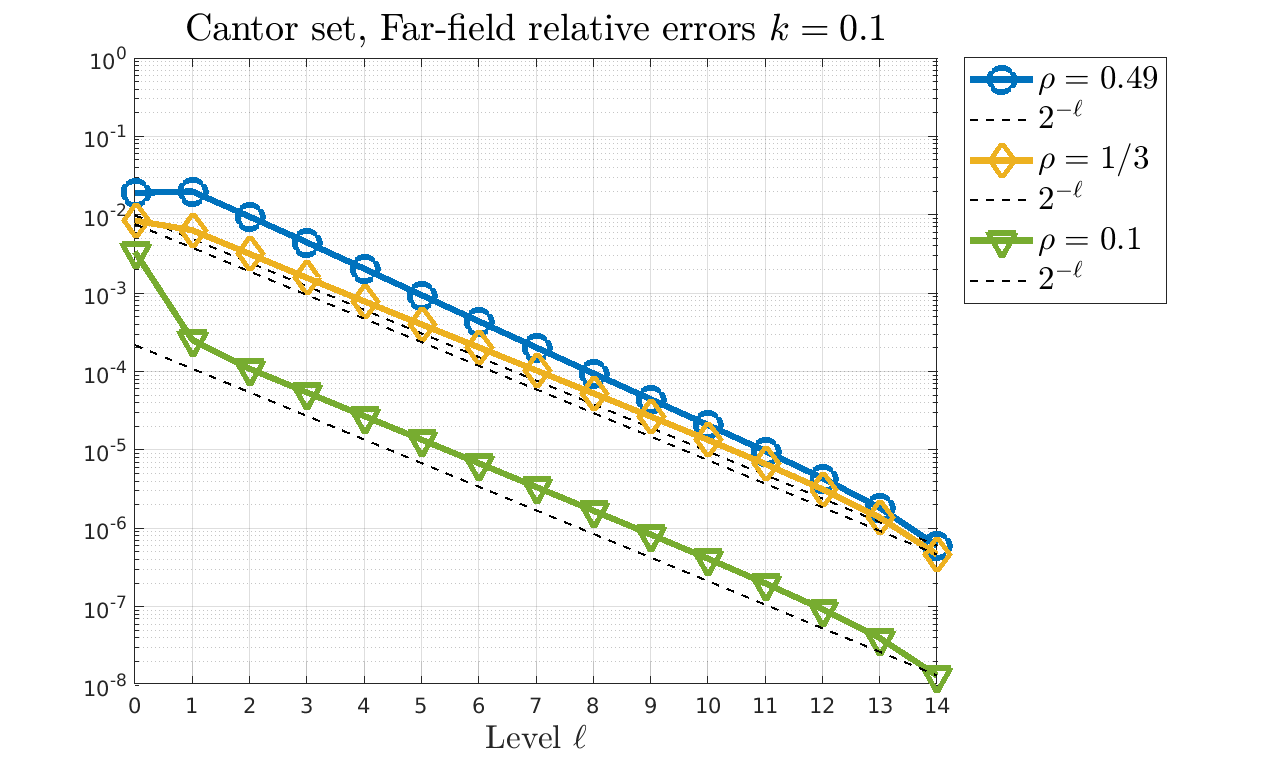}
\includegraphics[width=0.47\textwidth,clip,trim=30 0 30 0]{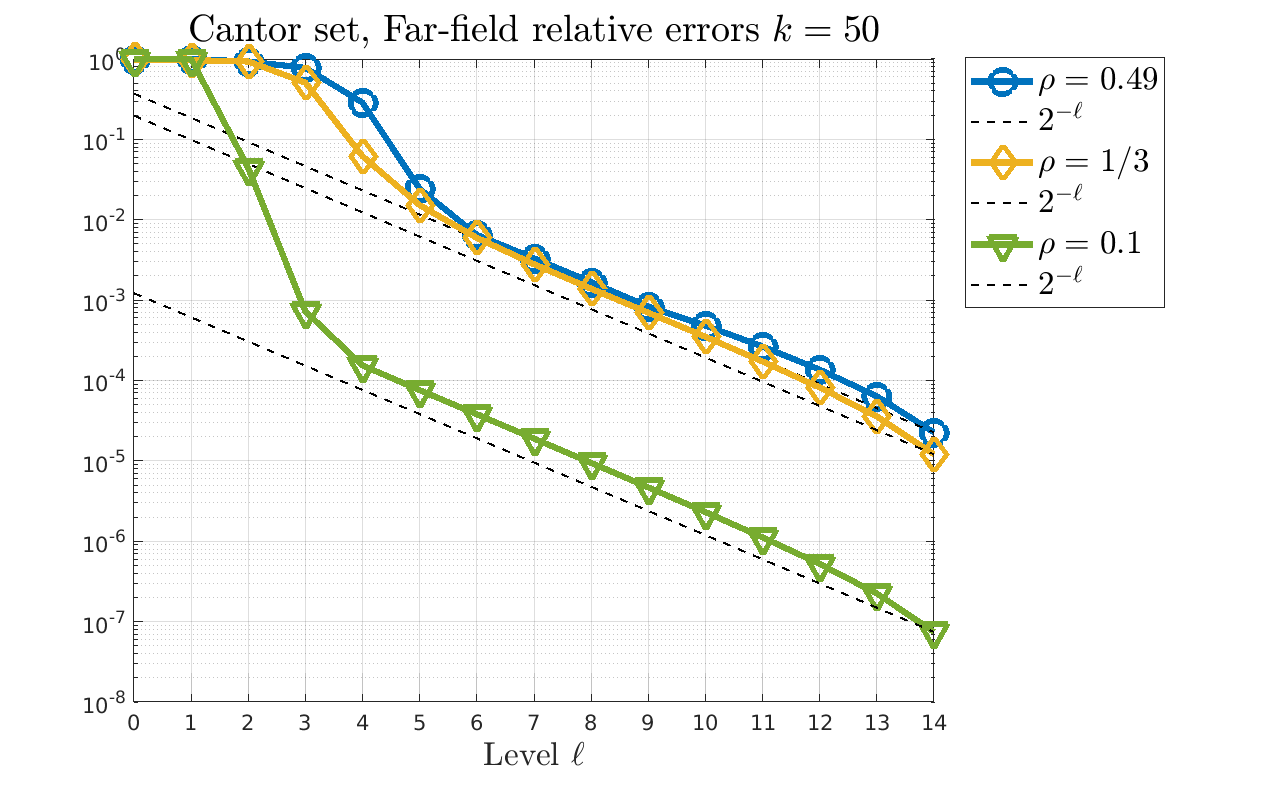}
\\
\includegraphics[width=0.47\textwidth,clip,trim=30 0 30 0]{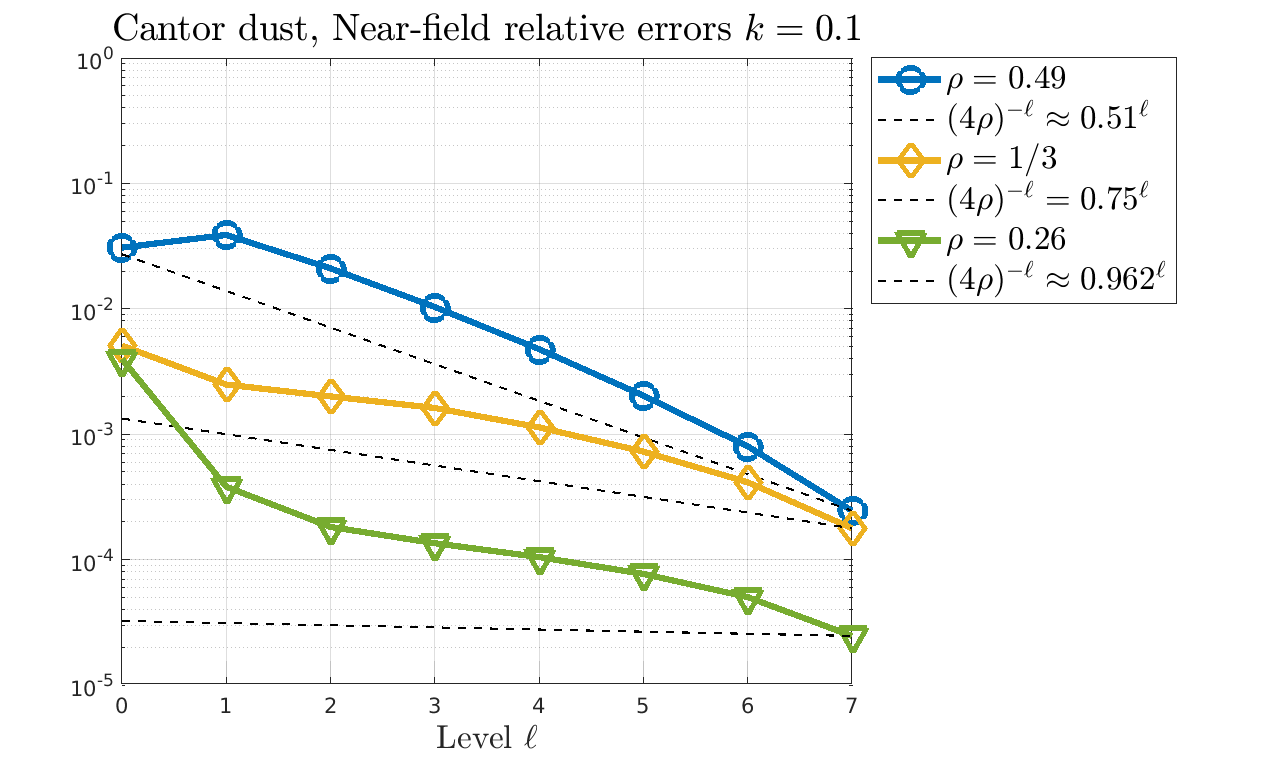}
\includegraphics[width=0.47\textwidth,clip,trim=30 0 30 0]{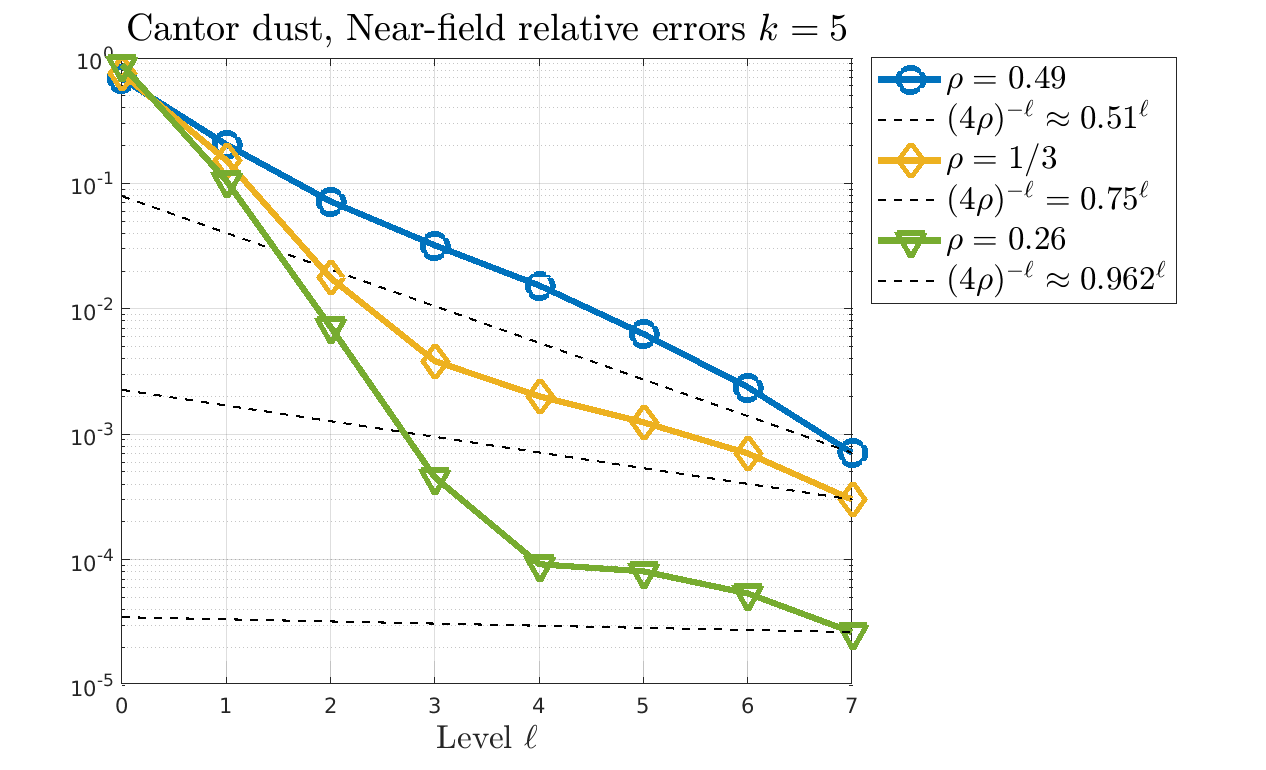}
\\
\includegraphics[width=0.47\textwidth,clip,trim=30 0 30 0]{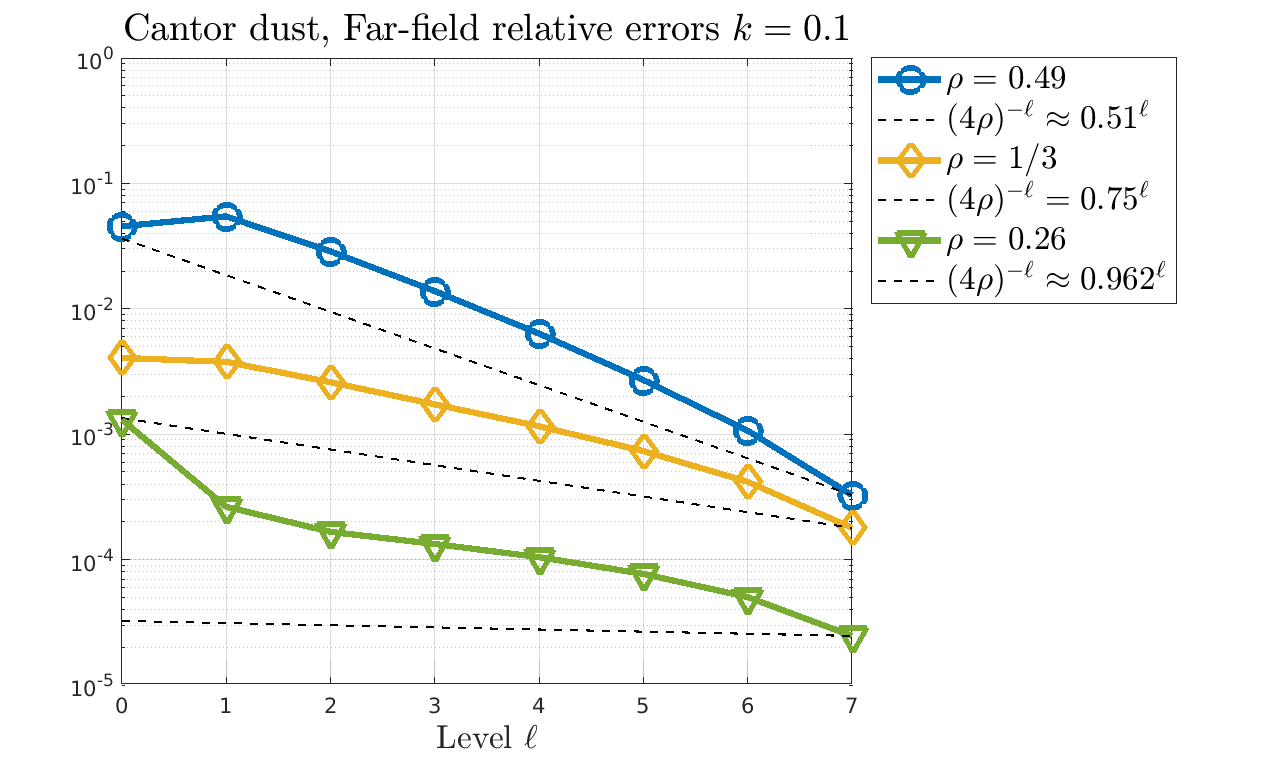}
\includegraphics[width=0.47\textwidth,clip,trim=30 0 30 0]{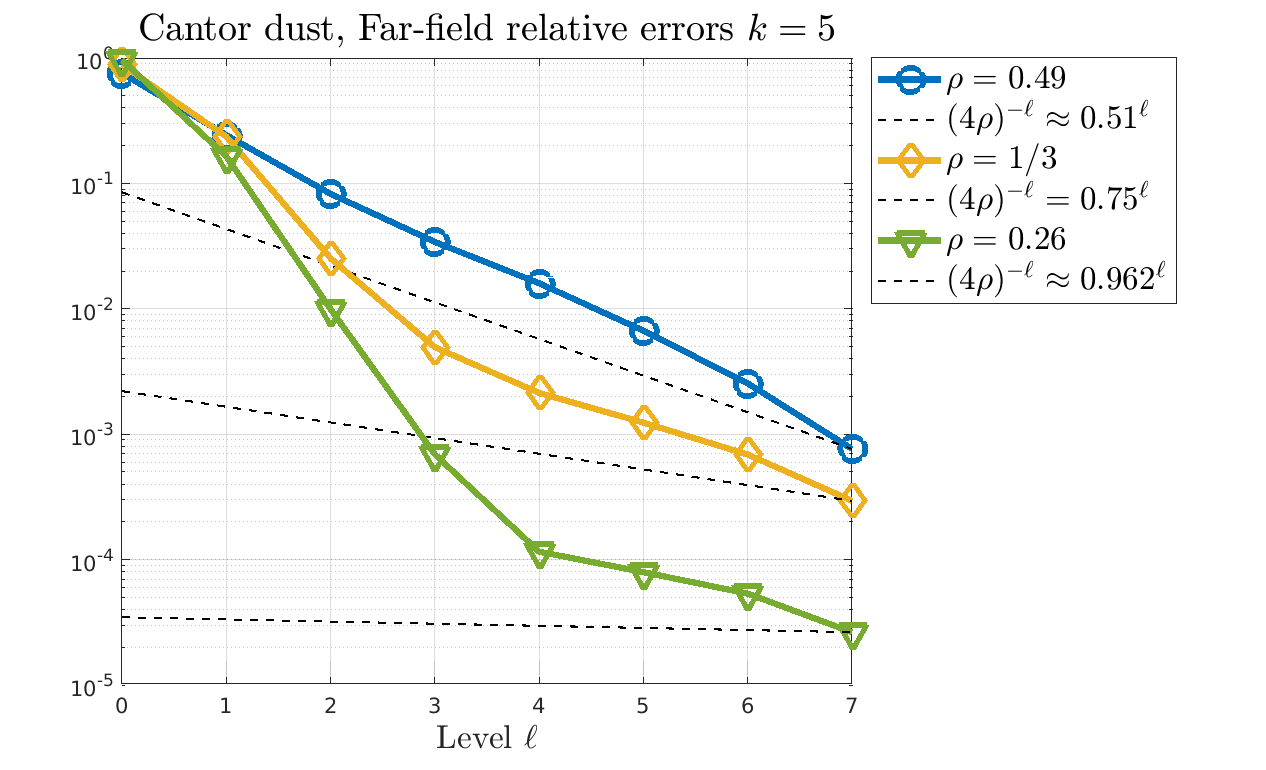}
\\

\caption{Convergence of the scattered near- and far-field for Cantor sets and dusts, measured as $\|u_\ell-u_{\ell_{\mathrm{ref}}}\|_{L_\infty}/\|u_{\ell_{\mathrm{ref}}}\|_{L_\infty}$ and $\|u^\infty_\ell-u^\infty_{\ell_{\mathrm{ref}}}\|_{L_\infty}/\|u^\infty_{\ell_{\mathrm{ref}}}\|_{L_\infty}$, respectively, with $\ellref=15$ for Cantor sets and $\ellref=8$ for Cantor dusts.}
\label{fig:NearFar}
\end{figure}

\begin{figure}[tb!]\centering
	\includegraphics[width=0.47\textwidth,clip,trim=30 0 30 0]{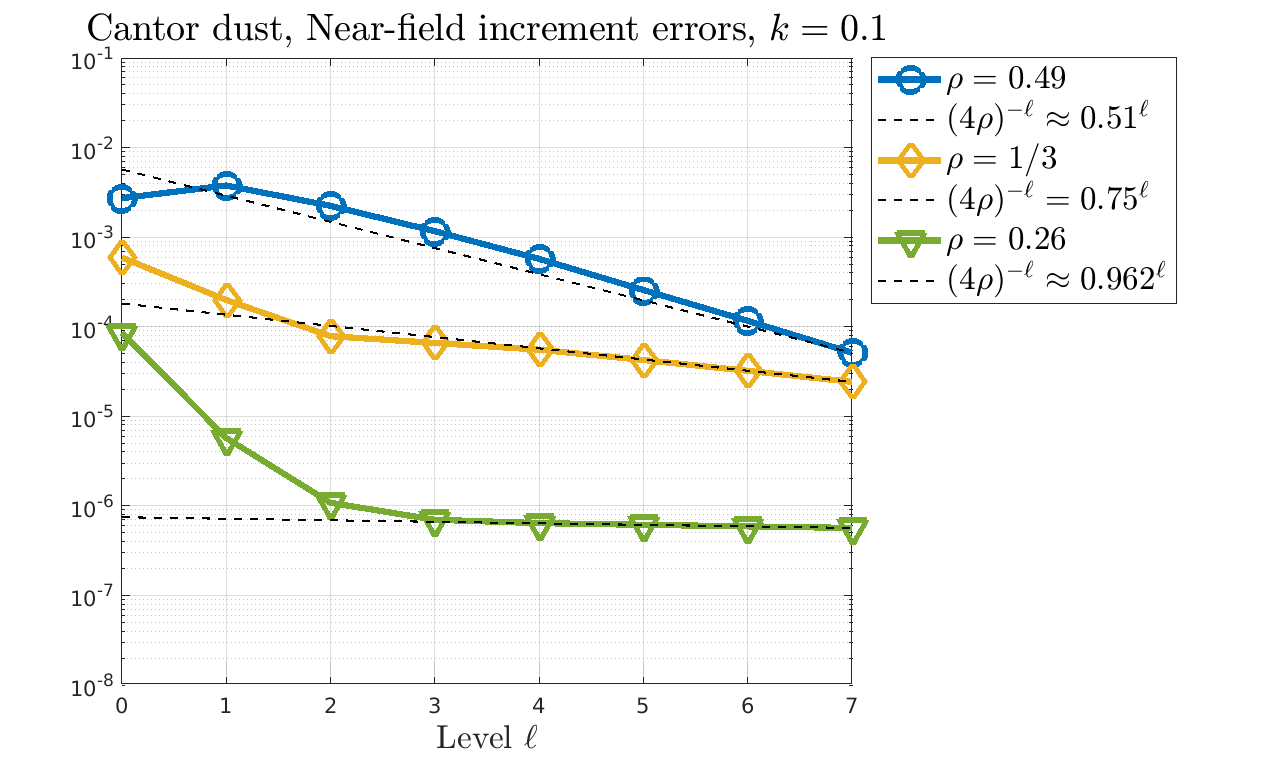}
	\includegraphics[width=0.47\textwidth,clip,trim=30 0 30 0]{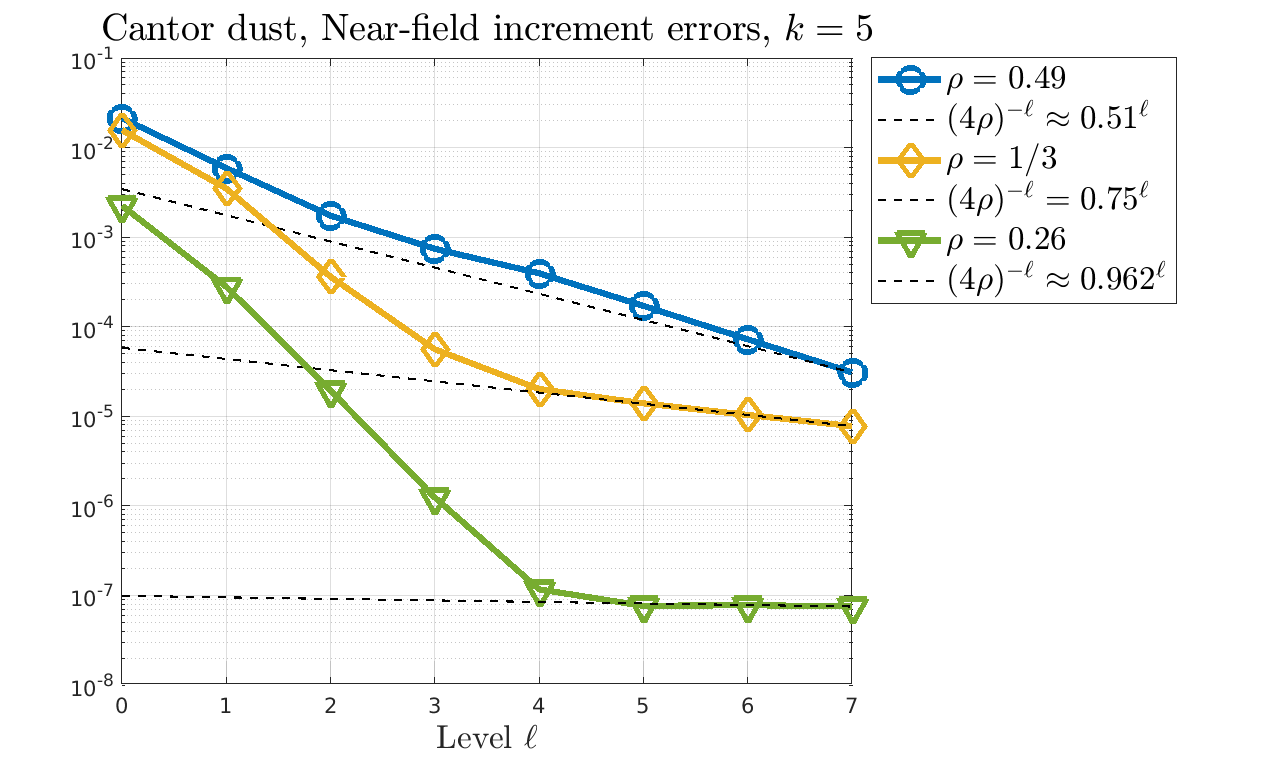}
	\\
	\includegraphics[width=0.47\textwidth,clip,trim=30 0 30 0]{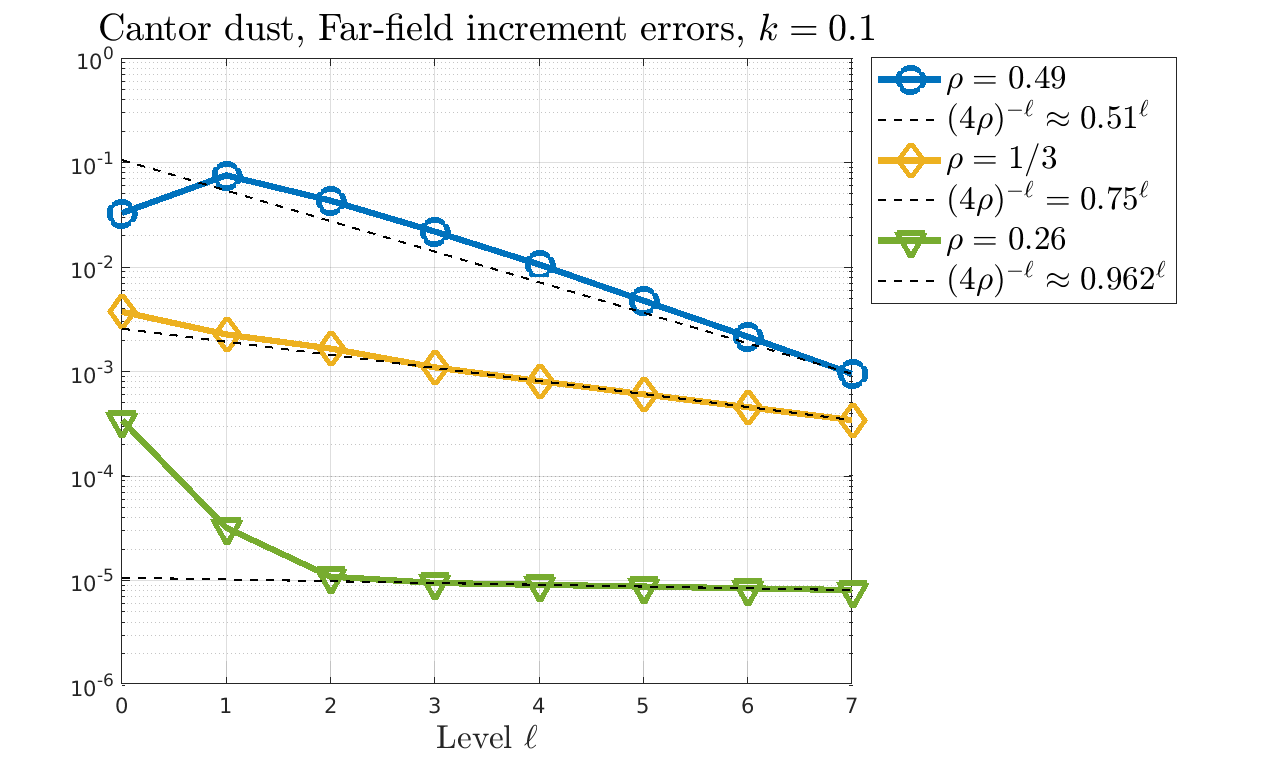}
	\includegraphics[width=0.47\textwidth,clip,trim=30 0 30 0]{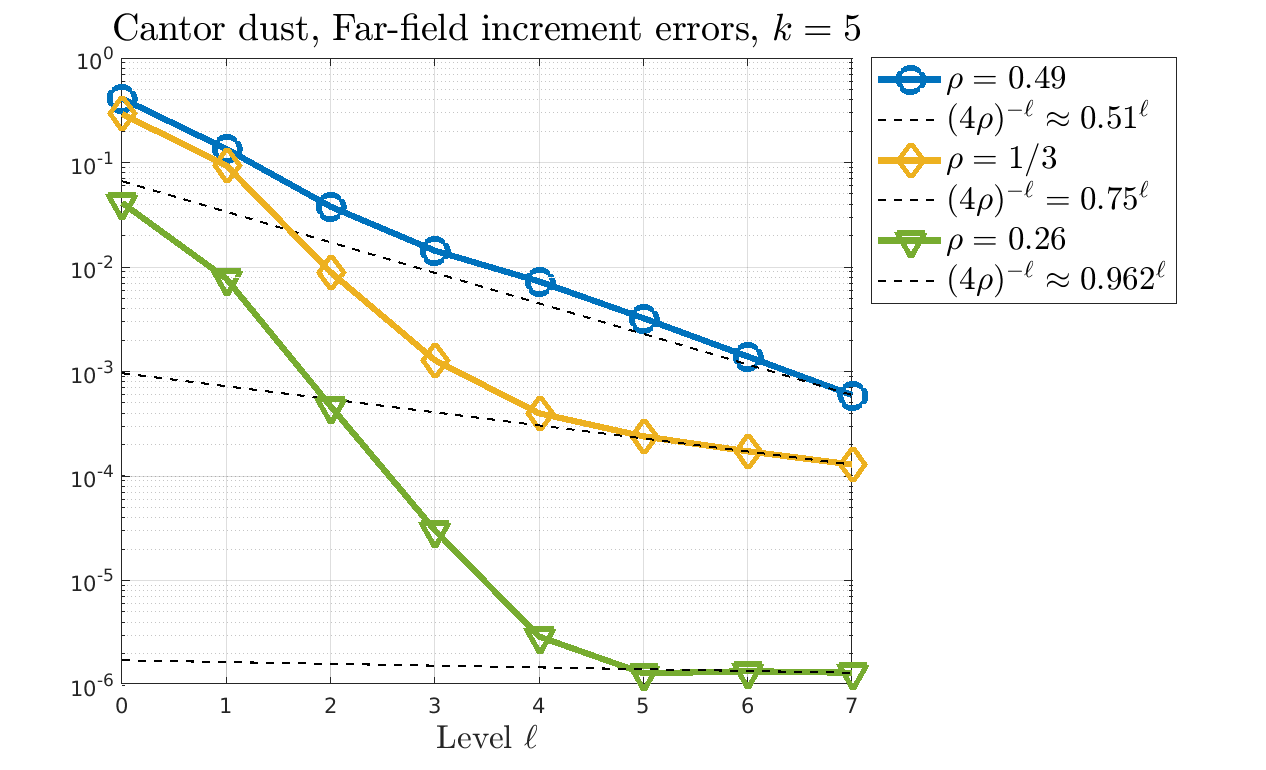}
	
	\caption{Incremental convergence of the scattered near- and far-field for Cantor dusts, measured as $\|u_\ell-u_{\ell+1}\|_{L_\infty}$ and $\|u^\infty_\ell-u^\infty_{\ell+1}\|_{L_\infty}$, respectively.}
	\label{fig:NearFar2}
\end{figure}

We now study the convergence of the Hausdorff-BEM approximation of the field scattered by $\Gamma$, for the same scatterers (Cantor sets and dusts), wavenumbers, incident waves and quadrature parameters   
considered in \S\ref{s:exp:Cantor}--\ref{s:exp:Dust}. The near field at $x\in\R^{n+1}$ and far-field pattern at $\hat{x}\in\mathds{S}^n$ are computed using \eqref{eq;JQdef}, choosing $\varphi=\Phi(x,\cdot)$ and $\varphi=\Phi^{\infty}(\hat{x},\cdot)$ respectively, using the same $h_Q$ values used to construct the associated BEM system. 
In Figure~\ref{fig:NearFar} we show 
	$L_\infty$ errors in the near- and far-field for scattering by Cantor sets and Cantor dusts, 
	obtained by computing the maximum error over a suitable set of sample points. In more detail, define the parameter $\cN(k):=10\max(k,2)$, which is always an even integer for the values of $k$ in our experiments. For the near-field, when $n=1$ we sample at $4\cN$ points on the boundary of the square $(-1,2)\times(-1.5,1.5)$, and when $n=2$ we sample at $\cN^2$ points on a uniform grid on the square $(-1,2)\times(-1,2)\times\{-1\}$ (recall that $\Gamma\subset[0,1]\times[0,1]\times\{0\}$). For the far-field, when $n=1$ we sample at $\cN$ points on the circle $\mathds{S}^1$, and when $n=2$ we sample at $\frac{\cN}{2}\times\cN$ points on the sphere $\mathds{S}^2$, chosen such that the points form a uniform grid in spherical coordinate space $[0,\pi]\times[0,2\pi]$.

For Cantor sets, we observe that near- and far-field errors converge to zero with rates $2^{-\ell}$, precisely as predicted by the theory in \eqref{eq:ConvRates}. For Cantor dusts, we observe convergence that is apparently faster than the predicted theory, similar to the observation made in relation to Figure \ref{fig:CantorDustH12}, which was explained in \S\ref{s:exp:Dust}. 
Applying the same reasoning 
as in \S\ref{s:exp:Dust} 
to achieve a robust comparison with the theoretical error bounds, in 
Figure \ref{fig:NearFar2} 
we plot 
$\|u_\ell-u_{\ell+1}\|_{L_\infty}$ and $\|u^\infty_\ell-u^\infty_{\ell+1}\|_{L_\infty}$ against $\ell$. Both agree with the predicted convergence rate of $(\rho M)^{-\ell}$.

\subsection{Comparison against ``prefractal BEM''}\label{sec:lebesgue}
\begin{figure}[tb!]\centering
	\includegraphics[width=\textwidth]{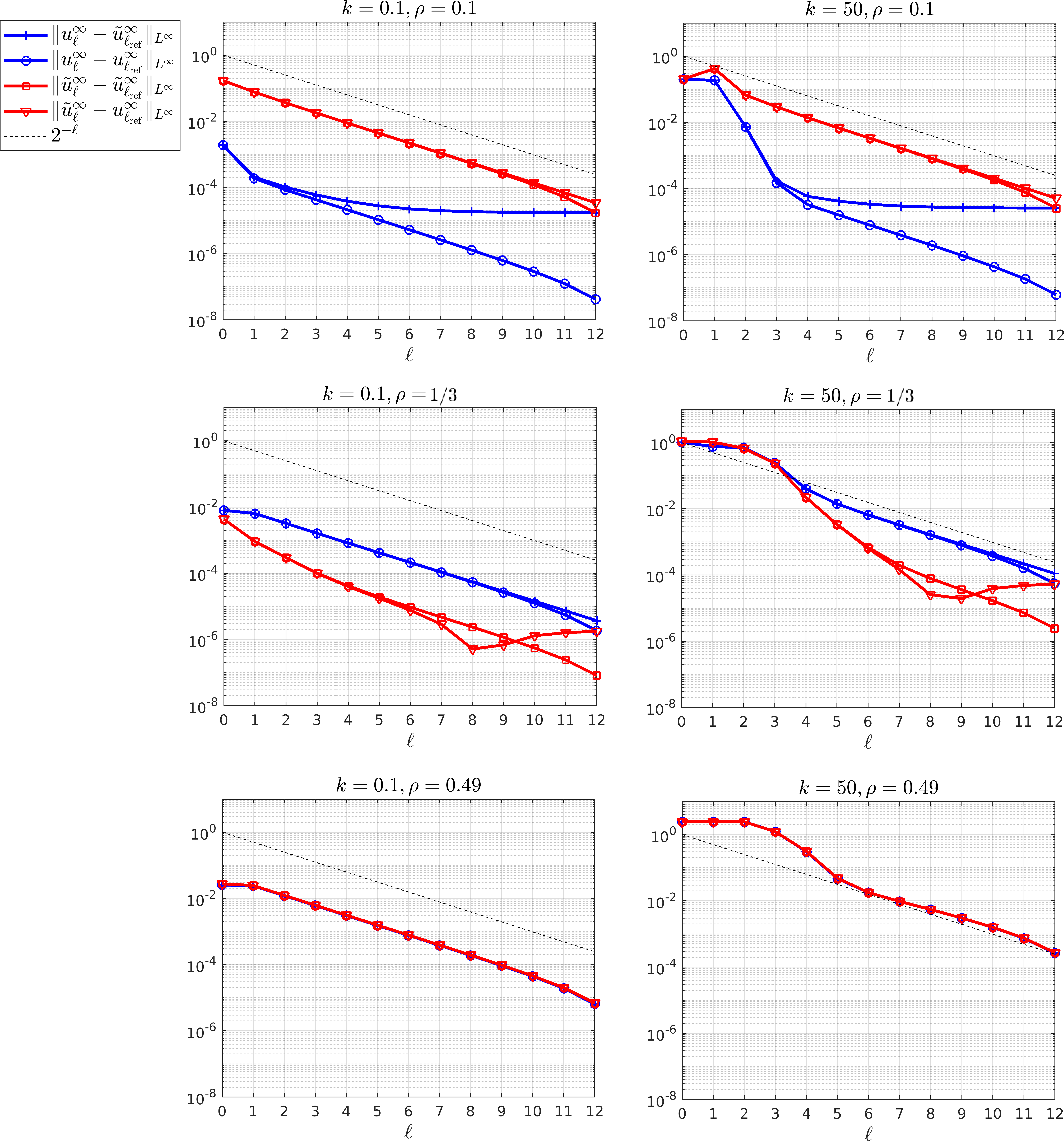}
	\caption{Absolute $L_\infty$ errors of Hausdorff-BEM (${u}_{\ell}^\infty$, blue curves) and prefractal-BEM ($\tilde{u}_{\ell}^\infty$, red curves) approximations of the far-field pattern $u^\infty$ for scattering by three Cantor sets with $k=0.1$ (left panels) and $k=50$ (right panels) and $\vartheta=(1/2,-\sqrt{3}/2)$. 
In all cases $\ellref=13$.
	}
	\label{fig:lebbem}
\end{figure}

\begin{figure}[tb!]\centering
	\includegraphics[width=0.48\textwidth]{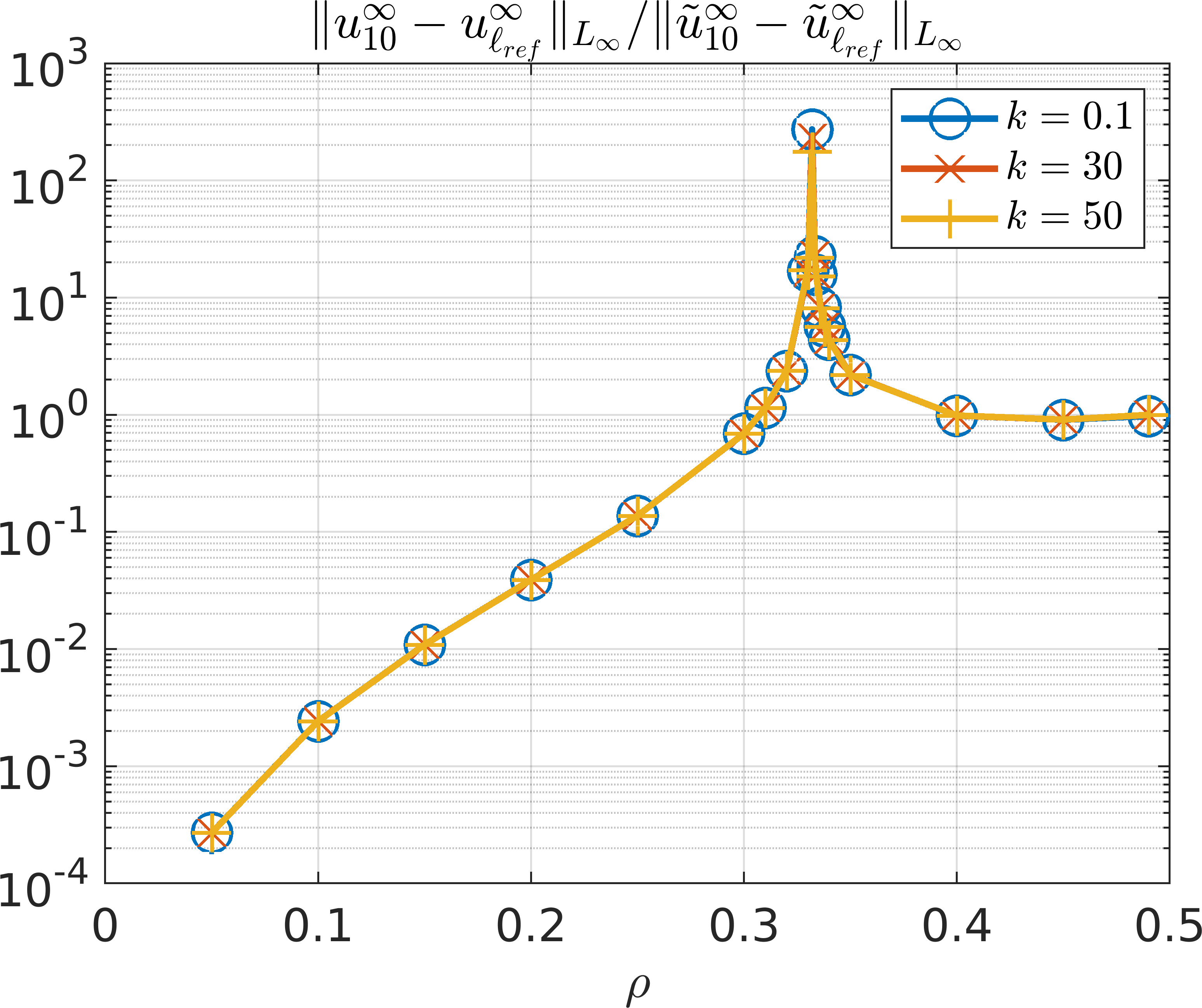}
	\hspace{3mm}
		\includegraphics[width=0.48\textwidth]{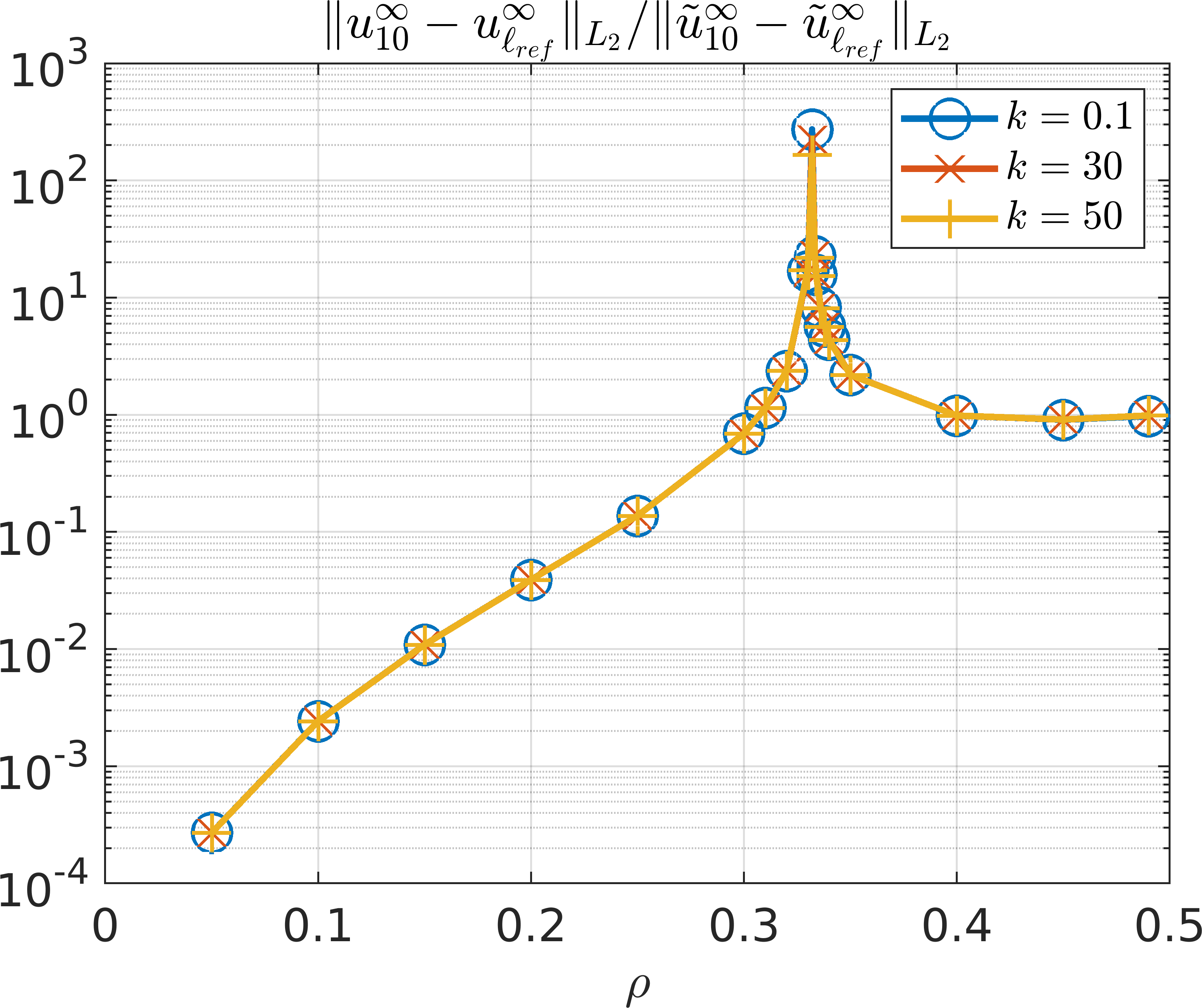}
	\caption{
	Ratio of errors of Hausdorff-BEM (${u}_{10}^\infty$) and prefractal-BEM ($\tilde{u}_{10}^\infty$) approximations of the far-field pattern $u^\infty$ with $\ell=10$ for scattering by  Cantor sets with $k=0.1$, $k=30$ and $k=50$ and $\vartheta=(1/2,-\sqrt{3}/2)$, for a range of values of $\rho\in(0,1/2)$, measured in the $L_\infty$ (left panel) and $L_2$ (right panel) norms on $\mathds{S}^1$. 
	A value below $1$ indicates that the Hausdorff BEM is more accurate, while a value above $1$ indicates that the prefractal BEM is more accurate. 
In both cases $\ellref=13$.
	}
	\label{fig:lebbem2}
\end{figure}

In this section we compare the approximations produced by our Hausdorff BEM with those produced by the method of \cite{BEMfract}. This provides a stronger validation of our Hausdorff BEM than the experiments so far presented, which compare our Hausdorff-BEM approximations against higher-accuracy approximations produced using the same Hausdorff-BEM code. We shall refer to the method of \cite{BEMfract} as ``prefractal BEM'', since it involves applying a standard BEM approximation on a prefractal approximation of $\Gamma$, which is the closure of a finite union of disjoint Lipschitz open sets. For brevity we focus only on the case of scattering by Cantor sets. In this case, for the prefractal BEM we take as prefractals the sequence of sets $\Gamma^{(0)}=[0,1]$, $\Gamma^{(\ell)}:=s(\Gamma^{(\ell-1)})$, $\ell\in\N$, where $s$ is as in \eqref{eq:fixedfirst} for the IFS \eqref{eq:CS_IFS}. 

For each $\ell\in\N_0$ we denote by $\tilde{u}^\infty_\ell$ the prefractal-BEM approximation to the far-field pattern $u^\infty$, 
computed on the prefractal $\Gamma^{(\ell)}$ using a standard piecewise-constant Galerkin BEM with one degree of freedom for each connected component of $\Gamma^{(\ell)}$. 
Quadrature was carried out using high-order Gauss and product Gauss rules, and quadrature parameters were chosen so that increasing quadrature accuracy did not noticeably change the results presented below.  
This leads to a total number of degrees of freedom $N(\ell)=2^\ell$, which is the same number of degrees of freedom used in our Hausdorff BEM at level $\ell$. As before, we write $u^\infty_\ell$ to denote the approximation to the far-field pattern for the Hausdorff BEM, which is computed as described in 
\S\ref{s:exp:NearFar}, 
except that now we use smaller values of the quadrature parameter $h_Q$ (as detailed below), to ensure that the results presented are not polluted by effects of insufficiently accurate quadrature, so our focus is on the Galerkin error \textit{per se}. 

In Figure \ref{fig:lebbem} we show $L_\infty$ errors in the far-field pattern between the two methods for $\rho = 0.1$, $1/3$ and $0.49$, with incident direction $\vartheta=(1/2,-\sqrt{3}/2)$ and wavenumbers $k=0.1$ and $k=50$ (with $h_Q=\rho^6 h$ for $k=0.1$ and $h_Q=\rho^8 h$ for $k=50$ for the Hausdorff BEM), obtained by computing the maximum error over 300 equally-spaced observation angles in $\mathds{S}^1$. 
For each method we calculate the errors in two different ways, using, firstly, a Hausdorff-BEM reference solution, and, secondly, a prefractal-BEM reference solution; in both cases the reference solution is computed with $\ellref=13$. 
In all cases, it is clear that both methods are converging to the same solution, validating both methods. 
Indeed, it appears that both methods converge at the same $2^{-\ell}$ rate. 
However, the results suggest that the relative accuracy of the two methods is dependent on the value of $\rho$, with the two methods having essentially the same accuracy for $\rho=0.49$, the prefractal BEM being more accurate for $\rho=1/3$, and the Hausdorff BEM being more accurate for $\rho=0.1$. To investigate this further we carried out similar experiments for more values of $\rho$, plotting the ratio of the errors obtained (measured in the $L_\infty$ and $L_2$ norms) with $\ell=10$ in Figure \ref{fig:lebbem2}.
We observe that:
\setlength{\parskip}{0pt}
\begin{itemize}
\setlength{\itemsep}{0pt}
\setlength{\parsep}{0pt}
\setlength{\parskip}{0pt}
\item for $\rho$ between $0.4$ and $0.5$ the two methods are very similar in accuracy;
\item for $\rho$ between $0.3$ and $0.4$ the prefractal BEM appears to be more accurate, significantly so, by a factor $>100$, for $\rho\approx1/3$ (this appears to be due entirely to some unexpected enhanced accuracy of the prefractal BEM for $\rho\approx 1/3$);
\item for $\rho$ between $0$ and $0.3$ the Hausdorff BEM appears to be more accurate, significantly so, by a factor $>1000$, for the lowest value of $\rho$;
\item the ratio of the errors for the two methods appears to be essentially independent of $k$ for the range of $k$ tested. To illustrate this we have also included in Figure \ref{fig:lebbem2} results for $k=30$ (computed with $h_Q=\rho^8 h$), alongside the results for $k=0.1$ and $k=50$; we observe that the results for all three $k$ values are almost identical. 
\end{itemize} 
These observations 
about the accuracy of the two methods (particularly the ``spike'' in Figure \ref{fig:lebbem2} near 
$\rho=1/3$)  
merit further investigation, but we leave this to future work.
\setlength{\parskip}{3mm}
\subsection{Non-homogeneous or non-disjoint IFS attractors}
\label{s:exp:SnowflakeWonky}

\begin{figure}[tb!]
	\centering
	\includegraphics[width=0.75\linewidth]{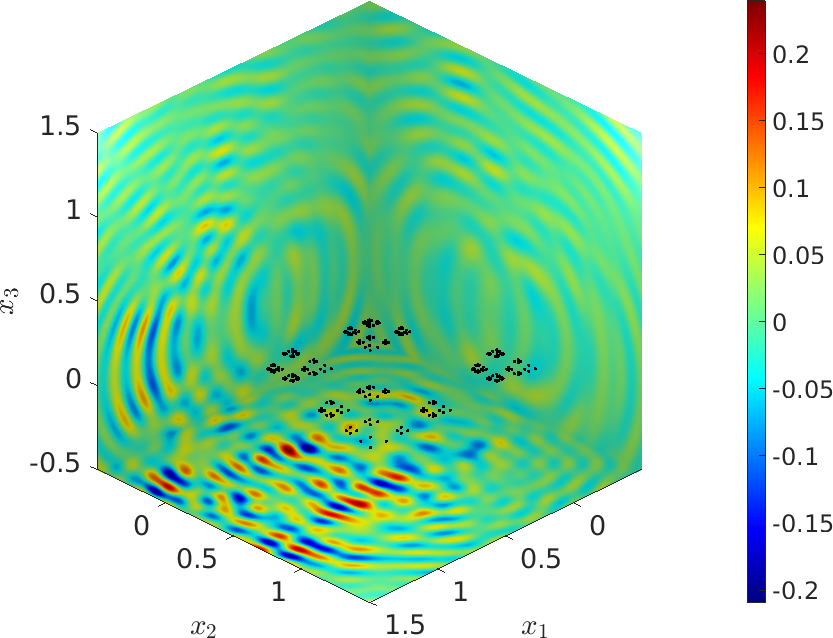}
	\caption{The scattered field induced by a plane wave, with wavenumber $k=50$ and direction vector $\vartheta=(0,1,-1)/\sqrt{2}$, incident on the ``non-homogenous dust" of \cite[Fig.~8b, eq.~(70)]{HausdorffQuadrature}, plotted on three faces of a cube, computed with mesh parameter $h=\hmeshref=\sqrt{2}/16384$. The scatterer, which is a subset of the plane $\R^2\times\{0\}$, is shown in black. Further details are given in \S\ref{s:exp:SnowflakeWonky}.}
	\label{fig:domainplot}
\end{figure}

\begin{figure}[tb!]\centering
\includegraphics[width=0.47\textwidth,clip,trim=20 0 0 0]{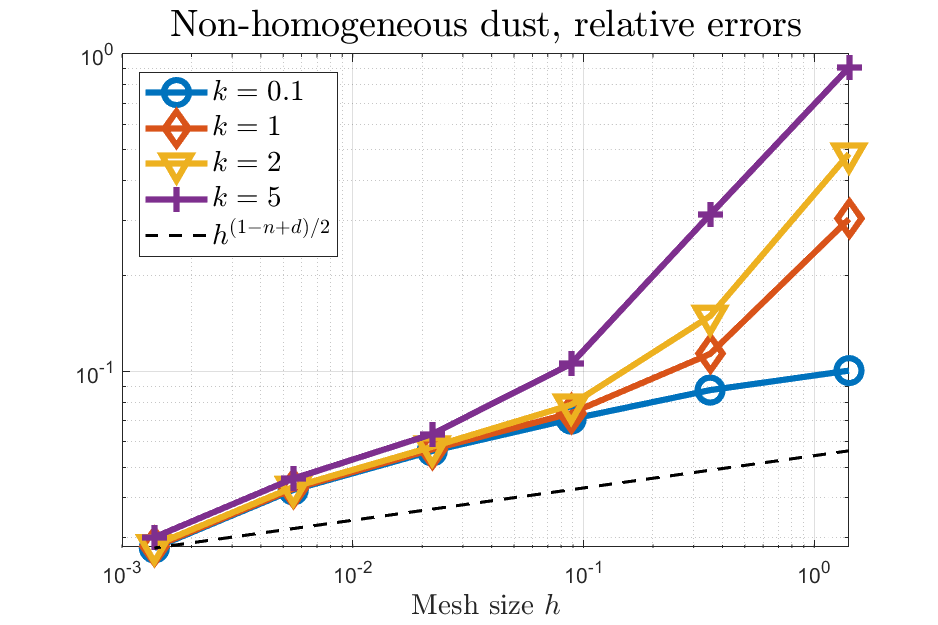}
\includegraphics[width=0.47\textwidth,clip,trim=20 0 0 0]{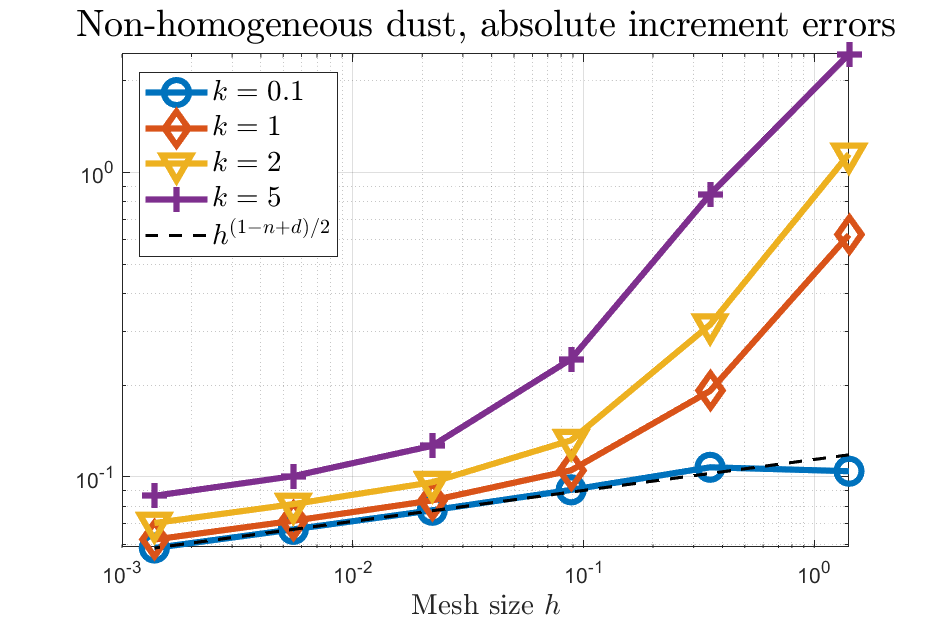}
\includegraphics[width=0.47\textwidth,clip,trim=20 0 0 0]{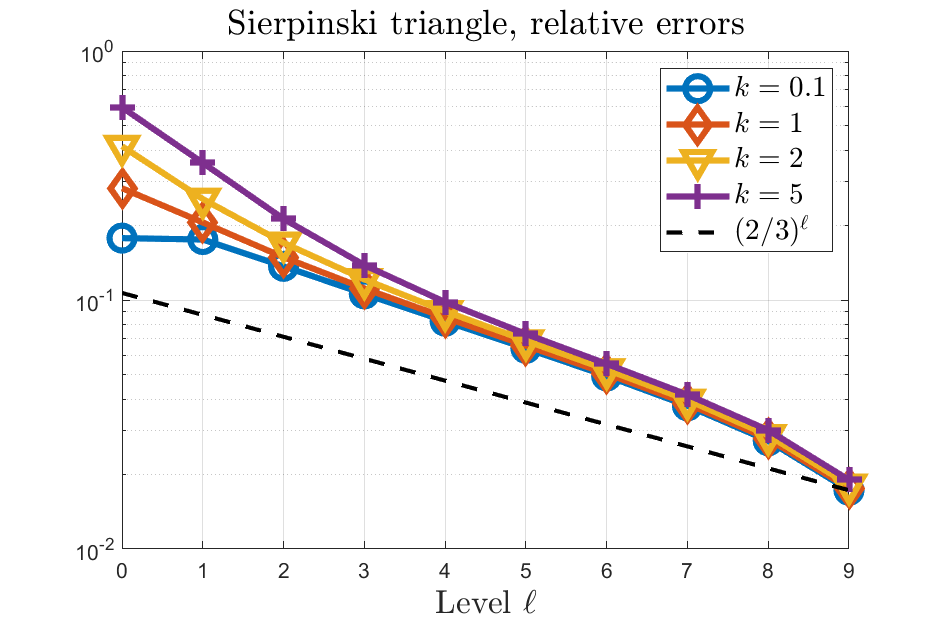}
\includegraphics[width=0.47\textwidth,clip,trim=20 0 0 0]{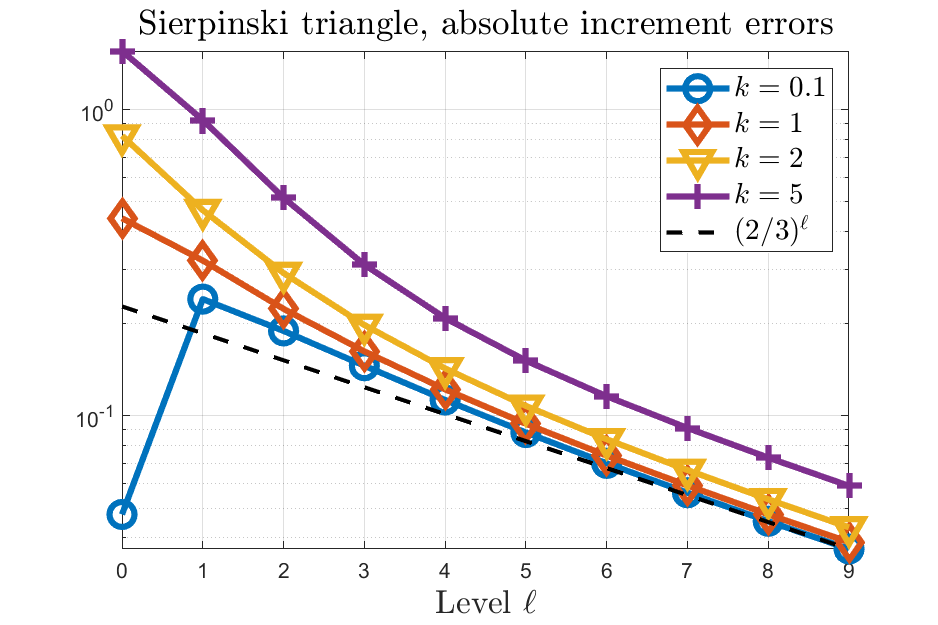}
\caption{$H^{-1/2}_\Gamma$ norms of the errors for the two IFS attractors described in \S\ref{s:exp:SnowflakeWonky}. Top left: $\|\phi_h-\phi_{\hmeshref}\|_{H^{-1/2}_\Gamma}/\|\phi_{\hmeshref}\|_{H^{-1/2}_\Gamma}$, bottom left: $\|\phi_\ell-\phi_{\ellref}\|_{H^{-1/2}_\Gamma}/\|\phi_{\ellref}\|_{H^{-1/2}_\Gamma}$, top right: $\|\phi_h-\phi_{h/4}\|_{H^{-1/2}_\Gamma}$ (see Footnote \ref{fn:inc}), bottom right: $\|\phi_\ell-\phi_{\ell+1}\|_{H^{-1/2}_\Gamma}$.
\label{fig:OtherShapes}}
\end{figure}

We now consider two IFS attractors that are not Cantor sets/dusts, both for $n=2$:
\begin{itemize}
\item[(i)] the ``non-homogeneous dust'', shown in Figure \ref{fig:domainplot}, and defined in \cite[Fig.~8b, eq.~(70)]{HausdorffQuadrature} (with $M=4$, $\rho_1=\rho_2=\rho_3=\frac14$, $\rho_4=\frac12$ and Hausdorff dimension $d\approx1.20$);
\item[(ii)] the Sierpinski triangle \cite[\S6.3]{BEMfract} (with $M=3$, $s_1(x)=\frac12x$, $s_2(x)=\frac12 x+(\frac12,0)$, $s_3(x)=\frac12x+(\frac14,\frac{\sqrt3}{4})$ and Hausdorff dimension $d=\log 3/\,\log2\approx 1.58$);
\end{itemize}
Both satisfy the OSC \eqref{oscfirst} and are hence $d$-sets for the stated values of $d$. But only (ii) is homogeneous, and only (i) is disjoint and satisfies the assumptions of the Hausdorff-BEM convergence theory of \S\ref{sec:GalerkinBounds}.
Figure~\ref{fig:OtherShapes} shows the $H^{-1/2}_\Gamma$ norms of the Hausdorff-BEM errors computed against a fixed reference solution (left column) and the norms of the differences between Hausdorff-BEM solutions on consecutive meshes as in Figure~\ref{fig:CantorDustIncrements} (right column), for a range of wavenumbers.%
\footnote{\label{fn:inc}
	In \S\ref{s:exp:Dust} we motivated the study of the errors $\|\phi_\ell-\phi_{\ell+1}\|_{H^\mhalf_\Gamma}$ between consecutive levels with the bound \eqref{eq:NumericsIncrement}.
In the case of non-homogeneous IFSs, denoting by $\phi_h$ the Hausdorff-BEM solution corresponding to a mesh of size $h$, assuming that the theoretical error bound \eqref{eq:GalerkinBound} is sharp, i.e.\ $\|\phi-\phi_h\|_{H^\mhalf_\Gamma}=Ch^{\mathfrak{a}}$ for some $C>0$ and $\mathfrak{a}=(1-n+d)/2$, it follows that
$
C h^{\mathfrak{a}}\big(1-(h'/h)^{\mathfrak{a}}\big)\le\|\phi_h-\phi_{h'}\|_{H^\mhalf_\Gamma}\le C h^{\mathfrak{a}}\big(1+(h'/h)^{\mathfrak{a}}\big)$
for $0<h'<h$. 
In the example (i), two consecutive meshes have $h'/h=1/4$; thus we expect theoretically that $\|\phi_h-\phi_{h'}\|_{H^\mhalf_\Gamma}$ is approximately proportional to $h^{\mathfrak{a}}$.
}
For the Sierpinski triangle (ii) we use meshes of $N=3^\ell$ elements of diameter $2^{-\ell}$ and we plot the errors against the level $\ell$ for $\ell=0,\ldots,9$, using a reference solution with $\ellref=10$, i.e.\ $\Nref=3^{10}=59049$.
For (i), since the IFS is non-homogeneous the BEM mesh is not parametrised by the level $\ell$ but rather by the mesh size $h=\frac{\sqrt2}{4^j}$, $j=0,\ldots,5$, giving the number of degrees of freedom as $N=1,7,40,217,1159,6160,$ respectively. The reference solution has $h=\hmeshref = \frac{\sqrt2}{4^6}= \frac{\sqrt2}{4096}$ and $\Nref=75316$.
For quadrature, for (i) we use $h_Q=\frac{h}{16}$ and for (ii) $h_Q=\frac{h}{4}$. %

For (i), as we found for the homogeneous Cantor dust in \S\ref{s:exp:Dust}, at first sight the convergence in the top-left panel of Figure~\ref{fig:OtherShapes} appears slightly faster than the theoretical rate $h^{(1-n+d)/2}$ predicted by Theorem \ref{thm:Convergence}. As before, this apparent mismatch is due to the limited accuracy of the reference solution; when we plot 
incremental errors (in the top right panel) 
we see much clearer agreement with the theory. 
Example (ii) is not covered by our convergence theory, but if our theory were to extend to this case our predicted convergence rate of $h^{(1-n+d)/2}$ would evaluate to $(2/3)^\ell$. Similarly to what we observed for (i), the relative errors in the bottom left panel of Figure~\ref{fig:OtherShapes} converge slightly faster than $(2/3)^\ell$, while the increments in the bottom right panel converge approximately in agreement with $(2/3)^\ell$. However, we leave theoretical justification of this empirical observation for future work.

\section{Conclusions and Future Work}
\label{sec:Conclusions}
In this paper we presented and analysed a piecewise-constant Galerkin BEM for acoustic scattering in $\R^{n+1}$, $n=1,2$, by a sound-soft planar screen $\Gamma\subset\Gamma_\infty = \R^n\times\{0\}$. Our BEM is defined whenever the screen $\Gamma$ is a compact $d$-set, for some $n-1<d\leq n$, which includes cases where $\Gamma$ is fractal. It is based on an integral equation formulation in which integration is carried out with respect to Hausdorff measure $\cH^d$. For any compact $d$-set we proved that the method converges as the mesh width $h$ tends to zero (see Theorem \ref{thm:Convergence}). 
Regarding the relationship between mesh width and wavelength, in our numerical results presented in \S\ref{sec:NumericalResults} we observe the same behaviour one obtains for a conventional BEM on a smooth scatterer: as $h$ decreases towards zero there is a pre-asymptotic phase until $h$ reaches the wavelength scale, beyond which one observes the predicted asymptotic behaviour. 
In the case where $\Gamma$ is the disjoint attractor of an IFS satisfying the OSC, the ``elements'' in the BEM are self-similar subsets of $\Gamma$, and in \S\ref{sec:Quadrature} we showed how the Galerkin integrals can be evaluated using quadrature rules from \cite{HausdorffQuadrature}. In this case we also proved fully-discrete convergence rates, under certain regularity assumptions on the integral equation solution (see Theorem \ref{thm:BEMConvergence}, Corollaries \ref{cor:fullydiscrete} and \ref{cor:fullydiscrete2}). 
Specifically, we showed that the BEM solution converges like $O(h^{s+1/2})$ as $h\to 0$, and the near- and far-field solutions like $O(h^{2s+1})$, assuming $\phi\in H^s_{\Gamma}$ for some $s>-1/2$. We proved the existence of such an $s$ in Remark \ref{rem:extra_smoothness}, using Proposition \ref{prop:epsilon}, but our numerical results in \S\ref{sec:NumericalResults} suggest that, for sufficiently smooth data, it may hold that $\phi\in H^s_{\Gamma}$ for all $s <-(n-d)/2$, i.e.\ for all $s$ such that the space $H^s_{\Gamma}$ is non-trivial. 
Guided by these observations, we formulated Conjecture \ref{ass:Smoothness}, which is a statement about the range of Sobolev spaces on which the Hausdorff-measure integral operator $\mathbb{S}$ (defined in \eqref{eq:ISdef}) is invertible.  

Proving or disproving Conjecture \ref{ass:Smoothness} is the main outstanding theoretical question relating to the paper. Other avenues for future research include:
\begin{itemize}
\item The extension of our analysis to non-disjoint fractal screens, such as the Sierpinski triangle screen considered in \S\ref{s:exp:SnowflakeWonky}. Singular quadrature rules for such cases (generalising those described in \S\ref{sec:Quadrature}) have been presented recently in \cite{NonDisjointQuad}, but extending our BEM convergence analysis will require a suitable generalisation of the wavelet approximation theory of \cite{Jonsson98}, which is yet to be worked out. 
\item The extension of our convergence rate analysis to the case $d=n$. This requires somewhat different techniques to the case $d<n$, 
and will be presented in a separate article \cite{dequalsnpaper}.
\item The generalisation of our Hausdorff BEM to scattering by non-planar fractal structures. Results in this direction were presented recently in \cite{HausdorffDomain}. 
\item The extension to Neumann problems. Theoretical results relating to integral equation formulations of Neumann screen problems were presented in \cite{ScreenPaper}. These show that the Neumann boundary condition is ``weaker'' than the Dirichlet condition, in the sense that Neumann screens do not scatter waves unless $H^{1/2}_\Gamma\neq \{0\}$, which requires in particular that $\Gamma\subset\Gamma_\infty\cong\R^n$ has positive $n$-dimensional Lebesgue measure. Hence, in the context of scattering by screens that are attractors of IFS satisfying the OSC, only the case $d=n$ is relevant. A specific example would be scattering by a sound-hard Koch snowflake screen. However, the development of a BEM for such Neumann problems using a mesh of fractal elements (as we consider in the current paper for the Dirichlet case) is complicated by the fact that non-trivial piecewise polynomials on such meshes cannot be continuous, and hence cannot be $H^{1/2}_\Gamma$-conforming. Thus a different approach is required, perhaps involving a discontinuous Galerkin discretization. 
We remark that a rigorous convergence analysis for a prefractal BEM approach to the related impedance problem was presented in \cite{ImpedanceScreen}. 
\item A more detailed investigation into the relative accuracy of our Hausdorff BEM and the alternative prefractal BEM of \cite{BEMfract}, extending the preliminary analysis of \S\ref{sec:lebesgue}. For many problems the Hausdorff BEM appears to be more accurate than the prefractal BEM. However, the comparison between the two is rather subtle: for the examples we considered in \S\ref{sec:lebesgue}, while both methods appear to converge at the same rate, for the same number of degrees of freedom the Hausdorff BEM can be over 1000 times more accurate than the prefractal BEM, or 100 times less accurate, depending on the fractal simulated. As yet we do not understand why. %
\item The combination of our Hausdorff BEM with accelerated linear algebra via a technique such as $\mathcal{H}$-matrix compression or the fast multipole method. This would allow the simulation of larger problems with finer mesh widths, permitting the calculation of higher accuracy solutions and/or the study of higher frequency problems. 
\end{itemize}

\appendix
\section{Besov spaces on \texorpdfstring{$d$}{d}-sets}
\label{app:Besov}

In this appendix we show that the spaces denoted by $B_{\alpha}^{p,q}(\Gamma)$,
and defined in terms of atoms in \cite[section 6]{Jonsson98},
coincide with the spaces denoted by $\mathbb{B}_{p,q,0}^{\alpha}(\Gamma)$
defined  in \cite[Def.~6.3]{caetano2019density}, under the assumptions
$\alpha>0$, $1\leq p,q<\infty$, and $\Gamma$ being a $d$-set, $0<d<n$, %
preserving Markov's inequality in the sense of \cite[\S4]{Jonsson98} (or see \cite[p.~34]{JoWa84}, and note Remark \ref{rem:key} below). As a consequence (see Corollary \ref{cor:equiv}) we have that
our space $\mathbb{H}^{\alpha}(\Gamma)$, defined in \S\ref{sec:FunctionSpaces}, coincides with $B_{\alpha}^{2,2}(\Gamma)$
under the same assumptions on $\alpha$ and $\Gamma$.\footnote{We recall (see the discussion in \S\ref{sec:FunctionSpaces} and \cite[Rem.\ 6.4]{caetano2019density}) that, for $0<\alpha<1$, $\IH^\alpha(\Gamma)$ also coincides with the Besov space $B^\alpha_{2,2}(\Gamma)$ of \cite{JoWa84}.} %

We start by rephrasing the above mentioned definition of $B_{\alpha}^{p,q}(\Gamma)$
in terms closer to the notation used in \cite{Caetano2011}, which
we want to use to make the announced connection. %
In order to do that, we need first to recall the notion of atom as
used in \cite{Jonsson98}. For each $\nu\in\No$, consider the family of cubes
\[
Q^{\nu m}:=\prod_{i=1}^{n}[2^{-\nu}m_{i},2^{-\nu}(m_{i}+1)],\quad m\in\Zn.
\]
Then, given $\alpha$, $p$, $d$, and $K$ as in Definition \ref{def:B^p,q_alpha}
below, an $(\alpha,p)$-atom $a^{\nu m}$ associated with $Q^{\nu m}$
is any  $a^{\nu m}\in C^K(\R^n)$ such that
\[
\supp a^{\nu m}\subset3Q^{\nu m},
\]
\[
|D^{\beta}a^{\nu m}(x)|\leq2^{-\nu(\alpha-|\beta|-d/p)},\quad x\in\R^n,\;|\beta|\leq K,
\]
where for a cube $Q\subset\R^n$ and a scalar $r>0$ the cube $r Q$ is defined to have the same centre $x_Q$ as $Q$ and side length $r$ times that of $Q$, i.e., $r Q := \{x_Q + r(x-x_Q):x\in Q\}$. %

The following definition of $B_{\alpha}^{p,q}(\Gamma)$ is the second, equivalent, definition of \cite[\S6]{Jonsson98}. %
\begin{defn} %
\label{def:B^p,q_alpha}Let $\alpha>0$, $1\leq p,q<\infty$, and $\Gamma$
be a $d$-set, $0<d<n$, preserving Markov's inequality. Assume that
$\N\ni K>\lfloor\alpha\rfloor$ %
Then $f\in B_{\alpha}^{p,q}(\Gamma)$ iff %
\begin{equation}
f=\sum_{\nu=0}^{\infty}\sum_{m\in\Zn}\lambda^{\nu m}a^{\nu m},\quad\mbox{ convergence in
}\mathbb{L}_{p}(\Gamma),\label{eq:atom rep Jonsson}
\end{equation}
for some family of $(\alpha,p)$-atoms $a^{\nu m}$ and numbers $\lambda=(\lambda^{\nu m})_{\nu\in \N_0, \, m\in \Z^n}$ satisfying
\begin{equation}
\|\lambda\|_{b_{p,q}}:=\left(\sum_{\nu=0}^{\infty}\left(\sum_{m\in\Zn}|\lambda^{\nu m}|^{p}\right)^{q/p}\right)^{1/q}<\infty,\label{eq:b_p,q}
\end{equation}
the norm of $f$ in $B_{\alpha}^{p,q}(\Gamma)$ being defined as the
infimum of the left-hand side of (\ref{eq:b_p,q}) taken over all possible representations
(\ref{eq:atom rep Jonsson}) of $f$.
\end{defn}

In \cite{Caetano2011}, atoms are also considered, but the definition
is different: for each $\nu\in\No$, consider the family of cubes
\[
Q_{\nu m}:=\prod_{i=1}^{n}[2^{-\nu}(m_{i}-1),2^{-\nu}(m_{i}+1)],\quad m\in\Zn.
\]
Then, given $s\in\R$, %
$0<p<\infty$, $K\in\No$, and $c\geq1$, an
$(s,p)_{K,0,c}$-atom $a_{\nu m}$ associated with $Q_{\nu m}$ is
any %
continuous function with (classical) partial derivatives up to
the order $K$ such that
\[
\supp a_{\nu m}\subset cQ_{\nu m},
\]
\[
|D^{\beta}a_{\nu m}(x)|\leq2^{-\nu(s-|\beta|-n/p)},\quad x\in \R^n,\;|\beta|\leq K.
\]

\begin{rem}
\label{rem:comparing atoms}It is easily seen that, for each $\nu\in\No$
and $m\in\Zn$, $Q_{\nu m}\subset3Q^{\nu m}\subset2Q_{\nu m}$. And
from this it is straightforward to show that:
\begin{enumerate}
\item Let $\alpha>0$, $1\leq p<\infty$, $0<d<n$, and $\N\ni K>\lfloor\alpha\rfloor$.
Any $(\alpha,p)$-atom associated with $Q^{\nu m}$ is an $(\alpha+\frac{n-d}{p},p)_{K,0,2}$-atom
$a_{\nu m}$ associated with $Q_{\nu m}$.
\item Let $0<d<n$, $s>\frac{n-d}{p}$, $1\leq p<\infty$, $\N\ni K>\lfloor s-\frac{n-d}{p}\rfloor$.
Any $(s,p)_{K+1,0,1}$-atom associated with $Q_{\nu m}$ is an $(s-\frac{n-d}{p},p)$-atom
associated with $Q^{\nu m}$.
\end{enumerate}
\end{rem}

\begin{prop}
\label{prop:atomic characterization traces}Let $\alpha>0$, $1\leq p,q<\infty$,
and $\Gamma$ be a $d$-set, $0<d<n$, preserving Markov's inequality.
Then the space $\mathbb{B}_{p,q,0}^{\alpha}(\Gamma)$ defined in
\cite[Def.~6.3]{caetano2019density} coincides with the space $B_{\alpha}^{p,q}(\Gamma)$
of Definition \ref{def:B^p,q_alpha} above, with equivalence of norms.
\end{prop}

\begin{proof}
Consider $\N\ni K>\alpha+\frac{n-d}{p}$. We have, in particular,
that $K>\lfloor\alpha\rfloor$.
Given $f\in B_{\alpha}^{p,q}(\Gamma)$, we have (\ref{eq:atom rep Jonsson})
with (\ref{eq:b_p,q}) for some family of $(\alpha,p)$-atoms $a^{\nu m}$
and numbers $\lambda^{\nu m}$ such that
$$
\|\lambda\|_{b_{p,q}} \leq 2\|f\|_{B_{\alpha}^{p,q}(\Gamma)}.
$$
On the other hand, since $K>\alpha+\frac{n-d}{p}$
and $0>-\alpha-\frac{n-d}{p}$, it follows from \cite[Thm.~2.3, Rem.~2.4 and Thm.~2.5]{Caetano2011}
and part 1 of the above Remark \ref{rem:comparing atoms} that
\[
g:=\sum_{\nu=0}^{\infty}\sum_{m\in\Zn}\lambda^{\nu m}a^{\nu m}\in B_{p,q}^{\alpha+\frac{n-d}{p}}\Rn
\]
with convergence in the sense of $B_{p,q}^{\alpha+\frac{n-d}{p}}\Rn$
and
\begin{equation}
\|g\|_{B_{p,q}^{\alpha+\frac{n-d}{p}}\Rn }\leq c\|\lambda\|_{b_{p,q}},\label{eq:B lesssim b}
\end{equation}
for some constant $c>0$ independent of $f$.
From elementary embeddings and the linearity and continuity of the
trace operator $\tr $ from \cite[Prop. 6.2]{caetano2019density},
\[
\tr g=\sum_{\nu=0}^{\infty}\sum_{m\in\Zn}\lambda^{\nu m}\tr a^{\nu m}=\sum_{\nu=0}^{\infty}\sum_{m\in\Zn}\lambda^{\nu m}a^{\nu m}|_{\Gamma},\quad\mbox{ convergence in }\mathbb{L}_{p}(\Gamma),
\]
the latter equality being justified by the arguments in \cite[Rem.~6.4]{caetano2019density}
and the fact that each $a^{\nu m}$ is continuous. But then $\tr g$
must be the same as $f$ (recall the representation (\ref{eq:atom rep Jonsson})
and the fact that $a^{\nu m}|_{\Gamma}$ and $a^{\nu m}$ define the
same class in $\mathbb{L}_{p}(\Gamma)$.
So, we have proved that
$f$ also belongs to $\mathbb{B}_{p,q,0}^{\alpha}(\Gamma)$. Additionally,
using (\ref{eq:B lesssim b}),
\[
\|f\|_{\mathbb{B}_{p,q,0}^{\alpha}(\Gamma)}\leq\|g\|_{B_{p,q}^{\alpha+\frac{n-d}{p}}\Rn }\leq c\|\lambda\|_{b_{p,q}} \leq 2c\|f\|_{B_{\alpha}^{p,q}(\Gamma)}.
\]

To see  the opposite inclusion,
given $f\in\mathbb{B}_{p,q,0}^{\alpha}(\Gamma)$, choose $g\in B_{p,q}^{\alpha+\frac{n-d}{p}}\Rn $
such that $f=\tr g$ and
\begin{align} \label{eq:extralabel}
\|g\|_{B_{p,q}^{\alpha+\frac{n-d}{p}}\Rn} \leq 2\|f\|_{\mathbb{B}_{p,q,0}^{\alpha}(\Gamma)}.
\end{align}
Since $K+1>\alpha+\frac{n-d}{p}$ and
$0>-\alpha-\frac{n-d}{p}$, it follows from \cite[Thm.~2.3 and Thm.~2.5]{Caetano2011}
that
\[
g=\sum_{\nu=0}^{\infty}\sum_{m\in\Zn}\lambda_{\nu m}a_{\nu m},\quad\mbox{ convergence in }B_{p,q}^{\alpha+\frac{n-d}{p}}\Rn ,
\]
for some family of $(\alpha+\frac{n-d}{p},p)_{K+1,0,1}$-atoms $a_{\nu m}$
and numbers $\lambda_{\nu m}$ which we can choose to satisfy the estimate
\begin{equation}
\left(\sum_{\nu=0}^{\infty}\left(\sum_{m\in\Zn}|\lambda_{\nu m}|^{p}\right)^{q/p}\right)^{1/q}\leq c^\prime\|g\|_{B_{p,q}^{\alpha+\frac{n-d}{p}}\Rn },\label{eq:b lesssim B}
\end{equation}
for some constant $c^\prime>0$ independent of $f$. Again using
elementary embeddings and the linearity
and continuity of the trace operator $\tr $ from \cite[Prop.~6.2]{caetano2019density},
\[
f=\tr g=\sum_{\nu=0}^{\infty}\sum_{m\in\Zn}\lambda_{\nu m}\tr a_{\nu m}=\sum_{\nu=0}^{\infty}\sum_{m\in\Zn}\lambda_{\nu m}a_{\nu m}|_{\Gamma},\quad\mbox{ convergence in }\mathbb{L}_{p}(\Gamma).
\]
But then we see from part 2 of the above Remark \ref{rem:comparing atoms}
that we have what is needed to conclude that $f$ also belongs to
$B_{\alpha}^{p,q}(\Gamma)$. Moreover, from \eqref{eq:extralabel} and (\ref{eq:b lesssim B})
it also follows that
\[
\|f\|_{B_{\alpha}^{p,q}(\Gamma)}\leq c^\prime\|g\|_{B_{p,q}^{\alpha+\frac{n-d}{p}}\Rn }  \leq 2c^\prime \|f\|_{\mathbb{B}_{p,q,0}^{\alpha}(\Gamma)}.
\]
\end{proof}
\begin{cor} \label{cor:equiv}
Let $\alpha>0$ and $\Gamma$ be a $d$-set, $0<d<n$, preserving
Markov's inequality. Then $\mathbb{H}^{\alpha}(\Gamma)=B_{\alpha}^{2,2}(\Gamma)$
with equivalence of norms.
\end{cor}

\begin{proof} %
Our space $\IH^\alpha(\Gamma)$ coincides with the spaces $\IH^\alpha_{2,0}(\Gamma)=\mathbb{F}^\alpha_{2,2,0}(\Gamma)$ defined in \cite[Def. 6.3]{caetano2019density}. Moreover, since (for all $s\in \R$)
the Triebel-Lizorkin space $F_{2,2}^{s}\Rn$ coincides with the Besov space $B_{2,2}^{s}\Rn $, with equivalent norms, it follows, inspecting the definitions of $\mathbb{F}^\alpha_{2,2,0}(\Gamma)$ and $\mathbb{B}_{2,2,0}^{\alpha}(\Gamma)$ in \cite[Def. 6.3]{caetano2019density}, that $\mathbb{F}^\alpha_{2,2,0}(\Gamma)=\mathbb{B}_{2,2,0}^{\alpha}(\Gamma)$,
also with equivalence of norms, and the result follows immediately
from the previous proposition.
\end{proof}
\begin{rem}
Characterizations of trace spaces in terms of atomic representations,
like the one given in Proposition \ref{prop:atomic characterization traces}
above, are not new. With a somewhat different approach and in the
larger setting of the so-called $h$-sets, they can already be seen
in \cite[Prop.~3.5.4]{Bricchi}.
\end{rem}

\begin{rem} \label{rem:key}
If $\Gamma\subset \R^n$ is a $d$-set and $d>n-1$,
then the hypothesis that $\Gamma$ preserves Markov's inequality is
automatically satisfied  \cite[Thm.~3 on p.~39]{JoWa84}.
\end{rem}

\section{Inverse estimates}
\label{app:Inverse}
In this appendix we prove that the inverse estimate \eqref{eq:InvEst}, established for $0<t<1$ in Theorem \ref{thm:inverse}, holds in fact for the extended range $0<t\leq J$, for any $J\in\N$, moreover with a constant in the inverse estimate that is independent of $t$.
Our proof is modelled on standard inverse estimate arguments for negative exponent Sobolev spaces, %
e.g.\ \cite[Theorem 4.6]{Dahmen04}, with the important difference that it is, of course, impossible to support the smooth ``bubble functions'' (in the terminology of  \cite[\S4.3]{Dahmen04}), that are a key tool in the arguments of \cite{Dahmen04} and in the proof of earlier inverse estimate results, inside an element $\Gamma_{\bm}\subset \Gamma$, as $\Gamma$ has empty interior. In the case that $\Gamma$ is a disjoint IFS attractor, it turns out that, to carry through an analogous argument, it is enough to replace the bubble function by a smooth function supported in a carefully chosen neighbourhood of $\Gamma_{\bm}$ (the function $\sigma_{\bm}\in C_0^\infty(\R^n)$ in the proof below). The assumption that $\Gamma$ is a disjoint IFS attractor implies that $\Gamma$ is a $d$-set with $0<d<n$ 
(see Lemma \ref{lem:disconnected}), 
but, in contrast to Theorem \ref{thm:inverse}, we make no additional constraint on $d$.
We continue to assume in this appendix that $h$ lies in the range \eqref{eq:hrange}, i.e.\ that
\begin{equation*}
0<h\leq h_0 := \diam(\Gamma),
\end{equation*}
and to use the notations of the main part of the paper, notably $L_h$, $\chi_\bm$, and $\IY_h$ defined by \eqref{eq:LhDef}, \eqref{eq:VkBasisDefn}, and \eqref{eq:YHDef2}, respectively.

\begin{thm} \label{thm:inverse2} Suppose that $\Gamma$ is the attractor of an IFS, satisfying \eqref{eq:fixedfirst}, that $\Gamma_1$, ..., $\Gamma_M$ are disjoint (i.e., $\Gamma$ is disjoint), and that $J\in \N$. Then there exists $c_I>0$ such that, for every $\psi_h\in \IY_h$,
\begin{equation} \label{eq:ii}
\|\psi_h\|_{\IL_2(\Gamma)} \leq c_I h^{-t}\|\psi_h\|_{\IH^{-t}(\Gamma)}, \quad 0<t\leq J.
\end{equation}
\end{thm}
\begin{proof} Let
$$
\rho_{\min} := \min_{j=1,...,M} \rho_j \quad \mbox{ and } \quad c_0 := \cH^d(\Gamma).
$$
Since $\Gamma_1$, ..., $\Gamma_M$ are disjoint, $\Gamma$ satisfies the OSC for some open set $O\supset \Gamma$  
by Lemma \ref{lem:DisjointOSC}. 
By a standard mollification we can construct a $\sigma\in C_0^\infty(\R^n)$ supported in $O$ such that $\sigma = 1$ in a neighbourhood of $\Gamma$. %
For each $\bm=(m_1,...,m_\ell) \in L_h$, let $\rho_{\bm}:= \prod_{j=1}^\ell \rho_{m_j}$,
\begin{equation} \label{eq:hbm}
h_{\bm}:=  \rho_{\bm}h_0 = \diam(\Gamma_{\bm}) \in (\rho_{\min} h,h],
\end{equation}
and
\[
\sigma_{\bm} := \sigma\circ s_{m_\ell}^{-1}\circ s_{{m_{\ell-1}}}^{-1}\circ\cdots \circ s_{m_1}^{-1},
\]
and note that, as a consequence of the OSC, the supports of $\sigma_{\bm}$ and $\sigma_{\bm'}$ are disjoint, for $\bm\neq \bm'$.

Now suppose that $0<t\leq J$, and define $s$ in terms of $t$ by \eqref{eq:st}. For $\bm\in L_h$, noting \eqref{eq:L2dualrep} and \eqref{eq:VkBasisDefn}, we have that
\begin{equation} \label{eq:chi_sigma}
\langle \sigma_{\bm},\tr^* \chi_{\bm}\rangle_{H^{s}(\R^n)\times H^{-s}(\R^n)}=(\tr \sigma_{\bm},\chi_{\bm})_{\IL_2(\Gamma)} = (\cH^d(\Gamma_{\bm}))^{1/2} = c_0^{1/2}\, \rho_{\bm}^{d/2}.
\end{equation}
Clearly \eqref{eq:ii} holds for $\psi_h=0$. Now suppose that $\psi_h\in \IY_h\setminus \{0\}$, %
so that, for some coefficients $a=(a_{\bm})_{\bm\in L_h}$, %
we have
$$
\psi_h = \sum_{\bm\in L_h}a_{\bm} \chi_{\bm},
$$
and
$$
\|\psi_h\|_{\IL_2(\Gamma)} = \|a\|_2   := \left(\sum_{\bm\in L_h}|a_{\bm}|^2\right)^{1/2}.
$$
Define $u_h\in Y_h\subset H^{-s}(\R^n)$ and $v_h\in H^s(\R^n)$ by
$$
u_h:= \tr^*\psi_h= \sum_{\bm\in L_h}a_{\bm} \tr^*\chi_{\bm} \quad \mbox{and} \quad v_h := \sum_{\bm\in L_h}a_{\bm}\sigma_{\bm}.
$$
Then, since the supports of the $\sigma_{\bm}$ are disjoint and again noting \eqref{eq:L2dualrep}, we have, using \eqref{eq:chi_sigma} and \eqref{eq:hbm}, that %
\begin{equation} \label{eq:lb1}
\langle v_h,u_h\rangle_{H^{s}(\R^n)\times H^{-s}(\R^n)} %
= c_0^{1/2}\, \sum_{\bm\in L_h}|a_{\bm}|^2 \rho_{\bm}^{d/2}\geq c_0^{1/2}\,(\rho_{\min}h/h_0)^{d/2} \,\|a\|_2^2.
\end{equation}
For $j\in \N_0$  and $\bm\in L_h$, %
using standard properties of Fourier transforms (e.g., \cite[Prop. 2.2.11]{Grafakos}),
\begin{eqnarray*}
\|\sigma_{\bm}\|^2_{H^{j}(\R^n)} &=& \rho_\bm^{2n} \int_{\R^n} \left(1+|\xi|^2\right)^j\left|\widehat \sigma(\rho_\bm \xi)\right|^2\, \rd \xi\\
&=& \rho_\bm^{n} \int_{\R^n} \left(1+|\xi/\rho_\bm|^2\right)^j\left|\widehat \sigma(\xi)\right|^2\, \rd \xi\\
 &\leq &\rho_{\bm}^{n-2j}\|\sigma\|^2_{H^j(\R^n)},
\end{eqnarray*}
since $\rho_\bm \leq 1$,
so, using that $\rho_{\bm} \in (\rho_{\min}h/h_0,h/h_0]$ by \eqref{eq:hbm}, and again that the supports of the $\sigma_{\bm}$ are disjoint,
$$
\|v_h\|^2_{H^{j}(\R^n)} = \sum_{\bm\in L_h}|a_{\bm}|^2\|\sigma_{\bm}\|^2_{H^j(\R^n)} \leq \left\{\begin{array}{ll}
                                                                                          \|\sigma\|^2_{H^j(\R^n)} (\rho_{\min}h/h_0)^{n-2j} \,\|a\|_2^2, & \mbox{if } j\geq n/2, \\
                                                                                          \|\sigma\|^2_{H^j(\R^n)} (h/h_0)^{n-2j} \,\|a\|_2^2, & \mbox{if } j<n/2.
                                                                                        \end{array}
\right.
$$
In particular this applies for $j=0,L$, where $L := \lceil\max(J+(n-d)/2,n/2)\rceil$, so that $L\geq \max(s,n/2)$. Hence, and applying H\"older's inequality to the definition of the $H^s(\R^n)$ norm and noting that $\|\sigma\|_{H^0(\R^n)}\leq \|\sigma\|_{H^L(\R^n)}$,
\begin{equation} \label{eq:lb2}
\|v_h\|_{H^s(\R^n)} \leq \|v_h\|^{1-s/L}_{H^0(\R^n)}\|v_h\|^{s/L}_{H^L(\R^n)} \leq \|\sigma\|_{H^L(\R^n)}\, \rho_{\min}^{(n/(2L)-1)s}\,(h/h_0)^{n/2-s}\,\|a\|_2.
\end{equation}
Now, recalling from \S\ref{sec:FunctionSpaces} that $\tr^*:\IH^{-t}(\Gamma)\to H^{-s}(\R^n)$ has unit norm, %
$$
\|\psi_h\|_{\IH^{-t}(\Gamma)}\geq
\|u_h\|_{H^{-s}(\R^n)} = \sup_{v\in H^s(\R^n)\setminus\{0\}}\frac{|\langle v,u_h\rangle_{H^{s}(\R^n)\times H^{-s}(\R^n)}|}{\|v\|_{H^s(\R^n)}}\geq  \frac{|\langle v_h,u_h\rangle_{H^{s}(\R^n)\times H^{-s}(\R^n)}|}{\|v_h\|_{H^s(\R^n)}}.
$$
Combining this inequality with \eqref{eq:lb1} and \eqref{eq:lb2} and recalling \eqref{eq:st}  we obtain that
$$
\|\psi_h\|_{\IH^{-t}(\Gamma)} \geq \frac{c_0^{1/2}}{\|\sigma\|_{H^L(\R^n)}} \,\rho^{t+n/2}_{\min}\,(h/h_0)^{t}\,\|a\|_2,
$$
and the result follows on recalling that $\|a\|_2= \|\psi_h\|_{\IL_2(\Gamma)}$.
\end{proof}

\section{Table of definitions}
\label{sec:TableOfDefns}
\begin{tabular}{|l|l|}
\hline
\textbf{Symbol/terminology} & \textbf{Defined/introduced} \\\hline
$\Gamma$, $\Gamma_\infty$  & \S\ref{sec:Introduction}   \\\hline
$\cH^\alpha$, $\dimH$, $|E|$ (Lebesgue measure), $d$-set & \S\ref{sec:HausdorffMeasure}    \\\hline
IFS, OSC, homogeneous, self-similar, disjoint   & \S\ref{sec:IFS}    \\\hline
$H^s(\Gamma_\infty)$, $H^s_E$, $H^s(\Omega)$, $\tH^s(\Omega)$   & \S\ref{sec:FunctionSpaces} \& \S\ref{sec:FunctionSpacesScreens}   \\\hline
$\IL_2(\Gamma)$, $\IL_\infty(\Gamma)$, $\IH^t(\Gamma)$, $\tr$, $\tr^*$  & \S\ref{sec:FunctionSpaces} \& \S\ref{sec:FunctionSpacesScreens}   \\\hline
$W^1(\Omega)$, $W^{1,{\rm loc}}(\Omega)$, $\gamma^\pm$, $\partial_{\bn}^\pm$, $C^\infty_{0,\Gamma}$  &  \S\ref{sec:FunctionSpacesScreens}   \\\hline
$I_\ell$, $I_{\N}$, $I_{\N_0}$, $\Gamma_\bm$, $\IW_\ell$, $\chi_\bm$, $\psi_0$, $\psi^m_\bm$, $J_\nu$, $\nu_0$, $\|\cdot\|_{t}$  &  \S\ref{sec:IFS}   \\\hline
$u^i$, $\Gamma^c$, $P$, $\cS$, $S$, $\phi$, $a(\cdot,\cdot)$, $a^\Omega(\cdot,\cdot)$, $S^\Omega$, $u^\infty$, $\Phi^\infty$, $\mathds{S}^n$  &  \S\ref{sec:BVPsBIEs}   \\\hline
$t_d$, $\IS$ &  \S\ref{sec:BIEondsets}   \\\hline
$\{T_j\}_{j=1}^N$, $\mathbb{V}_N$, $V_N$, $\phi_N$, $A$, $\vec{c}$, $\vec{b}$, $J$  & \S\ref{sec:HausdorffBEM}    \\\hline
$\IX_\nu$, $K_\nu$, $\IP_\nu$, $X_\nu$, $\IY_h$, $Y_h$, $L_h$, $h$  &  \S\ref{sec:BestApprox}   \\\hline
$c_t$  & \S\ref{sec:InverseEstimates}    \\\hline
$A^Q$, $\vec{b}^Q$, $\vec{c}^Q$, $\phi_N^Q$, $J^Q$, $\Hull$  & \S\ref{sec:Quadrature}    \\\hline
\end{tabular}

\vspace{3ex}

\small
{\bf Acknowledgements.} 
We are grateful to Martin Costabel for instructive discussions related to this project. SC-W was supported by Engineering and Physical Sciences Research Council (EPSRC) Grant EP/V007866/1, and DH and AG by EPSRC grants EP/S01375X/1 and EP/V053868/1. AM was supported by the PRIN project ``NA-FROM-PDEs'' and by the MIUR through the ``Dipartimenti di Excellenza'' Program (2018--2022)--Dept.\ of Mathematics, University of Pavia. 
AC was supported by CIDMA (Center for Research and Development in Mathematics and Applications) and FCT (Foundation for Science and Technology) within project UIDB/04106/2020. 
The authors acknowledge the use of the UCL Myriad High Performance Computing Facility (Myriad@UCL), and associated support services, in the completion of this work.
\bibliography{BEMbib_short2014}%
\bibliographystyle{siam}
\end{document}